\documentclass[aap]{imsart}
\RequirePackage[utf8]{inputenc}
\usepackage{mathrsfs,soul}
\usepackage{amssymb,amsmath,amsthm,amsfonts,amsbsy,latexsym,dsfont}
\usepackage[textsize=small]{todonotes}       
\usepackage[numeric,initials,nobysame]{amsrefs}
\usepackage{upref}
\usepackage{xcolor,colortbl}
\usepackage{enumerate}
\RequirePackage[colorlinks,citecolor=blue,urlcolor=blue]{hyperref}
\RequirePackage{graphicx}
\usepackage{subcaption}
\usepackage{pdfsync}   
\usepackage{paralist}
\usepackage{accents}
\usepackage[mathscr]{euscript}

\startlocaldefs
\newenvironment{enumeratei}{\begin{enumerate}[\upshape (i)]}{\end{enumerate}}
\newenvironment{enumeratea}{\begin{enumerate}[\upshape (a)]}{\end{enumerate}}

\allowdisplaybreaks

\newenvironment{inparaenuma}{\begin{inparaenum}[\upshape \bfseries (a) ]}{\end{inparaenum}}

\numberwithin{equation}{section}
\numberwithin{figure}{section}
\numberwithin{table}{section}

\newtheorem{thm}{Theorem}[section]

\newtheorem{theorem}[thm]{Theorem}

\newtheorem{corollary}[thm]{Corollary}
\newtheorem{prop}[thm]{Proposition}
\newtheorem{defn}[thm]{Definition}

\newtheorem{ass}[thm]{Assumption}

\newtheorem{conj}[thm]{Conjecture}
\newtheorem{lemma}[thm]{Lemma}

\theoremstyle{definition}
\newtheorem{rem}{Remark}

\newcommand{\ind}{\mathds{1}}
\newcommand{\eps}{\varepsilon}

\newcommand{\set}[1]{\left\{#1\right\}}

\newcommand{\probc}{\stackrel{\mathrm{P}}{\longrightarrow}}

\newcommand{\convas}{\stackrel{\mathrm{a.s.}}{\longrightarrow}}

\def\qed{ \hfill $\blacksquare$}

\newcommand{\cA}{\mathcal{A}}\newcommand{\cB}{\mathcal{B}}
\newcommand{\cE}{\mathcal{E}}
\newcommand{\cG}{\mathcal{G}}\newcommand{\cH}{\mathcal{H}}\newcommand{\cI}{\mathcal{I}}
\newcommand{\cL}{\mathcal{L}}
\newcommand{\cM}{\mathcal{M}}\newcommand{\cN}{\mathcal{N}}
\newcommand{\cP}{\mathcal{P}}\newcommand{\cR}{\mathcal{R}}
\newcommand{\cS}{\mathcal{S}}\newcommand{\cT}{\mathcal{T}}
\newcommand{\cV}{\mathcal{V}}
\newcommand{\cY}{\mathcal{Y}}\newcommand{\cZ}{\mathcal{Z}}

\newcommand{\vA}{\mathbf{A}}

\newcommand{\vM}{\mathbf{M}}\newcommand{\vN}{\mathbf{N}}
\newcommand{\vP}{\mathbf{P}}\newcommand{\vQ}{\mathbf{Q}}
\newcommand{\vS}{\mathbf{S}}
\newcommand{\vW}{\mathbf{W}}
\newcommand{\vY}{\mathbf{Y}}

\newcommand{\vh}{\mathbf{h}}
\newcommand{\vj}{\mathbf{j}}

\newcommand{\vp}{\mathbf{p}}
\newcommand{\vs}{\mathbf{s}}\newcommand{\vt}{\mathbf{t}}
\newcommand{\vw}{\mathbf{w}}
\newcommand{\vy}{\mathbf{y}}

\newcommand{\mvh}{\boldsymbol{h}}
\newcommand{\mvm}{\boldsymbol{m}}

\newcommand{\mveta}{\boldsymbol{\eta}}\newcommand{\mvkappa}{\boldsymbol{\kappa}}

\newcommand{\mvnu}{\boldsymbol{\nu}}
\newcommand{\mvpi}{\boldsymbol{\pi}}
\newcommand{\mvrho}{\boldsymbol{\rho}}

\newcommand{\mvxi}{\boldsymbol{\xi}}
\newcommand{\mvvarpi}{\boldsymbol{\varpi}}
\newcommand{\mvvarphi}{\boldsymbol{\varphi}}\newcommand{\mvPsi}{\boldsymbol{\Psi}}

\newcommand{\fD}{\mathfrak{D}}\newcommand{\fF}{\mathfrak{F}}
\newcommand{\fI}{\mathfrak{I}}
\newcommand{\fK}{\mathfrak{K}}
\newcommand{\fM}{\mathfrak{M}}\newcommand{\fN}{\mathfrak{N}}
\newcommand{\fP}{\mathfrak{P}}\newcommand{\fR}{\mathfrak{R}}
\newcommand{\fT}{\mathfrak{T}}\newcommand{\fU}{\mathfrak{U}}

\newcommand{\fp}{\mathfrak{p}}

\newcommand{\bL}{\mathbb{L}}
\newcommand{\bN}{\mathbb{N}}
\newcommand{\bR}{\mathbb{R}}
\newcommand{\bT}{\mathbb{T}}
\newcommand{\bX}{\mathbb{X}}

\newcommand{\sF}{\mathfrak{F}}

\DeclareMathOperator{\E}{\mathbb{E}}
\DeclareMathOperator{\pr}{\mathbb{P}}

 \DeclareMathOperator{\BP}{BP}

\DeclareMathOperator{\PF}{PF}

\newcommand{\sss}{\scriptscriptstyle}

\newcommand{\convd}{\stackrel{d}{\longrightarrow}}

\definecolor{aqua}{rgb}{0.0, 1.0, 1.0}
\definecolor{boo}{rgb}{1.0, 0.0, 1.0}

\newcommand{\Geom}{{\sf Geom} }
\newcommand{\Bern}{{\sf Bernoulli} }
\newcommand{\Exp}{{\sf Exp} }

\DeclareMathAlphabet{\mathscrbf}{OMS}{mdugm}{b}{n}

\newcommand{\scrP}{\mathscrbf{P}}
\newcommand{\scrU}{\mathscrbf{U}}

\newcommand{\scI}{\mathscr{I}}
\newcommand{\scZ}{\mathscr{Z}}
\newcommand{\sch}{{\mathcalb{h}}}

\usepackage{tikz}
\usepackage[most]{tcolorbox}

\usetikzlibrary{calc,intersections,through,backgrounds,shapes.geometric}
\usetikzlibrary{graphs}
\tikzset{every path/.style={line width=.07 cm}}

\newcommand{\tikzcircle}[2][blue,fill=blue]{\tikz[baseline=-0.5ex]\draw[#1,radius=#2] (0,0) circle ;}%
\newcommand{\cic}{\tikzcircle{1.5pt}}

\usepackage{tikz-cd}
\newcommand{\bbT}{\mathbb{T}}
\newcommand{\TT}{\mathcal{T}}
\newcommand{\bs}{\mathbf{s}}
\newcommand{\bt}{\mathbf{t}}
\newcommand{\probfr}{\stackrel{\mbox{$\operatorname{a.s.}$-\bf fr}}{\longrightarrow}}
\newcommand{\probcrf}{\stackrel{\mbox{$\operatorname{a.s.}$-\bf efr}}{\longrightarrow}}
\newcommand{\prob}{\mathbb{P}}
\newcommand{\bfomega}{{\boldsymbol \omega}}
\newcommand{\Zbold}{{\mathbb{Z}}}
\DeclareMathOperator{\PR}{PR}
\usepackage[font=scriptsize]{subcaption}
\usepackage{filemod}
\usepackage{adjustbox}

\newcommand{\fpm}{\mvvarpi}

\DeclareFontFamily{U}{BOONDOX-calo}{\skewchar\font=45 }
\DeclareFontShape{U}{BOONDOX-calo}{m}{n}{
  <-> s*[1.05] BOONDOX-r-calo}{}
\DeclareFontShape{U}{BOONDOX-calo}{b}{n}{
  <-> s*[1.05] BOONDOX-b-calo}{}
\DeclareMathAlphabet{\mathcalb}{U}{BOONDOX-calo}{m}{n}
\SetMathAlphabet{\mathcalb}{bold}{U}{BOONDOX-calo}{b}{n}
\DeclareMathAlphabet{\mathbcalb}{U}{BOONDOX-calo}{b}{n}
\newcommand{\uNS}{{\boldsymbol\fU}}
\newcommand{\dNS}{{\boldsymbol\fD}}
\newcommand{\indNS}{{\boldsymbol {\fI \fD}}}
\newcommand{\prNS}{{\boldsymbol {\fP\fR}}_c}
\newcommand{\prfNS}{{\boldsymbol {\fP\fR}}_M}

\newcommand{\cg}{\color{black}}

\endlocaldefs

\usepackage{wrapfig}

\begin{document}

\begin{frontmatter}
\title{Attribute network models, stochastic approximation, and network sampling}
\runtitle{Attribute networks, stochastic approximation and sampling}

\begin{aug}
\author[A]{\fnms{Nelson} \snm{Antunes}\ead[label=e1]{nantunes@ualg.pt}},
\author[B]{\fnms{Sayan} \snm{Banerjee}\ead[label=e2]{banerjee@email.unc.edu}},
\author[B]{\fnms{Shankar} \snm{Bhamidi}\ead[label=e3]{bhamidi@email.unc.edu}}
\and
\author[B]{\fnms{Vladas} \snm{Pipiras}\ead[label=e4]{pipiras@email.unc.edu}}
\address[A]{Center for Computational and Stochastic Mathematics,University of Lisbon, Avenida Rovisco Pais, Lisbon, Portugal, 1049-001\printead[presep={,\ }]{e1}}

\address[B]{Department of Statistics and Operations Research, 304 Hanes Hall, University of North Carolina, Chapel Hill, NC 27599\printead[presep={,\ }]{e2,e3,e4}}
\end{aug}

\begin{abstract}
Motivated by the central role of social networks in the diffusion of information, the study of network valued data where nodes and/or edges have attributes, which modulate the dynamics of both network evolution, and information flow on the network itself, has witnessed significant research interest across multiple disciplines. A key ingredient of this general area comprises probabilistic network models that incorporate (a) heterogeneity in edge creation across different attribute groups; (b) temporal network evolution and (c) popularity bias. Such models are then used to understand a host of domain specific questions, including bias in network sampling, PageRank and degree centrality scores and their impact in network ranking and recommendation algorithms. Despite significant interest, for these network models,  the main network functional amenable to analysis has so far been degree distribution asymptotics. 

In this paper, we analyze dynamic random network models where younger vertices connect to older ones with probabilities proportional to their degrees as well as a propensity kernel governed by their attribute types. Using stochastic approximation techniques we show that, in the large network limit, such networks converge in the local weak sense to {\cg limiting infinite random trees with an explicit description in terms of randomly stopped multi-type branching processes. This allows for the derivation of asymptotics for a wide class of network functionals implying, for example, that while degree distribution tail exponents depend on the attribute type (already derived by \cite{jordan2013geometric}), PageRank centrality scores have the \emph{same} tail exponent across attributes.} The limit results also give explicit formulae for the performance of various network sampling mechanisms. One surprising consequence is the efficacy of PageRank and walk based network sampling schemes for directed networks in the setting of rare minorities. 
\end{abstract}



\end{frontmatter}

\section{Introduction}
\label{sec:int}
Attributed networks, namely graphs in which nodes and/or edges have attributes, are at the center of network-valued datasets in many modern applications.  Areas such as network representation learning \cite{Fan:2021} aim to obtain low-dimensional embeddings for nodes through local explorations such as random walks, taking into account network topology and network attributes, for subsequent use in machine learning pipelines such as clustering \cite{Chang:2019} and  classification \cite{Lee:2017}. 

In a different direction, in settings where attributes represent social characteristics,  there is now significant interest in understanding fairness questions related to the positions of individuals in the network, where connections are regulated by multiple factors including: \begin{inparaenuma}
	\item inherent heterogeneity in connection between and across groups;
	\item time and path dependent nature of connections;
	\item popularity bias, i.e. the inherent tendency {\cg to} be attracted to popular individuals. 
\end{inparaenuma}  
The corresponding emergent networks modulate the diffusion of information within a network, effecting the availability and timeliness of information to minority groups \cite{jadidi2018gender}. In addition to the direct influence of networks on propagation of information, these networks are fed into algorithmic pipelines such as network sampling algorithms, leading to a potentially distorted view of the network \cite{Wagner:2017,espin2018towards,antunes2023learning}. The networks are further used for ranking individuals according to their centrality scores (measured via functionals such as degree or PageRank scores) which further exacerbate inequalities in the network \cite{espin2022inequality}, or effect the perceptions of minorities within the network \cite{lee2019homophily}. In this rapidly burgeoning field,  network models play a major role in obtaining insight into both fundamental drivers of network evolution, as well as phenomenon of co-evolution of the dynamics of the network with algorithms such as recommendation systems that use the underlying network structure to modify the network, thus impacting future evolution;  see \cite{karimi2022minorities}, {\cg a nice recent survey for relevant literature from applications}. The goal of this paper is a rigorous evaluation of the properties of these models and their alignment with the needs of domain scientists.

Let us start with the general class of models, {\cg belonging to the preferential attachment family of networks \cite{barabasi1999emergence},} underpinning a host of recent studies, and subsequently describe specific examples in this class.   Fix a latent space (referred to as attribute space) $\cS$ with associated $\sigma$-field $\cB(\cS)$ so that $(\cS, \cB(\cS))$ is a measurable space; in many but not all cases, this is a metric space. While this paper will only deal with the finite type space, the general setting below allows us to describe work in progress extending the findings in this paper to more general models.  Fix a probability measure $\mvpi$ on this measurable space and a measurable (potentially asymmetric) function $\kappa: \cS\times \cS\to \bR_+$; intuitively this function measures propensities of pairs of nodes to connect, based on their attributes. Fix a preferential attachment parameter $\gamma \in [0,1]$ and an out-degree function $\mvm: \cS \rightarrow \mathbb{N}$ which modulates the number of edges that a new vertex possesses (in particular the dependence of this on the attribute type) when it enters the system. At time $n=0$, initialize from a base connected graph $\tilde{\cG}_0$ where every vertex has an attribute in $\cS$. {\cg The next definition describes the dynamics of a sequence of growing random networks $\set{\tilde \cG_s:s\geq 0}$ from initial state $\tilde \cG_0$}.

\begin{defn}[Attributed evolving network model class $\scrP$]
	\label{def:attrib-evol}
	  Vertices enter the system sequentially at discrete times $n\geq 1$  starting with a base connected graph $\tilde \cG_0$ at time $n=0$. Having constructed $\set{\tilde \cG_s: 0\leq s\leq n}$, write $v_{n+1}$ for the vertex that enters at time $n+1$ and $a(v_{n+1})$ for the corresponding attribute;  every such vertex $v_{n+1}$ has attribute distribution 
		\begin{equation}
			\label{eqn:726}
			a(v_{n+1}) \sim \mvpi, \qquad \text{ independent of } \set{\tilde \cG_s: 0\leq s \leq n}. 
		\end{equation}
	 The dynamics of construction are recursively defined as: for any $n$ and $v\in \tilde \cG_n$, let $\deg(v,n)$ denote the degree of $v$ at time $n$ \textcolor{black}{(if $\tilde \cG_0 = \{v_0\}$, initialize $\deg(v_0,0)=1$)}. For $n \ge 0$, $v_{n+1}$ attaches to the network via $\mvm(a(v_{n+1}))$ outgoing edges. Each edge independently chooses an existing vertex in $\tilde \cG_n$ to attach to, with probabilities (conditionally on $\tilde \cG_n$ and $a(v_{n+1})$) given by: 
\begin{equation}
\label{eqn:912old}
	\pr(v_{n+1} \leadsto v \,| \,\tilde \cG_n, a(v_{n+1}) = a^{\star}) = \frac{\kappa(a(v), a^\star) [\deg(v,n)]^\gamma}{\sum_{v^\prime \in \tilde \cG_n } \kappa(a(v^\prime), a^\star) [\deg(v^\prime, n)]^\gamma}, \quad v \in \tilde \cG_n.
\end{equation}
Denote this model of evolving random networks by $\scrP(\gamma, \mvpi, \kappa, \mvm)$. 
\end{defn}
{\cg Note that, during the attachment of the vertex $v_{n+1}$ in the above construction, $\deg(v,n)$ remains constant for each $v \in \tilde \cG_n$ during the addition of the $\mvm(a(v_{n+1}))$ outgoing edges.}
When $\mvm \equiv 1$, we will abbreviate as $\scrP(\gamma, \mvpi, \kappa
)$.
The latent (attribute) space $\cS$ will always be clear from context and so the dependence of functionals on $\cS$ is suppressed to ease notation.  Some special cases are described below.


\begin{enumeratea}
	\item $\kappa\equiv 1$: the regime $\gamma=1$ corresponds to the Barabási–Albert model \cite{barabasi1999emergence} while $0<\gamma <1$ corresponds to sublinear preferential attachment \cite{krapivsky2001organization}. Here there is no dependence on type in the evolution of connectivity in the model. 
	\item $\cS = \bR_+$, $\kappa(a,a^\prime) =a$ and $\gamma=1$: here new attachment is driven only by the attribute of existing vertices and is often called preferential attachment with multiplicative fitness \cite{bianconi2001bose}. 
	\item $\cS = \set{a,b}$ so that $\kappa$ is a $2\times 2$ matrix and $\gamma =1$: this model was formulated and studied in \cite{Karimi:2018} to understand (degree centrality based) ranking algorithms when one type is a minority. 
	\item $\cS = [0,1]$, $\gamma=1$ and $\kappa(a,a^\prime) = 1- |a-a^\prime|$ is called the scale free homophilic model \cite{Almeida:2013}. 
	\item $\cS = \cS_2$ namely the surface of the three dimensional unit sphere in $\bR^3$ with associated distance $d(\cdot,\cdot)$, $\kappa(a,a^\prime) = F(d(a,a^\prime))$ namely an appropriate function of the corresponding distance on $\cS_2$ and $\gamma =1$ was termed geometric preferential attachment \cite{flaxman2007geometric}; a more general version of this model, still with $\gamma =1$, was studied in \cite{jordan2013geometric} and was particularly influential for this current paper. 
	\item The setting with two attributes and symmetric kernel $\kappa(a,b) = \rho + (1-\rho)\ind\set{a = b}$, for parameter $\rho \in [0,1]$, was studied in \cite{avin2015homophily} under the name of `biased preferential attachment'. They showed that, in this setting, this model could be interpreted as an acceptance-rejection scheme coupled with the usual (one-attribute) preferential attachment dynamics. More precisely, at each stage, when a new vertex enters the system, it chooses to attach to a vertex with probability proportional to the degree (as in standard preferential attachment), and if the vertex chosen is of the same type as the new vertex, an edge is formed, else with probability $\rho$, the connection is accepted, and with $1-\rho$ the connection is rejected, and this new vertex tries again, till an edge is formed. The main goal there was to study homophily and `glass ceiling' effects that rigorously establish the inequality emerging between the majority and minority attribute individuals under the network dynamics. Among other things, the difference in the exponents of the expected degree distributions for the two attributes was (independently of \cite{jordan2013geometric}) established there (see \cite[Theorem 4.12]{avin2015homophily}). In \cite{stoica2024fairness}, questions on fairness were investigated for the PageRank and `HITS' centrality measures for the biased preferential attachment model.
\end{enumeratea}


 \subsection{Motivation for this work and analysis}
 \label{sec:motivation-overview}
 
The above class of models (mostly in the $\gamma \equiv 1$ case) has been used to gain insight in multiple recent applications (see Section \ref{sec:disc} for starting points to this now large body of work). In particular,  researchers have been interested in understanding properties of complicated functionals of the model in terms of its parameters. Such functionals include (definitions given in Section \ref{sec:mod-su}): 
 \begin{enumeratei}
 	\item Global centrality scores such as the PageRank and contrasting scaling exponents between majority and minority nodes \cite{espin2022inequality};
	\item Scaling of maximal degree vertices, since these ``hubs'' play a key role in driving diffusion mechanisms \cite{Karimi:2018};
	\item Homophily measures \cite{shrum1988friendship,mcpherson2001birds,mislove2010you}, namely propensity of vertices to connect to vertices of the same/different ``type'', respectively, and the impact of such tendencies on other functionals, including diffusion of information,  modularity and the emergence of echo-chambers in the network \cite{Karimi:2018,Wagner:2017,espin2018towards}. 
 \end{enumeratei}
 {\cg Due to the complexity of the model, in most of these existing studies, the main functional that has been analyzed,  either through writing down fluid limits \cite{Karimi:2018}, or stochastic approximation techniques \cite{jordan2013geometric}, has  been the degree distribution.  The goal of this paper is to understand the network geometry beyond one-step neighborhoods. Towards this end, we first propose a closely related model that turns out to be more analytically tractable. Since the general setting (with ``continuous'' attribute space) and general $\mvm$ is more technical, let us explain the basic rationale in the simpler discrete setting with $\mvm \equiv 1$ and finite attribute space. Here, if the base graph $\cG_0$ is a tree, then we get a sequence of growing trees. Let $\cS = [K]:=\set{1,2,\ldots, K}$ so that the attribute probability measure $\mvpi$ is just a probability mass function (pmf) on $[K]$. Fix a (potentially, and in most cases different from $\mvpi$) weight measure $\mvnu$ (need not be normalized) on $\cS$ and consider the attributed tree network model $\set{\cG_n: n\geq 0}$ with dynamics:
 \begin{equation}
 	\pr\left(a(v_{n+1}) = a^\star,  v_{n+1} \leadsto v \,| \, \cG_n\right) := \frac{\mvnu(a^\star)\kappa(a(v), a^\star) [\deg(v,n)]^\gamma}{\sum_{a\in \cS}\sum_{v^\prime \in \cG_n } \mvnu(a) \kappa(a(v^\prime), a) [\deg(v^\prime, n)]^\gamma},
	\label{eqn:model-scu}
 \end{equation}
for $a^\star \in \cS, \,v \in \cG_n$. 
We will refer to the above class of models in \eqref{eqn:model-scu} as $\scrU(\gamma, \mvnu, \kappa)$.  Thus, here $\mvnu$ plays the role of a weight and further, unlike the model $\scrP$ where each new arriving vertex has attribute sampled independently from the current state of the network, here the attribute distribution of new vertices is closely dependent on the entire state of the current network:
\begin{equation}
\label{eqn:attr-scu}
\pr(a(v_{n+1}) = a^\star| \cG_n) = \frac{\sum_{v\in \cG_n} \mvnu(a^\star)\kappa(a(v), a^\star) [\deg(v,n)]^\gamma}{\sum_{a\in \cS}\sum_{v^\prime \in \cG_n } \mvnu(a) \kappa(a(v^\prime), a) [\deg(v^\prime, n)]^\gamma}, \qquad a^\star \in [K]
\end{equation}
The purpose of introducing this new model can be explained through the following.
 \begin{enumeratea}
	\item {\bf Tractability of $\scrU$ as contrasted to $\scrP$:} The issue with the dynamics in the original model $\scrP$ is that at each stage, the evolution depends on two factors, the type of the new vertex that is generated autonomously and the current state of the network. This is one of the major reasons why till date, only degree distribution asymptotics for $\scrP$ have been amenable and more complex functionals such as PageRank score asysmptotics have proven intractable.   It turns out there is a specific construction of $\scrU$ (described in the next section) where the evolution is driven {\bf only by existing vertices in the system}; further this construction is via embedding into a (continuous time) branching process which confers significant amount of conditional independence between progenies of existing vertices in the system. Thus analyzing asymptotics of local neighborhoods of $\scrU$ is much more tractable.  This allows for a full description of asymptotics for a wide range of functionals, including degree distribution and PageRank scores, both for the whole network, and for each attribute, allowing one to contrast, for example, the dependence of tail exponents of various network functionals, in the large network limit,  on the attribute type. In particular, write $\hat{\mvpi}_n = n^{-1}\sum_{i=1}^n \delta\set{a(v_i)}$ for the empirical distribution of attribute types; here $\delta$ is the Dirac delta function and we are ignoring attributes of the initial graph $\cG_0$ for simplicity. Then under regularity conditions, there exists a limit deterministic measure $\mvpi_\infty$ such that
	\begin{equation}
	\label{eqn:mvpi-convg}
		\hat{\mvpi}_n \probc {\mvpi}_\infty, 
	\end{equation}
	where the limit ${\mvpi}_\infty ={\mvpi}_\infty(\gamma, \mvnu, \kappa) $ is obviously a functional of the parameters driving the model. 
	\item {\bf Evolution dynamics of $\scrU$:} Next, the specific construction of $\scrU$ in the next section further satisfies for each $n$:
	\begin{equation}
	\label{eqn:913}
		\pr(v_{n+1} \leadsto v \, | \, \cG_n, a(v_{n+1}) = a^{\star}) = \frac{\kappa(a(v), a^\star) [\deg(v,n)]^\gamma}{\sum_{v^\prime \in \cG_n } \kappa(a(v^\prime), a^\star) [\deg(v^\prime, n)]^\gamma}. 
	\end{equation} 
	
	\item {\bf Resolvability and transferring asymptotics from $\scrU$ to $\scrP$: } Now compare \eqref{eqn:mvpi-convg} and \eqref{eqn:913} describing properties of $\scrU$ to \eqref{eqn:726} and \eqref{eqn:912old} describing dynamics of $\scrP$ suggesting that, if one can choose $\mvnu$ for $\scrU$ appropriately so that the limit attribute  distribution for $\scrU$ namely ${\mvpi}_\infty$ is the same $\mvpi$ for model class $\scrP$ then asymptotically the evolutionary dynamics of $\scrU$ are close to the dynamics of $\scrP$.  
In the $\gamma =1$ regime,	we will show that, in the finite attribute setting, given a model $\scrP(1, \mvpi, \kappa)$, there is a {\bf specific} (and easily computable) choice of $\mvnu = \mvnu(\mvpi, \kappa)$ so that for the corresponding model class $\scrU(1, \mvnu, \kappa)$, ${\mvpi}_\infty(1, \mvnu, \kappa) = \mvpi$. 
What is perhaps surprising is that the {\bf local weak limit} of $\scrP(1, \mvpi, \kappa)$ is {\bf exactly} the same as $\scrU(1, \mvnu(\mvpi,\kappa), \kappa)$. Thus  asymptotics for the technically challenging $\scrP$ model are closely captured by a mathematically more tractable network model $\scrU$ leading to a general path for deriving rigorous results for the model class $\scrP$ sketched in Figure \ref{fig:label}. 
	\begin{figure}[htbp]
		\centering
		
		\adjustbox{scale=1.4,center}{
			\begin{tikzcd}
			\scrP(\gamma, \mvpi, \kappa) \arrow[r, blue, thick] \arrow[d, swap,  "??", red, dashed, thick] & \mvnu(\gamma, \mvpi, \kappa) \arrow[d, blue, thick] \\   
			\text{Asymptotics}  & \arrow[l, blue, thick]\scrU(\gamma, \mvnu, \kappa)
			\end{tikzcd}
			}
		\caption{Path to derive rigorous asymptotics for resolvable models.}
		\label{fig:label}
	\end{figure}

\end{enumeratea}
}
Paraphrasing the meta-principle in words (eschewing technical assumptions):

\begin{quote}
	\emph{If the choice of the weighting measure $\mvnu$ for the more easily analyzable model $\scrU$ is such that the asymptotic proportions of types matches that of $\scrP$, then asymptotics of a host of functionals of both models match.  }
\end{quote}

We call this phenomenon {\bf resolvability} of the model $\scrP$ (see Section \ref{sec:main-res}). \textcolor{black}{As will be shown below, even for non-tree network models in class $\scrP$ ($\mvm \not\equiv 1$), resolvability holds and the local limit is given by an associated version of $\scrU$ whose dynamics now involve $\mvm$.} The corresponding asymptotics further give insight into the bias of various proposed network sampling algorithms and their ability to glean insight into connectivity properties of rare minorities. To keep this paper to manageable length, we specialize to the $\gamma \in \set{0,1}$, finite state space setting, leaving the general type space as well as sublinear case to future analysis.

\subsection{Organization of the paper}  
 {\cg In Section \ref{sec:mod-su} we describe a specific construction of the model class $\scrU$ which elucidates why asymptotics for this model is more tractable. We further review the concepts of local weak convergence and give precise definitions of relevant network functionals including homophily measures and PageRank scores. Section \ref{sec:main-res} has the the statements of the main results for the tree case (all out-degrees one) in the context of preferential attachment driven ($\gamma =1$) attribute models. Section \ref{sec:nt} extends these results to the non-tree setting. Section \ref{sec:exten-res} describes results in the uniform attachment  setting ($\gamma =0$). Section \ref{fairsec} employs the theoretical foundations developed in the previous sections to understand attribute estimation schemes when only partial information on network structure is available.  We discuss related work and indicate future work, in Section \ref{sec:disc}. The ensuing sections contain proofs of the results in the paper. }

 \section{Constructions and basic definitions}
 \label{sec:mod-su}
\subsection{The $\scrU$ Model}
\label{sec:model-def}
We now describe the precise construction of the model class $\scrU$. The setting is as in Definition \ref{def:attrib-evol} but we will use $\mvnu$ to denote the weighting measure on $(\cS, \cB(\cS))$ as it plays a very different role from $\mvpi$ in model class $\scrP$.  Notationally we mainly follow \cite{jagers1996asymptotic}. Given the driving parameters $\mvnu, \kappa, \gamma$ and a fixed type $a\in \cS $, first construct a point process $\mvxi$ on the space $\bR_+\times \cS$ \textcolor{black}{which is a Markovian pure birth process with intensity function} 
\begin{equation}
\label{eqn:rate}
{\cg	\lambda_a(dt, db):=  \kappa(a,b) (|\mvxi([0,t]\times \cS)| +1)^\gamma dt \cdot  \mvnu(db). }
\end{equation}
{\cg Since the majority of this paper only deals with finite attribute space $\cS$, let us describe the explicit construction  in this setting. Here $\kappa$ is just a (potentially assymmetric) matrix, and $\mvnu$ is a weight measure.  For fixed $a\in \cS$, one constructs $\cS$ pure birth Markov processes $\set{\xi_{a,b}[0,t]: t\geq 0)}$ whose dynamics are coupled via their respective intensity measures:
\begin{equation}
\label{eqn:rate-finite}
\lambda_a(dt, b) = \kappa(a,b) (|\mvxi([0,t]\times \cS)| +1)^\gamma dt \cdot  \mvnu(b), \qquad b\in \cS, t>0,
\end{equation}
where $\mvxi[0,t] = \sum_{b\in \cS} \xi_{a,b}[0,t] $. }
\begin{defn}[Network model class $\scrU$]
	\label{defn:ctbp}
	Fix a type $a\in \cS$. Consider the continuous time, multi-type, branching process $\set{\BP_{a; (\gamma, \mvnu, \kappa)}(t):t\geq 0}$ starting with a single individual of type $a$ and where each individual of type $a^\prime$ entering the population has an independent offspring point process with intensity rate $\lambda_{a^\prime}(\cdot)$ as in \eqref{eqn:rate} encoding reproduction times. For $n\geq 0$, define the stopping time (with respect to the natural filtration),
	\[T_n:= \inf\set{t\geq 0: |\BP_{a; (\gamma, \mvnu, \kappa)}(t)| = n + 1}.\]
\end{defn}
Since the context of applications of this model is for the evolution of networks, we will interchangeably use individual or node in the rest of the paper. The following Lemma {\cg is a manifestation of the Athreya-Karlin embedding \cite{athreya1968embedding}. In the Lemma, $\sim$ is used to denote ``has the same distribution as'' relationship between random objects. } 

\begin{lemma}\label{lem:embedding-ok}
	Assume $\cS$ is compact and $\kappa(\cdot, \cdot)$ is bounded. Then the above branching does not explode i.e. $T_n\convas \infty$ as $n\to\infty$. Let $\set{\cG_n:n\geq 0} \equiv \set{\BP(T_n):n\geq 0}$ denote the sequence of discrete  growing attributed network models. Then $\set{\cG_n:n\geq 0} \sim \scrU(\gamma, \mvnu, \kappa)$ starting with one vertex of type $a$. 
\end{lemma}
This is the construction we will use. {\cg Conceptually, the above construction also explains why model class $\scrU$ is easier to analyze since (in continuous time) the dynamics is completely driven by the reproduction of individuals currently in the population, which happens (conditional on their types) independently across individuals; this conditional independence allows for a host of asymptotics to be derived for $\scrU$. }


%
%
\subsection{Local weak convergence for trees}

The aim of this section is to formalize the notion of convergence of neighborhoods of large random trees {\cg around typical vertices} to neighborhoods of limiting infinite discrete structures.  Local weak convergence of discrete random structures  has now become quite standard in probabilistic combinatorics see e.g. \cite{aldous-fringe,aldous-steele-obj, benjamini-schramm,van2023random}. {\cg Since this is technical there are two main motivating rationales for describing this in depth: 
\begin{inparaenuma}
 \item Rooted trees have a natural orientation towards the root; Aldous in \cite{aldous-fringe} showed that for families of random trees with size $n\to\infty$, if the subtree of descendants of a uniformly chosen vertex converges in distribution, then under general conditions, this implies convergence (in an appropriate sense) of the sequence of trees themselves, rooted uniformly at random, to a `one-ended' infinite rooted random tree;
 \item Convergence to such an infinite random objects implies not just asymptotic information for local functionals such as degree distribution, but even {\bf global} functionals such as the spectral distribution of the adjacency matrix (Theorem \ref{thm:rand-adj}).
 \end{inparaenuma} Thus, we start with notions of fringe decomposition for random trees required to formally address this mode of convergence. 
 Since our results also address non-tree networks, we describe notions of local weak convergence for general directed attributed networks in a later Subsection.}


\subsubsection{Extended fringe decomposition for marked trees}
\label{sec:fri-decomp}
Fix attribute space $\cS$ and assume for the rest of the paper that $\cS$ is a Polish space with distance metric $d_{\cS}$. For $n\geq 1$, let $ \bT_{n,\cS} $ be the space of all rooted trees on  $n$  vertices where every vertex has a mark in $\cS$. Let $ \bbT_{\cS} =
\cup_{n=0}^\infty \bT_{n,\cS} $ be the space of all finite rooted marked trees.  Here $\bT_{0,\cS} = \emptyset $ will be used to represent the empty tree (tree on zero vertices). For any $\bt\in  \bbT_{\cS} $ and $v\in \bt$, write $a(v)$ for the corresponding attribute of that vertex. Let $\rho_{\bt}$ denote the root of $\bt$.  For any $r\geq 0$ and $\bt\in \bbT_{\cS}$, let $B(\bt, r) \in \bbT_{\cS}$ denote the subgraph of $\bt$ of vertices within graph distance $r$ from $\rho_{\bt}$, viewed as an element of $\bbT_{\cS}$ and rooted again at $\rho_{\bt}$. 

 Given two rooted finite trees $\bs, \bt \in \bbT_{\cS}$,  say that $\bs \simeq \bt $ if, after ignoring all attribute information,  there exists a {\bf root
preserving} isomorphism between the two trees viewed as unlabelled graphs.  Given two rooted trees $\bt,\bs \in \bbT_{\cS}$ (adapting \cite[Equation 2.3.15]{van2023random}), define the distance 
\begin{equation}
\label{eqn:distance-trees}
	d_{\bbT_{\cS}}(\bt,\bs):= \frac{1}{1+R^*}
\end{equation}
where 
\begin{align*}
	R^* =\sup\{r&: B(\bt, r) \simeq B(\bs, r), \text{ and } \exists ~ \text{ isomorphism } {\cg \sF_r} \text{ between }\\
	& B(\bt, r) \text{ and } B(\bs, r) \text{ with } d_{\cS}(a(v), a({\cg \sF_r(v)})) \leq 1/r~ \forall~ v\in B(\bt, r)
	\}.
\end{align*}

 Next, fix a tree $\bt\in \bbT_{\cS}$ with root $\rho = \rho_\bt$ and a vertex $v\in \bt$ at (graph) distance $h$ from the root.  Let $(v_0 =v, v_1, v_2, \ldots, v_h = \rho)$ be the unique path from $v$ to $\rho$. The tree $\bt$  can be decomposed as $h+1$ rooted trees $f_0(v,\bt), \ldots, f_h(v,\bt)$, where $f_0(v,\bt)$ is the tree rooted at $v$, consisting of all vertices for which there exists a path from the root passing through $v$, and for $i \ge 1$, $f_i(v,\bt)$ is the subtree rooted at $v_i$, consisting of all vertices for which the path from the root passes through $v_i$ but not through $v_{i-1}$. 
  Call the map $(v,\bt) \leadsto \bbT_{\cS}^\infty$ where $v\in \bt$, defined via, 
\[F(v, \bt) = \left(f_0(v,\bt), f_1(v,\bt) , \ldots, f_h(v,\bt), \emptyset, \emptyset, \ldots \right),\]
as the fringe decomposition of $\bt$ about the vertex $v$. Call $f_0(v,\bt)$ the {\bf fringe} of the tree $\bt$ at $v$.
For $k\geq 0$, call $F_k(v,\bt) = (f_0(v,\bt) , \ldots, f_k(v,\bt))$ the {\bf extended fringe} of the tree $\bt$ at $v$ truncated at distance $k$ from $v$ on the path to the root ({\cg see Figure \ref{fig:fringe}}).
\begin{figure}
\centering
\includegraphics[scale=.24]{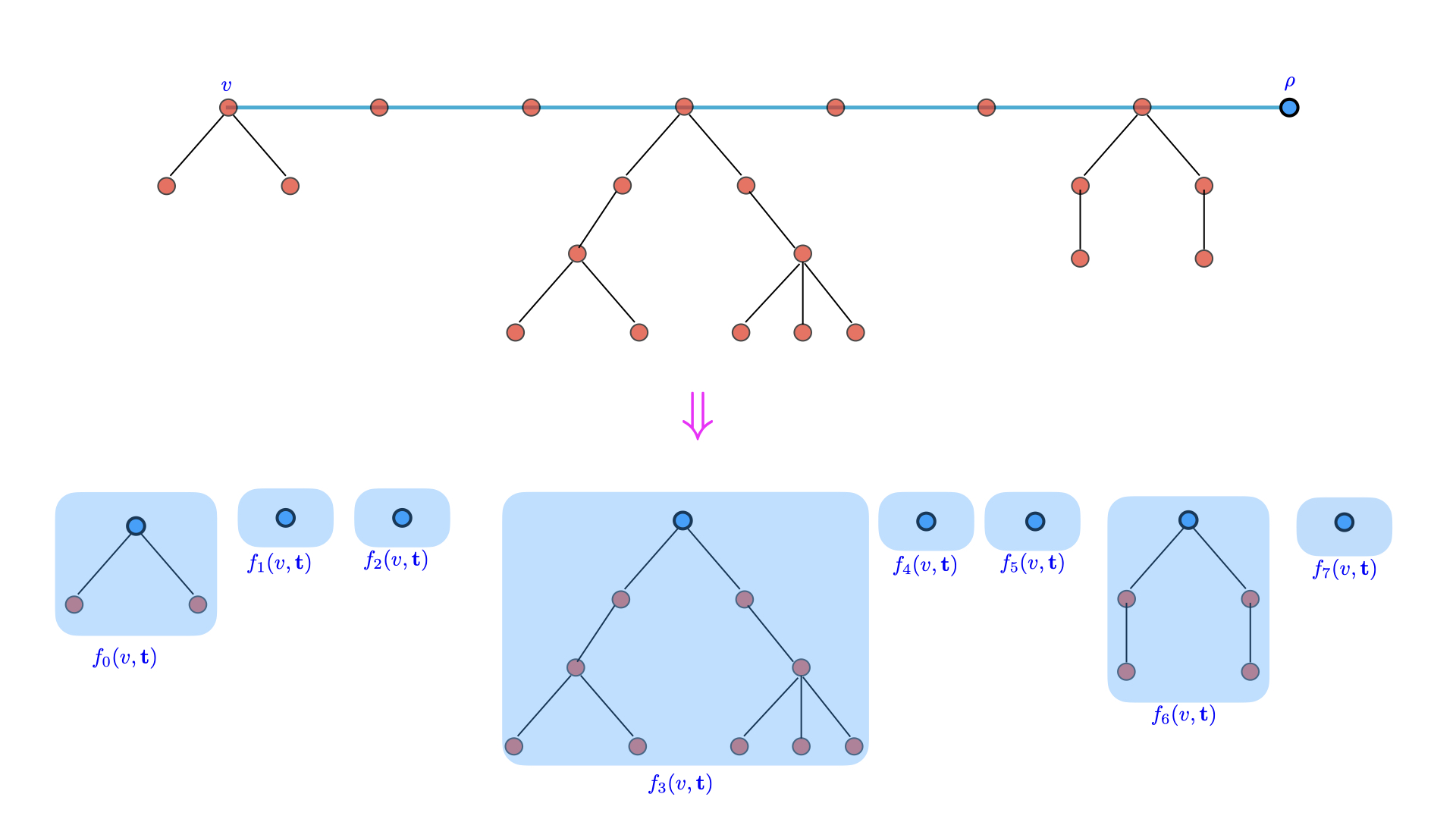}
\caption{Fringe decomposition around vertex $v$ of a finite tree rooted at $\rho$. Here the blue colors represent roots of the respecitve trees. }
\label{fig:fringe}
\end{figure}
Now consider the space $\bbT^\infty_{\cS}$. The metric in \eqref{eqn:distance-trees} extends in a straightforward fashion to $\bbT^\infty_{\cS}$ e.g. 
\begin{equation}
\label{eqn:dist-inf}
	d_{\bbT_{\cS}^\infty}((\bt_0, \bt_1, \ldots),(\bs_0, \bs_1, \ldots)):= \sum_{i=0}^\infty \frac{1}{2^i} d_{\bbT_{\cS}}(\bt_i, \bs_i). 
\end{equation}
{\cg Here $\set{\bt_i:i\geq 1}$ (respectively $\set{\bs_i:i\geq 1}$) are a countable sequence of trees each in $\bbT_{\cS}$.  We can also define analogous extensions of the metric} to $\bT_{\cS}^k$ for finite $k$.  

Next, an element $\bfomega = (\bt_0, \bt_1, \ldots) \in \bbT^\infty_{\cS}$, with $|\bt_i|\geq 1$ for all $ i\geq 0$,  can be thought of as a locally finite infinite rooted tree with a {\bf s}ingle  path to {\bf in}finity (thus called a {\tt sin}-tree \cite{aldous-fringe}), as follows: Identify the sequence of roots of $\set{\bt_i:i\geq 0}$ with the integer lattice $\Zbold_+ = \set{0,1,2,\ldots}$, equipped with the natural nearest neighbor edge set, rooted at $\rho=0$ ({\cg see Figure \ref{fig:sin}}).
\begin{figure}
\centering
\includegraphics[scale=.22]{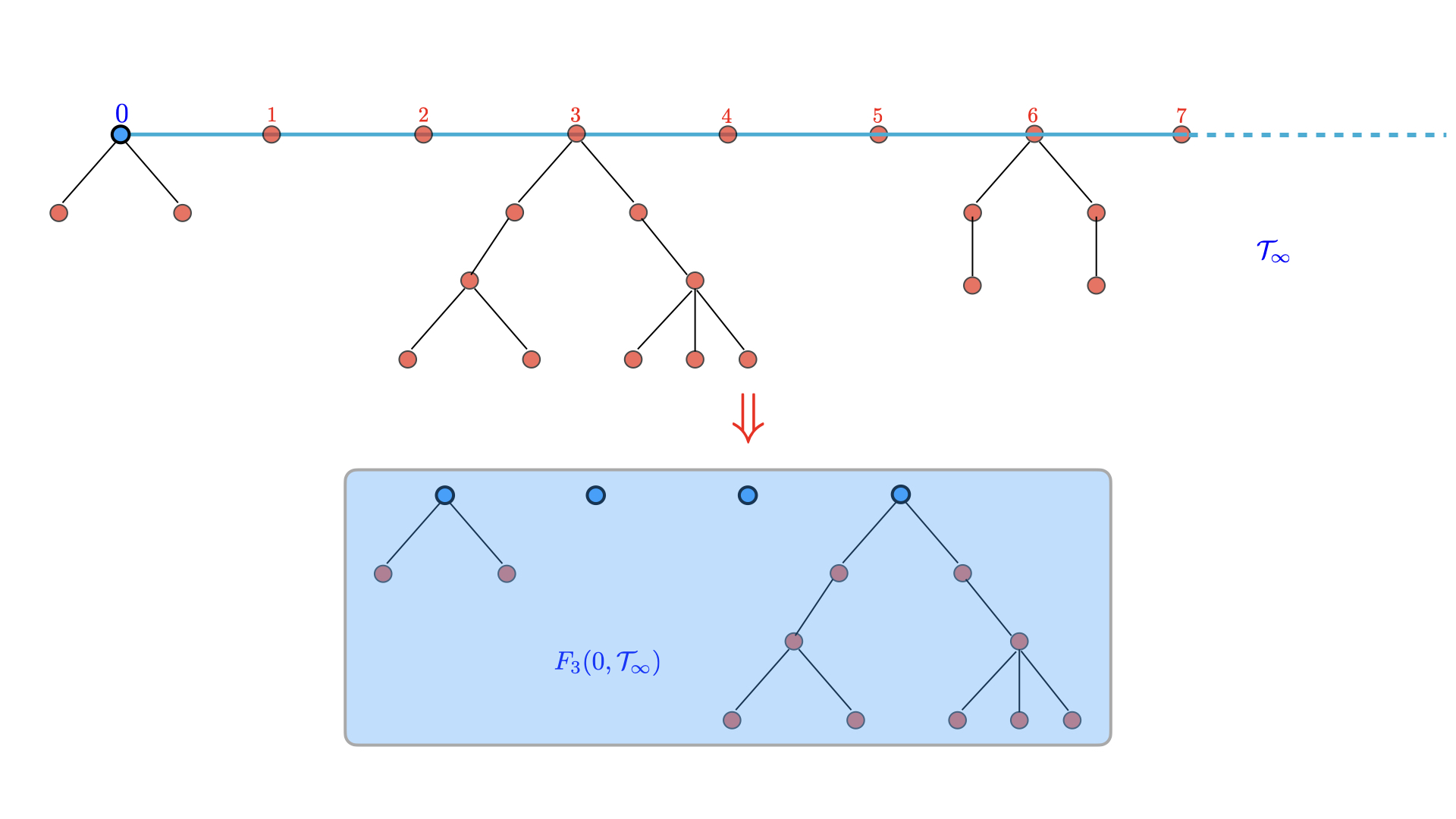}
\caption{A {\tt sin}-tree $\cT_\infty$, namely a tree rooted at $0$ with a single infinite path to infinity, and the corresponding extended fringe $F_3(0,\cT_\infty)$ upto level three about $0$. }
\label{fig:sin}
\end{figure}
  Analogous to the definition of extended fringes for finite trees, for any $k\geq 0$ write 
$F_k(0,\bfomega)= (\bt_0, \bt_1, \ldots, \bt_k)$. 
Call this the extended fringe of the tree $\bfomega$ at vertex $0$, till distance $k$, on the infinite path from $0$. Call $\bt_0 = F_0(0,\bfomega)$ the {\bf fringe} of the {\tt sin}-tree $\bfomega$. Now suppose $\prob$ is a probability measure on $\bbT^\infty_{\cS}$ such that, for $\TT:= (\bt_0(\TT), \bt_1(\TT),\ldots)\sim \prob$,  $|\bt_i(\TT)|\geq 1$ {\cg almost surely (a.s.)}  $\forall~ i\geq 0$. Then $\TT$ can be thought of as an infinite {\bf random} {\tt sin}-tree. 

\textcolor{black}{There is a slightly different representation for infinite {\tt sin}-trees that turns out to be handy in studying convergence properties of random tree processes.
\begin{wrapfigure}{r}{0.25\textwidth}
	  \begin{center}
		\includegraphics[width=0.24\textwidth]{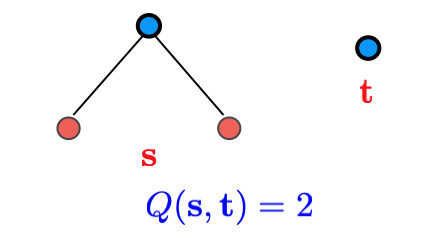}
	  \caption{}
	\label{fig:qst}
	  \end{center}
 	  \caption{Examples of multilayer networks }
	\end{wrapfigure} 
 Define a matrix $\vQ = (\vQ(\vs,\vt): \vs, \vt \in \bbT_{\cS})$ via the following prescription: suppose the root $\rho_{\vs}$ in $\vs$ has degree $\deg(\rho_{\vs}) \ge 1$, and let $(v_1,\ldots, v_{\deg(\rho_{\vs})})$ denote its children. For $1\leq i\leq v_{\deg(\rho_{\vs})}$, let {\cg $\fT(\vs, v_i)$} be the subtree below $v_i$ and rooted at $v_i$, viewed {\cg as} an element of $\bbT_{\cS}$. Write,
\begin{equation}
\label{eqn:Q-def}
	\vQ(\vs,\vt):= \sum_{i=1}^{\deg(\rho_{\vs})} \ind\set{d_{\bbT_{\cS}}({\cg \fT(\vs, v_i)}, \vt) = 0}. 
\end{equation} 
{\cg See the figure on the right for an example.} In words, $Q(\vs, \vt)$ counts the number of descendant subtrees of the root of $\vs$ that are isomorphic, both in the sense of topology and also associated marks, to $\vt$. If $\deg(\rho_{\vs})=0$, define $Q(\vs, \vt)=0$. Now consider a sequence,
{\cg 
\begin{equation}
\label{eqn:1025}
(\bar \vt_0, \bar \vt_1, \dots) \subseteq \bbT \mbox{ such that } Q(\bar \vt_i, \bar \vt_{i-1}) \ge 1~~ \forall ~ i\geq 1. 
\end{equation}
}
Then there exists a unique {\tt sin}-tree $\TT$ with infinite path indexed by $\Zbold_+$ such that $\bar \vt_i$ is the subtree rooted at $i$ for all $i \in \Zbold_+$. Conversely, it is easy to see, by taking $\bar \vt_i$ to be the union of (vertices and induced edges) of $\vt_0,\dots, \vt_i$ for each $i \in \Zbold_+$, that every infinite {\tt sin}-tree has {\cg a representation of the form \eqref{eqn:1025}}. Following \cite{aldous-fringe}, we call this the \emph{monotone representation} of the {\tt sin}-tree $\TT$.}

\subsubsection{Convergence on the space of trees}
\label{sec:fringe-convg-def}
Now for any $1\leq k\leq \infty$, let $\cM_{\pr}(\bbT_{\cS}^k)$ denote the space of probability measures on the associated space, metrized using the topology of weak convergence inherited from the corresponding metric on the space $\bbT_{\cS}^k$, see e.g. \cite{billingsley2013convergence}. 
Suppose $\set{\cT_n}_{n\geq 1} \subseteq \bbT_{\cS}$ be a sequence of {\bf finite} rooted random trees on some common probability space (for notational convenience,  assume $|\cT_n|= n$, all one needs is $|\cT_n|\convas \infty$). For $n\geq 1$ and for each fixed $k\geq 0$, consider the empirical distribution of fringes up to distance $k$ ($\delta\{\cdot\}$ below represents Dirac mass):
\begin{equation}
\label{eqn:empirical-fringe-def}
	\fP_{n}^k:= \frac{1}{n} \sum_{v\in \cT_n} \delta\set{F_k(v,\cT_n)}. 
\end{equation}
Thus $\set{\fP_{n}^k:n\geq 1}$ can be viewed as random sequence in  $\cM_{\pr}(\bbT_{\cS}^k)$, with accompanying {\cg notions} of almost sure convergence and convergence in distribution.

\begin{defn}[Local weak convergence]
	\label{def:local-weak}
	Consider two 
	notions of convergence of $\set{\cT_n:n\geq 1}$:
	\begin{enumeratea}
	    \item \label{it:fringe-a} Fix a probability measure $\varpi$ on $\bT_{\cS}$. Say that a sequence of trees $\set{\cT_n}_{n\geq 1}$ converges almost surely, in the fringe sense, to $\varpi$, if 
		\[\fP_n^{0}:= \frac{1}{n} \sum_{v\in \cT_n} \delta\set{F_0(v,\cT_n)} \convas \varpi, \qquad \text{ as } n\to\infty.  \]
	Denote this convergence by $\TT_n\probfr \varpi$ as $n\to\infty$.
	    \item \label{it:fringe-b} Say that a sequence of trees $\set{\cT_n}_{n\geq 1}$ converges almost surely, in the {\bf extended fringe sense}, to a limiting infinite random {\tt sin}-tree $\TT_{\infty}$ if for all $k\geq 0$ one has
	  \[\fP_n^k \stackrel{\mathrm{a.s.}}{\longrightarrow} \prob\left(F_k(0,\TT_{\infty}) \in \cdot \right), \qquad \text{ as } n\to\infty. \]
	Denote this convergence by $\TT_n\probcrf \TT_{\infty}$ as $n\to\infty$.
	\end{enumeratea}
\end{defn} 
{\cg Intuitively, fringe convergence, namely (a) above implies that, if we look at the subtree below \emph{a typical} (i.e. selected uniformly at random) vertex, then the corresponding random tree converges in distribution as the network size $n\to\infty$. Extended fringe convergence implies not just the structure of the neighborhood {\bf below} typical vertices, but the entire local neighborhood, within any finite distance converges. Convergence in (b) above clearly implies convergence in notion (a) with $\varpi(\cdot) =\pr(F_0(0, \cT_{\infty}) = \cdot)$.
More surprisingly, if the limiting distribution $\varpi$ in (a) has a certain `stationarity' property (defined next), \emph{convergence in the fringe sense implies convergence in the extended fringe sense}. }

\textcolor{black}{\begin{defn}[Fringe distribution \cite{aldous-fringe}] \label{fringedef}
	Say that a probability measure $\varpi$ on $\bbT_{\cS}$ is a fringe distribution if 
	\[\sum_{\vs} \varpi(\vs) \vQ(\vs, \vt) = \varpi(\vt), \qquad \forall \ \vt \in \bbT_{\cS}. \]
\end{defn}
For any fringe distribution $\varpi$ on $\bbT_\cS$, one can uniquely obtain the law $\varpi^{EF}$ of a random {\tt sin}-tree $\TT$ with monotone decomposition $(\bar\bt_0(\TT), \bar\bt_1(\TT),\ldots)$ such that for any $i \in \Zbold_+$, any $\bar \vt_0, \bar \vt_1, \dots$ in $\bbT_\cS$,
\begin{equation}\label{ftoef}
\varpi^{EF}((\bar\bt_0(\TT), \bar\bt_1(\TT),\ldots, \bar\bt_i(\TT)) = (\bar \vt_0,\vt_1,\dots,\vt_i)) := \varpi(\vt_i) \prod_{j=1}^{i}Q(\vt_i,\vt_{i-1}),
\end{equation}
where the product is taken to be one if $i=0$. The following Lemma follows by adapting the proof of \cite[Propositions 10 and 11]{aldous-fringe} and {\cg the proof is thus omitted}.}

\textcolor{black}{\begin{lemma}\label{ftoeflemma}
Suppose a sequence of trees $\set{\cT_n}_{n\geq 1}$ converges almost surely, in the fringe sense, to $\varpi$. Moreover, suppose that $\varpi$ is a fringe distribution in the sense of Definition \ref{fringedef}. Then $\set{\cT_n}_{n\geq 1}$ converges almost surely, in the extended fringe sense, to a limiting infinite random {\tt sin}-tree $\TT_{\infty}$ whose law $\varpi^{EF}$ is uniquely obtained from $\varpi$ via \eqref{ftoef}.
\end{lemma}}

\textcolor{black}{Both notions imply convergence of functionals such as the degree distribution. For example, in notion (a), letting $\cT_{\varpi} \sim \varpi$ with root denoted by $0$ say, convergence in notion (a) in particular implies,  for any {\cg $\ell \geq 0$}, 
\begin{equation}
\label{eqn:deg-convg-fr}
	\frac{\#\set{v\in \TT_n, \deg(v) = {\cg \ell} +1}}{n} \convas \prob(\deg(0,\cT_{\varpi})={\cg \ell}).
\end{equation}
However, both convergences give a lot more information about the asymptotic properties of $\cT_n$, beyond its degree distribution, including convergence of some global functionals such as spectral properties of adjacency matrices \cite{bhamidi2012spectra}, and functionals described in Section \ref{sec:func-interest}.} 

\subsection{Local weak convergence for directed graphs}\label{lwcnt}
We will now outline the notion of local weak convergence for general directed graphs, largely following \cite{garavaglia2020local}.
Let $\mathbb{G}$ denote the space of directed, marked, rooted graphs with finite in-degrees and out-degrees \cite[Definition~3.8]{garavaglia2020local}. Elements of this space comprise directed graphs $G$ with a distinguished vertex $\emptyset$ called the root. Moreover, each vertex carries a mark which, for simplicity, is assumed here to take values in a finite set $\cS$. There is a natural concept of isomorphism `$\cong$' between two elements of $\mathbb{G}$ \cite[Definition 3.9]{garavaglia2020local}: two such elements are isomorphic if there exists a bijection between the vertex sets which maps root to root and preserves the directed adjacency structure and marks of the vertices. We will denote by $\mathbb{G}_{\star}$ the quotient space of $\mathbb{G}$ with respect to the equivalence given by isomorphisms. A generic element of $\mathbb{G}_{\star}$ will be denoted by $(G, \emptyset, M(G))$, where $M(G)$ denotes the set of marks on vertices of $G$.

For any $k \in \mathbb{N} \cup \{0\}$, the $k$-neighborhood of the root $\emptyset$ in $(G, \emptyset, M(G))$, denoted by $U_{\le k}(\emptyset)$, is obtained by progressively exploring vertices using incoming edges \emph{in the opposite direction}, starting from the root, up till graph distance $k$ from the root and revealing the marks and connectivity structure of the explored vertices \cite[Definition 3.10]{garavaglia2020local}. Note that $U_{\le k}(\emptyset)$ thus constructed is a marked directed subgraph of $(G, \emptyset, M(G))$. This leads to a natural distance on the space $\mathbb{G}_{\star}$: for two elements $(G, \emptyset, M(G))$, $(G', \emptyset', M(G'))$ in $\mathbb{G}_{\star}$, define $d_{loc}((G, \emptyset, M(G)), (G', \emptyset', M(G'))) := 1/(\fK + 1)$, where $\fK := \inf\{k \ge 1 : U_{\le k}(\emptyset) \not\cong U_{\le k}(\emptyset')\}$. Unlike for undirected graphs, this distance is not a metric on $\mathbb{G}_{\star}$, but a pseudonorm, as the distance between two distinct elements in $\mathbb{G}_{\star}$ can be zero. The reason is that, in the directed setting, edges are explored only in one direction, leaving parts of the graph unexplored \cite[Figure 4]{garavaglia2020local}. Thus, two rooted marked directed graphs can have identical explorable root neighborhoods without being isomorphic. However, if $d_{loc}((G, \emptyset, M(G)), (G', \emptyset', M(G')))=0$, it can be shown that the incoming neighborhoods for the root (the possibly infinite subgraph that can be explored starting from the root) in the two graphs are isomorphic. Namely, denoting the respective incoming neighborhoods by $U_{\infty}(\emptyset)$ and $U_{\infty}(\emptyset')$, we have $U_{\infty}(\emptyset) \cong U_{\infty}(\emptyset')$. This leads to an equivalence relation on $\mathbb{G}_{\star}$. Denoting by $\tilde{\mathbb{G}}_{\star}$ the quotient space of $\mathbb{G}_{\star}$ under this equivalence relation, it follows that $(\tilde{\mathbb{G}}_{\star}, d_{loc})$ is a Polish space.

Now, we have all the tools to describe the concept of local weak convergence in the directed graph setting. For a sequence $\{(\cG_n, M(\cG_n))\}_{n \in \mathbb{N}}$ of marked, directed random graphs, define 
$$
\mathbb{P}_n := \frac{1}{|V(\cG_n)|} \sum_{v \in V(\cG_n)} \delta\set{(\cG_n, v, M(\cG_n))}
$$
as the empirical measure corresponding to selecting the root in $(\cG_n, M(\cG_n))$ uniformly at random in the vertex set $V(\cG_n)$.

\begin{defn}[Local weak convergence for directed graphs]\label{def:lwc}
Consider a sequence $\{(\cG_n, M(\cG_n))\}_{n \in \mathbb{N}}$ of marked, directed random graphs defined on the same probability space. We say $(\cG_n, M(\cG_n))$ converges \emph{almost surely in the local weak sense} to a random element $\mathcal{G}^* \in \tilde{\mathbb{G}}_{\star}$ with law $\mathbb{P}^*$ if for any bounded continuous function $f : \tilde{\mathbb{G}}_{\star} \rightarrow \mathbb{R}$,
$$
\mathbb{E}_{\mathbb{P}_n}(f) \convas \mathbb{E}_{\mathbb{P}^*}(f) \qquad n \to \infty,
$$
where $\mathbb{E}_{\mathbb{P}_n}$ and $\mathbb{E}_{\mathbb{P}^*}$ respectively denote expectation taken with respect to the laws $\mathbb{P}_n$ and $\mathbb{P}^*$. In this case, we write $\cG_n \stackrel{\mbox{$\operatorname{a.s.}$-\bf loc}}{\longrightarrow}\mathcal{G}^*$.
\end{defn}

\subsection{Functionals of interest}
\label{sec:func-interest}
Since the majority of the paper will deal with discrete type space, we will mainly phrase functionals of interest in this context, \textcolor{black}{briefly discussing the more general context in Section \ref{futw}}.  Fix a sequence of network models $\set{\cG_n:n\geq 1}$. For $n\geq 1$ and fixed \textcolor{black}{$k\geq 1$}, we will let $N_n(k)$ denote the number of vertices with degree $k$ in $\cG_n$. We let $\vp_n = \set{N_n(k)/|{\cg V(\cG_n)}|:k\geq 1}$ denote the corresponding empirical probability mass function. In the attributed network setting, we let the joint empirical distribution of degree and attributes and the marginal distribution of the attributes be:
\begin{equation}
\label{eqn:jt-dist}
	\varkappa_n(\cdot) = \frac{1}{|{\cg V(\cG_n)}|} \sum_{v\in V(\cG_n)} \delta_{\set{(\deg(v), a(v))}}, \qquad \hat{\mvpi}_n(\cdot) = \varkappa_n(\bN\times \cdot). 
\end{equation}
We let $M_n$ denote the maximal degree of all vertices in $\cG_n$. When $\cS$ is discrete, we will without loss of generality assume we have the discrete metric $d(a,a^\prime) = \ind\set{a=a^\prime}$.  We let $M_n^a$ denote the maximal degree amongst all vertices with attribute $a$. Similarly, in the discrete setting, we let $\vp_n^a$ denote the empirical degree distribution of all vertices with attribute $a$. 

 Next, once again for discrete latent space $\cS$ one can define empirical notions of homophily (propensity of vertices to connect to other vertices of the same type) and heterophily (propensity to connect to vertices of other types) \cite{park2007distribution}. For fixed $n$, let $\cE_n$ denote the total edge set of $\cG_n$; for $a\in \cS$, let $\cV_{n,a}$ be the set of nodes of type $a$, and for $a,a'\in \cS$, let $\cE_{n, aa'}$ be the set of edges {\cg in either direction (since our underlying graphs are directed)} between nodes of type $a$ and $a'$. Let $p_n = |\cE_n|/{n \choose 2} $ be the edge density. For $a\in \cS$,  $D_{n, a} = |\cE_{n, aa}|/({|\cV_{n, a}| \choose 2} p_n) $ contrasts density of edges within the cluster of nodes of type $a$ versus a setting where all edges are randomly distributed; thus $D_{n, a} > 1$ signals homophilic characteristics of type $a$ nodes while $D_{n,a}<1$ signifies heterophilic nature of type $a$ vertices. Similarly, for $a\neq a'$, $H_{n, aa'} = |\cE_{n, aa'}|/(|\cV_{n, a}||\cV_{n,a'}|p_n)$ denotes propensity of type $a$ nodes to connect to type $a'$ nodes as contrasted with random placement of edges \textcolor{black}{with probability equal to the} global edge density; as before  for $a\neq a^\prime $, $H_{n, aa'} > 1$ signals higher propensity of edges being present between nodes with these attributes than one would expect out of pure random placement of the edges.  

Finally we define PageRank scores \cite{page1999pagerank}. For the rest of the paper, we view the networks under consideration as directed trees with edges pointing from offspring to parents (i.e. if a new vertex $v$ is attached to an extant vertex $v^\prime$ then this is recorded as a directed edge $v\leadsto v^\prime$). 

    \begin{defn}[PageRank scores with damping factor $c$]
    \label{def:page-rank}
	    For a directed graph $\cG = (\cV, \cE)$, the PageRank scores of vertices $v\in \cV$ with damping factor $c$ is the stationary distribution $(\fR_{v,c}: v\in \cG)$ of the following random walk: at each step, with probability $c$, follow an outgoing edge (uniform amongst available choices) from the current location in the graph while, with probability $1-c$, restart at a uniformly selected vertex in the entire graph. These scores are given by the linear system of equations:
	    \begin{equation}
	        \label{eqn:page-rank}
	        \fR_{v,c} = \frac{1-c}{n} + c\sum_{u\in \cN^{-}(v)} \frac{\fR_{u,c}}{d^{+}(u)}
	    \end{equation}
	    where $\cN^{-}(v)$ is the set of vertices with edges pointed at $v$ and $d^{+}(u)$ is the out-degree of vertex $u$. 
	\end{defn}
     At vertices with zero out-degree (e.g. the root of a directed tree), the random walk stays in place with probability $c$ and jumps to a uniformly chosen vertex with probability $1-c$. {\cg For such a vertex $v$, the associated stationary distribution mass $\hat \fR_{v,c}$ then takes the form $\hat \fR_{v,c} = c\hat \fR_{v,c} +  \frac{1-c}{n} + c\sum_{u\in \cN^{-}(v)} \frac{\fR_{u,c}}{d^{+}(u)}$. The corresponding PageRank value is defined as $\fR_{v,c} := (1-c)\hat \fR_{v,c}$ to keep the formula \eqref{eqn:page-rank} the same for all vertices.}

\section{Main Results: Tree networks}
\label{sec:main-res}

To ease the reader into the heart of the paper, we will start with the tree setting $\scrP(1,\mvpi,\kappa)$ (all out-degrees one). These results will be shown to generalize to non-tree models ($\mvm \not\equiv 1$) in the following section. We will start with a precise definition of resolvability that allows the potential connection between model class $\scrP$ and $\scrU$. Then,  the rest of this section focuses on the linear preferential attachment ($\gamma \equiv 1$ case), when the attribute space $\cS$ is finite, describing local weak limits, and subsequent implications for functionals of interest. 

\subsection{Resolvability}

Recall the main model class of interest $\scrP$ and the corresponding model $\scrU$ with construction as in Definition \ref{defn:ctbp}, {\cg where from \eqref{eqn:rate-finite}, the weight measure $\mvnu$ plays the role of modulating the rate of new vertices of specific types being born into the system in the continuous time construction of $\scrU$}. For $\set{\cG_n}_{n\geq 0} \sim \scrU$, let $\hat{\mvpi}_n= n^{-1}\sum_{i=1}^n \delta\set{a(v_i)} $ denote the empirical measure of attribute types in $\cG_n$. 

\begin{defn}[Resolvability]
	\label{def:resolv} 
	Say that a model $\scrP(\gamma, \mvpi, \kappa)$ is resolvable if  there exists a probability measure $\mvnu$ such that,  for the matching model $\scrU(\gamma, \mvnu, \kappa)$, the corresponding limit density $\hat{\mvpi}_n \probc \mvpi_\infty$ exists and further $\mvpi_\infty = \mvpi$. Call $\mvnu$ \textcolor{black}{a} $\scrP\leadsto\scrU$ resolving distribution for $\mvpi$. 
\end{defn}

The next few subsections describe resolvability of the finite attribute, linear attachment setting and elaborate on some of the implications of this resolvability.  Unlike the introduction where we used $\tilde{\cG}_n$ for network from model class $\scrP$ and $\cG_n$ for a network from model class $\scrU$, to minimize notational overhead, we will use $\set{\cG_n:n\geq 0}$ for a generic sequence of growing networks; the driving model will always be clear from context.

\subsection{Local weak convergence for finite attribute models}

For the rest of this section, assume the state space $\cS = [K] := \set{1,2,\ldots, K}$ is a finite set. Let $\cP({\cg [K]})$ denote the space of all probability measures on ${\cg [K]}$; here this is just the $K$-dimensional simplex. For simplicity, we will often use matrix notation $\kappa = (\kappa_{a,b})_{a,b\in [K]}$ for the attractiveness kernel. Fix model class $\scrP(\gamma \equiv 1, \mvpi, \kappa)$.

\begin{ass}
	\label{ass:finite-case}
 Assume $\mvpi(\set{a})>0 ~ \forall a\in {\cg [K]}$ and $\kappa_{a,b} >0~ \forall a,b \in {\cg [K]}$. 
\end{ass}

Next define (in the interior of $\cP({\cg [K]})$) the function:
\begin{equation}
\label{eqn:vpi-def}
	 V_{\mvpi}(\vy):= 1-\frac{1}{2}\sum_{j\in {\cg [K]}} \pi_j\left(\log(y_j) + \log(\sum_{k\in {\cg [K]}} y_k \kappa_{k,j})\right).
\end{equation}

{\cg Writing $n_0 = |\cG_0|$ for the size of the initial graph, let $\tilde{Y}_a^{\sss(n)}:= {\sum_{v\in \cG_n: a(v) = a} \deg(v,n)}/{2(n+n_0)}$, $a\in [K]$, denote the proportion of degrees within each attribute in $\cG_n$.}
\begin{lemma}[{\cite[P7]{jordan2013geometric}}]\label{lem:minimizer}
	Under Assumption \ref{ass:finite-case}, $V_{\mvpi}(\cdot)$ has a \emph{unique} minimizer $\mveta:= \mveta(\mvpi) = (\eta_1(\mvpi), \ldots, \eta_K(\mvpi))$ in the interior of $\cP({\cg [K]})$. {\cg Moreover, as $n \rightarrow \infty$, $(\tilde{Y}_a^{\sss(n)}:a\in [K]) \convas \mveta(\mvpi)$.}
\end{lemma}
{\cg We briefly describe how the above result arose in \cite{jordan2013geometric} in the context of analyzing the degree distribution of $\scrP$. It was observed in \cite{jordan2013geometric} that the process $(\tilde{Y}_a^{\sss(n)} : a \in [K])_{n \ge 1}$ evolves as a `stochastic approximation' scheme whose `drift' is given in terms of the gradient of the function $V_{\mvpi}$. From this (which can be seen as a `noisy gradient descent'), the result follows.}

Next define,
\begin{equation}
\label{eqn:phi-def}
	\nu_b:= \frac{\pi_b}{\sum_{l=1}^K \kappa_{l,b}\eta_l}, \qquad \phi_{a,b}:= \kappa_{a,b}\nu_b, \qquad \phi_a := \sum_{b=1}^K \phi_{a,b} = 2- \frac{\pi_a}{\eta_a},
\end{equation}
where the final identity follows from \cite[Proposition 3.2]{jordan2013geometric}.

\begin{defn}
	\label{def:model-u-corr-p}
	Fix $\mvpi, \kappa$ satisfying Assumptions \ref{ass:finite-case} and construct $\mvnu$ as in \eqref{eqn:phi-def}. Using $\mvnu, \kappa$,  consider the model class $\scrU(1,\mvnu, \kappa)$ started with ancestor (root) of an arbitrary type $a^\prime \in [K]$, constructed using the branching process $\set{\BP_{a^\prime;(1,\mvnu, \kappa)}(t):t\geq 0}$ as in Definition \ref{defn:ctbp}. \textcolor{black}{For (possibly random) $A, T \ge 0$, we will denote by $\cL(\BP_{A;(1,\mvnu, \kappa)}(T))$ the law of the progeny tree of $\BP_{A;(1,\mvnu, \kappa)}(T)$ rooted at the ancestor.}
\end{defn}

\begin{theorem}[Local weak convergence for linear model]\label{thm:fringe-convergence}
	Fix finite attribute space $\cS = [K]$. Let $\gamma \equiv 1$ and fix $\mvpi, \kappa$ satisfying Assumptions \ref{ass:finite-case} and consider the sequence of networks  constructed as $\set{\cG_n}_{n\geq 0} \sim\scrP(1, \mvpi, \kappa)$. Then, as in Definition \ref{def:local-weak} \eqref{it:fringe-b}, 
	\[\cG_n \probcrf \cT_\infty,\]
	where $\cT_\infty$ is the unique {\tt sin}-tree with fringe distribution 
	\begin{equation}
	\label{eqn:fringe-limit}
		\fpm_{\infty}(\cdot) = \cL(\BP_{A;(1,\mvnu, \kappa)}(\tau)).
	\end{equation}
{\cg Here $\set{\BP_{A;(1,\mvnu, \kappa)}(t):t\geq 0}$, is a branching process as in Definition \ref{def:model-u-corr-p} with the type of the root $A\sim \mvpi$ and $\tau$ is an independent exponential random variable with rate $\lambda =2$. }
\end{theorem}
{\cg
\begin{rem}
\begin{enumeratea}
\item The local weak limit of the preferential attachment model with one type has been previously established as the law of a Yule process stopped at an independent $\Exp(2)$ time \cite{rudas2007random}. The above result shows that, although the associated branching process in the multi-type case has a more involved description, the law of the independent stopping time is still $\Exp(2)$. In our approach (very different from that of \cite{rudas2007random}), $\fpm_\infty$ in \eqref{eqn:fringe-limit} arises as the unique distribution that satisfies a certain recursive form (see Prop. \ref{prop:prop-of-fp}(b)) obtained from the stochastic evolution equations of $\set{\cG_n}_{n\geq 0} \sim\scrP(1, \mvpi, \kappa)$. 
\item The proof of Theorem \ref{thm:fringe-convergence} can be adapted to show that the network model $\scrU(1,\mvnu, \kappa)$ constructed in Definition \ref{def:model-u-corr-p} has the same local weak limit described in the Theorem. In particular, the attribute proportions in this model class approach $\mvpi$ as the network size grows, and hence the network model $\scrP(1, \mvpi, \kappa)$ is \emph{resolvable} in the sense of Definition \ref{def:resolv}. Thus, the above result provides more evidence towards the meta-principle that resolvability is key to model classes $\scrP$ and $\scrU$ having the same local weak limits.
\end{enumeratea}
\end{rem}
}
\textcolor{black}{For the $K=1$ (single attribute) case, the local limit was obtained before by \cite{rudas2007random} building on ideas from \cite{nerman1981convergence,jagers1984growth}.} {\cg The next section describes explicit asymptotics of various local functionals that can be derived from the above theorem. In passing, we describe how the above local weak limit implies even convergence of global functions. Let $\set{\cG_n}_{n\geq 0} \sim\scrP(1, \mvpi, \kappa)$ as above and let $\vA_n$ denote the adjacency matrix of $\cG_n$, $\set{\lambda_i^{\sss(n)}:1\leq i\leq n}$ denote the eigen-values and $\hat{\mu}_n = n^{-1} \sum_{i=1}^n \delta_{\lambda_i^{\sss(n)}}$ denote the empirical spectral distribution where $\delta$ denotes the Dirac delta function. 

\begin{theorem}
\label{thm:rand-adj}
Let $\set{\cG_n}_{n\geq 0} \sim\scrP(1, \mvpi, \kappa)$ satisfying the assumptions in Theorem \ref{thm:fringe-convergence}. Then there exists a deterministic distribution $\mu_\infty$ (whose specific form depends on the parameters $\mvpi, \kappa$) such that $\hat{\mu}_n \convd \mu_\infty$. The limit distribution has an infinite set of atoms in $\bR$.
\end{theorem}

The proof of this result follows directly by combining the extended fringe convergence in Theorem \ref{thm:fringe-convergence} with \cite[Theorem 4.1]{bhamidi2012spectra}. 
}

\subsection{Asymptotics for degree distribution, PageRank scores and homophily measures}
In this section, we will describe the implications of Theorem \ref{thm:fringe-convergence} for asymptotics of various functionals of interest. Fix model $\scrP(1, \mvpi, \kappa)$ and consider the corresponding model $\scrU(1, \mvnu, \kappa)$ as in Definition \ref{def:model-u-corr-p} using the quantities defined in \eqref{eqn:phi-def}. Consider the matrix, 
\begin{equation}
	\label{eqn:matrix-def}
	\vM = \left(\vM_{a,b}:=\frac{\phi_{a,b}}{2 - \phi_a}\right)_{a, b\in [K]}.
\end{equation} 
\textcolor{black}{For $a,b \in [K]$, it can be checked that $\vM_{(a,b)}$ corresponds to the expected number of type $b$ children by a type $a$ ancestor (root) in the branching process $\cL(\BP_{A;(1,\mvnu, \kappa)}(T))$ constructed in Definition \ref{def:model-u-corr-p}, with $T = \tau \sim \Exp(2)$ independent of the branching process. Matrices analogous to $\vM$, which are more generally obtained via Laplace transforms of the intensity measures of the associated reproduction processes, are known to play a key role in asymptotics of multi-type branching processes \cite{jagers1996asymptotic,jagers1989general}. These Laplace transforms are evaluated at the \emph{Malthusian rate of growth} which quantifies the rate of exponential growth of the total population size of the branching process. In particular, the Perron-Frobenius eigen-value is one and the associated (normalized) left eigen-vector gives the asymptotic proportions of vertices of different types in the population. The following Proposition will be a crucial ingredient in our analysis.}
 
\begin{prop}
	\label{prop:left-eigen}
	The matrix {\cg $\vM$ in \eqref{eqn:matrix-def}} has a unique Perron-Frobenius eigen-value $\lambda_{(\PF,1)}(\vM)\equiv 1$. Further the attribute density $\mvpi$ is the left eigen-vector of $\vM$ corresponding to $\lambda_{(\PF,1)}(\vM)\equiv 1$. 
\end{prop}

{\cg \begin{proof}
	 Since $\vM$ has strictly positive entries, by Perron-Frobenius theorem it is enough to show that $\mvpi$ is a left eigen-vector with eigen-value one. By the last identity in \eqref{eqn:phi-def}, for any $a\in [K]$, $2-\phi_a = \pi_a/\eta_a$. Thus for any $b\in [K]$,
	\begin{align}
		(\mvpi \vM)_b = \sum_{a\in [K]} \pi_a \vM_{a,b} =\sum_{a\in [K]} \pi_a \frac{\eta_a}{\pi_a} \frac{\kappa_{a,b}\pi_b}{\sum_{l\in [K]} \kappa_{l,b} \eta_l} = \pi_b \notag. 
	\end{align}
	\end{proof}
}
Now recall the empirical pmfs $\vp_n^a$ for vertices of each attribute type and homophily and heterophily statistics $\set{D_{n,a}:a\in [K]}$ and $\set{H_{n, (a,a^\prime)}: a\neq a^\prime \in [K]}$ as defined in Section \ref{sec:func-interest}. 
\begin{theorem}\label{thm:deg-homophil-linear}
	Assume the model $\scrP(1, \mvpi,\kappa)$ satisfies Assumptions \ref{ass:finite-case}.
	\begin{enumeratea}
		\item  \textcolor{black}{For each $a\in [K]$, $\vp_n^a \convas \vp_\infty^a$ where 
	\begin{equation}
		\label{eqn:pinf-def}
		\vp^a_\infty(k) = \frac{2}{\phi_a}\frac{\Gamma\left(1+\frac{2}{\phi_a}\right) \Gamma(k)}{\Gamma\left(k+1+\frac{2}{\phi_a}\right)}, \qquad k\geq 1.  
	\end{equation}}
	In particular $\vp_\infty^a(k)\sim k^{-(1+2/\phi_a)}$ as $k\to\infty$.
	\item The homophily and heterophily statistics $\set{D_{n,a}: a\in [K]}$ and $\set{H_{n, (a,a^\prime)}: a\neq a^\prime \in [K]}$ satisfy the asymptotics, 
	\[D_{n,a} \convas \frac{\vM_{(a,a)}}{\pi_a}, \qquad H_{n, (a,a^\prime)} \convas \frac{1}{2}\left[\frac{\vM_{(a^\prime,a)}}{\pi_a} + \frac{\vM_{(a, a^\prime)}}{\pi_{a^\prime}} \right]. \]
	\end{enumeratea}
	  
\end{theorem}
\begin{rem}
	The degree distribution asymptotics in part (a) above was previously derived using stochastic analytic techniques in \cite{jordan2013geometric}. One goal of this paper is to show how such techniques can be extended to yield local weak convergence of the entire network as described in Theorem \ref{thm:fringe-convergence}. This leads to asymptotics for more complex functionals as discussed below.  
\end{rem}

The next result shows that one can read off not just degree distribution asymptotics but also limit information for functionals which require connectivity information beyond one-step neighborhoods of vertices. Recall the PageRank scores $\set{\fR_{v,c}(n): v\in \cG_n}$ in Definition \ref{def:page-rank}. Recall that $\cG_n$ is directed with edges from child to parent.  For any $v\in \cG_n$ and $l \in \mathbb{N}$, let $P_l(v,n)$ denote the number of \emph{directed} paths of length $l$ that end at $v$ in $\cG_n$. Since $\cG_n$ is a directed tree, it is easy to check that the PageRank scores have the following explicit formulae for any vertex $v$: 
\begin{equation}
	\label{eqn:non-root-page-rank}
	\fR_{v,c}(n)=\frac{(1-c)}{n}\left(1 + \sum_{l=1}^\infty c^l P_l(v,n)\right).
\end{equation}

For the sequel, it will be easier to formulate results in terms of the \emph{graph normalized} PageRank scores \cite{garavaglia2020local} $\set{R_{v,c}(n):v\in \cG_n} := \set{n\fR_{v,c}(n): v\in \cG_n}$. Define the empirical distribution of normalized PageRank scores, 
$$\hat{\mu}_{n, \PR} := n^{-1} \sum_{v\in \cG_n} \delta\set{R_{v,c}(n)}.$$   
General results on the implication of local weak convergence of sparse graphs on the convergence of the empirical distribution of PageRank scores were derived in \cite{garavaglia2020local,banerjee2021pagerank}. In particular, the local weak convergence in Theorem \ref{thm:fringe-convergence} coupled with \cite{garavaglia2020local,banerjee2021pagerank} leads to the following result. Fix $a \in [K]$ and consider the branching process $\BP_{a;(1,\mvnu, \kappa)}(\cdot)$ as in Definition \ref{def:model-u-corr-p} starting with initial type $a$ and root denoted by $\emptyset$. For fixed $t\geq 0$ and $l \in \mathbb{N}$, let $P_{l,\emptyset}(t)$ denote the number of directed paths of length $l$ that end at the root in $\BP_{a;(1,\mvnu, \kappa)}(t)$ (i.e. the number of descendants at level $l$). For any $t\geq 0$ define, 
\begin{equation}
\label{eqn:t-power}
	\cR_{\emptyset, c}(t) = (1-c)\left(1 + \sum_{l=1}^{\infty} c^l P_{l,\emptyset}(t)\right).
\end{equation}
Motivated by Theorem \ref{thm:fringe-convergence}, let $\tau\sim \Exp(2)$ independent of $\BP_{a;(1,\mvnu, \kappa)}(\cdot)$ and write $\pr_a$ (resp. $\E_a$) for the distribution of the random finite rooted tree  $\BP_{a;(1,\mvnu, \kappa)}(\tau)$ (resp. expectation under $\pr_a$) and define the \emph{normalized} PageRank score at $\emptyset$ as,
\begin{equation}\label{lwlpr}
    \cR_{\emptyset, c} = \cR_{\emptyset, c}(\tau)=  (1-c)\left(1 + \sum_{l=1}^{\infty} c^l P_{l,\emptyset}(\tau)\right).
\end{equation}

\begin{theorem}[PageRank asymptotics]\label{thm:page-rank}
   Suppose $\set{\cG_n:n\geq 0}\sim \scrP(1,\mvpi,\kappa)$  satisfying Assumption \ref{ass:finite-case}. Fix $a\in [K]$ and consider the random variable $\cR_{\emptyset, c}$ defined in \eqref{lwlpr}. Then: 
   \begin{enumeratea}
   	\item For any $a\in [K]$, the random variable $\cR_{\emptyset, c}$ is finite a.s.
	\item For every continuity point $r$ of the distribution of $\cR_{\emptyset, c}$ under $\pr_a$, 
    \[n^{-1}\sum_{v\in \cG_n} \ind\set{a(v) = a, R_{v,c}(n) > r} \convas \pi_{a}\pr_a(\cR_{\emptyset, c} > r).  \]
   \end{enumeratea}
    
\end{theorem}
The next result derives explicit properties of the limiting PageRank distribution for the specific models considered in this paper. We first need some notation. Fix damping factor $c\in (0,1)$ and define the matrix
\begin{equation}
\label{eqn:matrix-mc-def}
	\vM^{\sss(c)} = \left(\vM_{a,b}^{\sss(c)}:= c\phi_{a,b} + \phi_a\ind\set{a=b} \right)_{a, b\in [K]}.
\end{equation}
{\cg Recall the definition of the matrix $\vM$ in \eqref{eqn:matrix-def} and the discussion below this equation describing the interpretation of $\vM$ as the mean matrix of the associated multi-type branching process run for a random exponential rate two time. In the proofs below, the matrix $\vM^{\sss(c)}$ will have a similiar interpretation, but now for a percolated version of the branching process (see Def. \ref{def:perc-bp})}

The next result quantifies tail asymptotics, expectations and some laws of large numbers associated with the limiting PageRank distribution in Theorem \ref{thm:page-rank}.

\begin{theorem}\label{thm:page-rank-tails}
	Consider the setting of Theorem \ref{thm:page-rank}. Let $\lambda_c$ denote the Perron-Frobenius eigen-value of the matrix $\vM^{\sss(c)}$ defined in \eqref{eqn:matrix-mc-def}.  Then:
	\begin{enumeratea}
		\item There exist finite constants $0< B_1\le B_2< \infty$ such that for all $a\in [K]$, the limiting PageRank distribution of vertices of type $a$, i.e. the distribution of $\cR_{\emptyset, c}$ under $\pr_a$, satisfies:
		\[B_1 r^{-2/\lambda_c} \leq \pr_a(\cR_{\emptyset, c} > r) \leq B_2 r^{-2/\lambda_c}.\]
		\item Further for any $a\in [K]$ and $n\geq 1$ define the strings of attributes of length $n+1$ starting with type $a$ attribute:
		\[\scI_n^{\sss(a)} = \set{\vj:= (j_0, j_1, \ldots, j_n) \in [K]^n: j_0 =a}.\]
		Then, 
		\begin{equation}
		\label{eqn:expec-page-rank}
			\E_a(\cR_{\emptyset, c})  = (1-c)\left[1+\sum_{n=1}^\infty c^n \sum_{\vj\in \scI_n^{\sss(a)}} \prod_{l=0}^{n-1} \left(\frac{\phi_{j_j, j_{l+1}}}{2-\phi_{j_l}}\right) \right].
		\end{equation}
		\item $\sum_{a\in [K]} \pi_a \E_a(\cR_{\emptyset, c}) =1$. Further, 
		$n^{-1} \sum_{v\in \cG_n} R_{v,c}(n) \convas 1$,
		and for each $a\in [K]$,
		\[n^{-1} \sum_{v\in \cG_n} R_{v,c}(n)\ind\set{a(v) =a} \convas \pi_a \E_a(\cR_{\emptyset, c}).   \]
	\end{enumeratea}
	 
\end{theorem}

We describe qualitative properties of $\lambda_c$ and then discuss the implications of Theorem \ref{thm:page-rank-tails}. 

\begin{prop}
	\label{prop:lc-bound}
	For any $c\in (0,1)$, $\lambda_c < 2$. Further, \textcolor{black}{$\lambda_c > \max_{a\in [K]} \phi_a$ and}
	\begin{equation}
	\label{eqn:lambc-bounds}
		(1+c) \min_{a\in [K]} \phi_a\leq \lambda_c \leq (1+c) \max_{a\in [K]} \phi_a. 
	\end{equation}
\end{prop}

\begin{rem} We now give some qualitative and quantitative implications of the above result:
	\begin{enumeratei}
		\item Part (a) of Theorem \ref{thm:page-rank-tails} implies that the tail exponent of the limiting PageRank distribution of vertices of a given attribute does {\bf not} depend on the attribute type; this should be contrasted with the result on the asymptotic degree distribution (Theorem \ref{thm:deg-homophil-linear}) where tails of the degree distribution depend on the attribute under consideration. This result thus gives information on the extremal behavior of the PageRank scores and their (lack of) dependence on the attribute type.
		\item Since the tail exponent of the limiting PageRank distribution is insensitive to the attribute type, parts (b) and (c) of Theorem \ref{thm:page-rank-tails} quantify the {\bf average (bulk)} behavior of the PageRank distribution and its dependence on the attribute type. 
		\item Further refined properties of the ``bulk'' behavior of PageRank scores (i.e. non-extremal behavior) and their dependence on attribute type, leveraging the technical tools developed in proving (b) above, is left for future work.  
	\end{enumeratei}
	
\end{rem}

The following Proposition gives a more succinct representation of the expected limiting PageRank in \eqref{eqn:expec-page-rank} in terms of an associated Markov chain. This representation (for PageRank as well as other centrality measures) will be used to quantify and compare associated sampling schemes in Section \ref{fairsec}. {\cg Recall the matrix $\vM$ from \eqref{eqn:matrix-def} which has Perron-Frobenius eigen-value one. Let $\mvPsi = (\Psi_1,\Psi_2, \ldots, \Psi_K)$ denote the corresponding right eigen-vector (thus for all $i\in [K]$, $\sum_{j} \vM_{i,j} \Psi_j = \Psi_i$), normalized so that $\sum_{a\in [K]} \pi_a \Psi_a =1$. }
\begin{prop}
	\label{prop:mc-descp}
		 	Consider the Markov chain $\vS:= \set{S_n:n\geq 0}$ on $[K]$ with transition probability matrix 
	\[\pr^\vS_i(S_1 =j):= \pr^{\vS}(S_1 = j|S_0 = i) = \frac{\vM_{i,j} \Psi_j}{\Psi_i}, \qquad i,j\in [K].\]
	Write $\E^{\vS}_i$ for the expectation operator under $\pr_i^{\vS}$. Consider the setting of Theorem \ref{thm:page-rank-tails}. Then
	\begin{equation}
	\label{eqn:211}
		\E_a(\cR_{\emptyset, c})  = \Psi_a \E^{\vS}_a\left[\frac{1}{\Psi_{S_N}}\right],
	\end{equation}
	where $N\sim \Geom(1-c) - 1$ independent of $\vS$. 
\end{prop}
{\cg
Since the proofs of these two Propositions are short, we give them here. 
\begin{proof}[\bf Proof of Propositions \ref{prop:lc-bound} and \ref{prop:mc-descp}]
By \cite[Theorem 2.7]{noutsos2006perron}, the Perron-Frobenius (PF) eigen-value is strictly increasing as a function of the matrix entires. Since for $c\in (0,1)$ $\vM^{\sss(c)} < \vM^{\sss(1)}$ (entry-wise domination), and further the PF eigenvalue of $\vM^{\sss(1)}$ is $2$ (see the proof of Prop. \ref{prop:left-eigen}), we have $\lambda_c <2$. 

By the characterization of $\mvh$ as the (strictly positive) right eigen-vector of the matrix $\vM^{\sss(c)}$, we have for each fixed $a$,
\[\sum_{b\in [K]} c \phi_{ab} h_b + \phi_a h_a = \lambda_c h_a.\]
Thus $\lambda_c > \max_{a\in [K]} \phi_a$. Further, since the Perron-Frobenius eigen-value is bounded between the minimum and maximum row sums, this yields \eqref{eqn:lambc-bounds}. This completes the proof of Proposition \ref{prop:lc-bound}. {\cg Finally,} Proposition \ref{prop:mc-descp} is just a reformulation of Theorem \ref{thm:page-rank-tails}(b). 
\end{proof}
}

\textcolor{black}{If there is only one type ($K=1$), or if $K \ge 2$ and all the entries in each column of the matrix $\kappa$ are the same (all rows identical), it is straightforward to check from \eqref{eqn:vpi-def} and \eqref{eqn:phi-def} that $\phi_a = 1$ for all $a \in [K]$. In this case, $\phi_{a,b} = \pi_b$ for all $a,b \in [K]$ and hence $\lambda_c = 1 + c$ for any $c \in (0,1)$ and $\mvPsi = (1,1,\dots,1)$. The above results then take a particularly simple form.} 

\textcolor{black}{\begin{corollary}
	Consider the case when $K=1$, or when $K \ge 2$ and all the entries in each column of the matrix $\kappa$ are the same (all rows identical). Then, there exist finite constants $0< B_1\le B_2< \infty$ such that for all $a\in [K]$,
	\[B_1 r^{-2/(1+c)} \leq \pr_a(\cR_{\emptyset, c} > r) \leq B_2 r^{-2/(1+c)}\]
	and 
	$\E_a(\cR_{\emptyset, c}) =1.$
\end{corollary}}

\textcolor{black}{The tail exponent of the limiting PageRank distribution in the one type case was previously obtained, among other things, in \cite{banerjee2021pagerank}.}

\subsection{Asymptotics for global functionals}
Much of the previous discussion dealt with ``local'' functionals and showing how local neighborhoods of the $\scrP$ model class can be approximated by the $\scrU$ model class. The goal of this section is to show that similar asymptotics hold even for global functionals such as the maximal degree. The main result shows (in a weak sense) that for any fixed attribute $a \in [K]$, the maximal degree scales like $n^{\phi_a/2}$.

\begin{theorem}\label{thm:max-degree}
	{\cg With attribute space $[K]$}, consider the model $\scrP(1,\mvnu, \kappa)$ satisfying Assumptions \ref{ass:finite-case}. Let ${\cg \fM_n^a}$ denote the maximal degree of vertices of attribute type $a$ while ${\cg \fF_n^{a}}$ denote the degree at time $n$ of the first vertex of type $a$ born into $\set{\cG_n:n\geq 0}$ \textcolor{black}{(setting ${\cg \fF_n^a=0}$ if no type $a$ vertex exists in $\cG_n$)}. Then with $\phi_a$ as in \eqref{eqn:phi-def},  
	\begin{enumeratea}
\item There exists a random $\set{C_n:n\geq 0}$ sequence of normalizing constants and a non-negative \textcolor{black}{non-degenerate} random variable $W_a$ so that ${\cg \fF_n^{a}}/C_n \to W_a~$a.s and in $\bL^2$. Further $C_n \approx n^{\phi_a/2}$ as $n\to \infty$ in the sense that 
\[\log{C_n} - \frac{\phi_a}{2} \log{n} = o_{\pr}(\log{n}).\]
\item For any fixed $\eps >0$, 
\[\frac{{\cg \fM_n^a}}{n^{\frac{(\phi_a+\eps)}{2}}} \probc 0, \qquad \mbox{ as } n\to\infty. \]
	\end{enumeratea}
\end{theorem}

\begin{conj}
There exists a strictly positive finite a.s. random variable $W_a$ such that, 
	\[\frac{{\cg \fM_n^a}}{n^{\phi_a/2}} \convas W_a.  \]
\end{conj}

\section{Main results: non-tree regime}\label{sec:nt}
Although we focussed on tree networks in the previous section for expositional clarity, most of our results can be extended to the analogous non-tree network model defined in Definition \ref{def:attrib-evol}. Again we consider a finite attribute set $[K]$ and attribute pmf $\mvpi$ on $[K]$. We will denote the out-degree function as a vector $\mvm = (m_1, \dots, m_K) \in \mathbb{N}^K$, where $m_a$ denotes the out-degree of a vertex of attribute type $a$. The network sequence will be denoted by $\{\cG_n^{\mvm} : n \ge 0\}$ (with initial base connected graph $\cG_0^{\mvm}$) to highlight dependence on $\mvm$. Suppressing $\gamma$ (which equals one in this case), denote this model of evolving random networks by $\scrP(\mvpi, \kappa, \mvm)$. Degree asypmtotics for this model with $m_a \equiv m$ was obtained in \cite{jordan2013geometric}. 

		

We will now show that our techniques can be used to obtain the (almost sure) local weak limit of the network model $\scrP(\mvpi, \kappa, \mvm)$ in terms of an associated continuous time branching process stopped at a random exponential time.
To describe the local limit of $\scrP(\mvpi, \kappa, \mvm)$, we need the following extension of Lemma \ref{lem:minimizer} and \cite[Proposition 3.1]{jordan2013geometric}. {\cg Recall the discussion below Lemma \ref{lem:minimizer} explaining the origin of the minimization problem in \eqref{eqn:vpi-def} via stochastic approximation schemes related to the evolution of weights of degrees in various attributes in the tree regime. An analysis of similar functionals in the non-tree regime leads to the minimization problem in Lemma \ref{lem:minimizernt}. } As before $n_0 = |\cG_{0}|$ is the number of vertices in the seed graph at time zero. 

\begin{lemma}\label{lem:minimizernt}
Define
	\begin{equation}
\label{eqn:vpi-defnt}
	 V^{\mvm}_{\mvpi}(\vy):= \sum_{i=1}^K y_i -\frac{1}{2}\sum_{j=1}^K m_j\pi_j\left(\log(y_j) + \log(\sum_{k=1}^K y_k \kappa_{k,j})\right), \ \ \vy \in \mathbb{R}^K_+.
\end{equation}
Under Assumption \ref{ass:finite-case}, $V^{\mvm}_{\mvpi}(\cdot)$ has a unique minimizer $\mveta^{\mvm}:= \mveta^{\mvm}(\mvpi) = (\eta^{\mvm}_1(\mvpi), \ldots, \eta^{\mvm}_K(\mvpi))$ that lies in the interior of the set $S^{\mvm} := \{\vy \in \mathbb{R}^K_+: \sum_{i=1}^K y_i = \sum_{i=1}^K \pi_im_i\}$, satisfying $\partial_{a}V^{\mvm}_{\mvpi}(\eta^{\mvm})=0$ for all $a \in [K]$. Moreover, almost surely as $n \rightarrow \infty$,
$$
\tilde{Y}_a^{\mvm,\sss(n)}:= \frac{\sum_{v\in \cG^{\mvm}_n: a(v) = a} \deg(v,n)}{2(n+n_0)} \rightarrow \eta^{\mvm}_a, \qquad a \in [K].
$$
\end{lemma}
Now define, for $a,b \in [K]$,
\begin{equation}
\label{eqn:phi-defnt}
	\nu^{\mvm}_b:= \frac{\pi_b}{\sum_{l=1}^K \kappa_{l,b}\eta^{\mvm}_l},  \qquad \phi^{\mvm}_{a,b}:= \kappa_{a,b}m_b\nu^{\mvm}_b, \qquad \phi^{\mvm}_a := \sum_{b=1}^K \phi^{\mvm}_{a,b} = 2- \frac{m_a\pi_a}{\eta^{\mvm}_a},
\end{equation}
where the last identity follows from the assertion on partial derivatives in Lemma \ref{lem:minimizernt}. 

Next, we define the continuous time branching process that describes the local limit of $\scrP(\mvpi, \kappa, \mvm)$. To ease notation we suppress the dependence on $(\mvpi, \kappa, \mvm)$ in the constructions and functionals below, when there is no scope for confusion. 

\begin{defn}\label{CTBPnt}
Using the quantities defined in \eqref{eqn:phi-defnt}, for any $a \in [K]$,  construct the continuous time, multi-type, branching process $\set{\BP^{\mvm}_{a}(t):t\geq 0}$ started with ancestor (root) of type $a$ where each individual of type $a^\prime$ in the population reproduces independently at times following the law of a Markovian pure birth process $\xi_{a^\prime}^{\mvm}(\cdot)$ with rate of birth of a type $b$ offspring at time $t$ given by
$
\phi^{\mvm}_{a^\prime,b} (\xi_{a'}^{\mvm}(t) + m_{a^\prime}).
$
For (possibly random) $A, T \ge 0$, we will denote by $\cT^{\mvm}_{A}(T)$ the progeny tree of $\BP^{\mvm}_{A}(T)$ rooted at the ancestor, viewed as a directed, marked, rooted graph with each vertex marked with its attribute type.
\end{defn}

Now, we state our main theorem of this Subsection. {\cg Recall the notion of local weak convergence for directed graphs in Definition \ref{def:lwc}.}

\begin{theorem}[Local weak convergence for $\scrP(\mvpi, \kappa, \mvm)$]\label{thm:fringe-convergencent}
	Fix finite attribute set {\cg $[K]$ and $\mvpi, \kappa$} satisfying Assumptions \ref{ass:finite-case}. Consider the sequence of networks  constructed as $\set{\cG^{\mvm}_n}_{n\geq 0} \sim \scrP(\mvpi, \kappa, \mvm)$. Then, 
	\[\cG^{\mvm}_n \, \stackrel{\mbox{$\operatorname{a.s.}$-\bf loc}}{\longrightarrow} \, \cT^{\mvm}_{A}(\tau),\]
 where $A\sim \mvpi$ and $\tau$ is an independent exponential random variable with rate $\lambda =2$. 
\end{theorem}

In particular, although the network itself is not tree-like, its local limit is supported on the space of (directed, marked, rooted) trees.

{\cg\begin{rem}
The local limit for $\scrP(\mvpi, \kappa, \mvm)$ in the $K=1$ case was obtained before in \cite{berger2014asymptotic,garavaglia2022universality,banerjee2023local} using different techniques. The `stopped CTBP' description of the local limit, which appears in the above theorem, was obtained for the $K=1$  (non-tree) case in \cite{banerjee2023local}. The previous works in \cite{berger2014asymptotic,garavaglia2022universality} obtained a description of the local weak limit in terms of the so-called \emph{P\'olya point graph}, by using a P\'olya urn description of the single type linear preferential attachment graph process. Although these works provide a description of the joint local limit of both the in- and out-components in comparison to our directed version for the in-component, the methods in \cite{berger2014asymptotic,garavaglia2022universality} are very specific to the single type linear case and do not extend to more general attachment schemes (like the $K>1$ case studied here). The stochastic approximation method applied here gives a shorter, more direct way to obtain local weak limits that is not rigidly tied to the specific graph process and seems to be generalizable to many other random graph models. Moreover, contrasted with the involved description of the P\'olya point graph through point processes with non-linear intensities given in terms of Gamma random variables (see \cite[Section 2.3]{berger2014asymptotic}), the stopped CTBP description of the local limit given here is more amenable to quantifying refined asymptotics of functionals like PageRank. Lastly, due to the locally tree-like nature of these models, we believe that the directed local weak limit can be naturally upgraded to the joint local limit of the in- and out-components, using ideas from \cite{aldous-fringe}, in a similar fashion as the tree case ($m_a\equiv 1$) discussed above. We leave this for future work.
\end{rem}}

As applications of the above theorem, we can derive the following asymptotics for the empirical degree and PageRank distributions. The following result gives degree asymptotics. Below, for $a \in [K]$, $\vp_n^{\mvm,a} = \set{\vp_n^{\mvm,a}(k) : k \ge m_a}$, where $\vp_n^{\mvm,a}(k)$ denotes the proportion of vertices with degree $k$. 

\begin{theorem}\label{thm:deg-nt}
	Assume the model $\scrP(\mvpi,\kappa,\mvm)$ satisfies Assumptions \ref{ass:finite-case}.
For each $a\in [K]$, $\vp_n^{\mvm,a} \convas \vp_\infty^{\mvm,a}$ where 
	\begin{equation}
		\label{eqn:pinf-def}
		\vp^{\mvm,a}_\infty(k) = \frac{2}{\phi^{\mvm}_a}\frac{\Gamma\left(m_a+\frac{2}{\phi^{\mvm}_a}\right) \Gamma(k)}{\Gamma\left(k+1+\frac{2}{\phi^{\mvm}_a}\right)\Gamma(m_a)}, \qquad k\geq m_a.  
	\end{equation}
	In particular $\vp_\infty^{\mvm,a}(k)\sim k^{-(1+2/\phi^{\mvm}_a)}$ as $k\to\infty$.	  
\end{theorem}

One question of interest in applications are mechanisms where minority vertices can increase their degree centrality propensity, resulting in heavier degree distribution exponents, and one proposal is increasing the number of incoming edges vis-\'a-vis majority nodes \cite{karimi2018homophily}. The above result gives if and only if conditions to check this in terms of the driving parameters of the model. We first need some notation. Suppose we have a two attribute space $\cS = {\cg [2]}$ with relative densities $\pi_1 = \pi, \pi_2 = 1-\pi$ with $\pi<1/2$ and referred to as the minority. Let $\mvm = (m_1, m_2)$ and define $M = \pi m_1 + (1-\pi) m_2$. Define the function on $0 < y < M$:
\begin{align}
V_{\mvpi}(y):= M - \frac{1}{2}\bigg[ &\big(m_1\pi \big(\log(y) + \log(y\kappa_{1,2} + (M-y)\kappa_{2,1})\big)\big) \notag\\
 &+ \big(m_2(1-\pi) \big(\log(M - y) + \log(y\kappa_{2,1} + (M-y)\kappa_{2,2})\big)\big)\bigg]. \label{eqn:two-type}
\end{align}
\begin{corollary}
\label{corr:two-type}
The function in \eqref{eqn:two-type} has a unique minimizer $\eta^{\mvm} \in  (0,M)$. The asymptotic degree distribution of type $1$ vertices will have a heavier tail than type $2$ vertices if and only if $m_1 \pi (M - \eta^{\mvm}) < m_2(1-\pi)\eta^{\mvm}$. 
\end{corollary}
\begin{rem}
Using a fluid limit approach and under the assumption that various functionals of the models converge (i.e. $\sum_{v\in \cG_n, a(v) = a} \deg(v,n) \approx C n $ and $\sum_{v\in \cG_n, a(v) = b} \deg(v,n) \approx (M - C) n $ for large $n$), \cite[Supplementary information]{karimi2018homophily} derived conditions on the parameters of the model involving solving a cubic equation in $C$. The above gives a different but equivalent characterization for the same goal. We delve further on the impact of the parameters on various degree exponents and their role in network sampling in the next section. 
\end{rem}

The next result quantifies PageRank asymptotics. In the following, as before, $\pr_a$ and $\E_a$ will respectively denote the law of the stopped branching process $\BP^{\mvm}_{a}(\tau)$, and corresponding expectation, with root having attribute type $a$. For $v \in \cG_n^{\mvm}$ and damping factor $c \in (0,1)$, the PageRank $\fR^{\mvm}_{v,c}(n)$ of vertex $v$ is defined as in Definition \ref{def:page-rank}, and $R^{\mvm}_{v,c}(n) := n \fR^{\mvm}_{v,c}(n)$ is the corresponding graph normalized version.

\begin{theorem}\label{thm:page-rank-tailsnt}
 Suppose $\set{\cG^{\mvm}_n:n\geq 0}\sim \scrP(\mvpi,\kappa, \mvm)$  satisfying Assumption \ref{ass:finite-case}. Fix $a\in [K]$ and define the random variable $\cR^{\mvm}_{\emptyset, c}$ as
 $$
 \cR^{\mvm}_{\emptyset, c} := (1-c)\left(1 + \sum_{l=1}^\infty c^l \sum_{\vj \in \cP^{\mvm}_{l,\emptyset}}\prod_{h=0}^{l-1}\frac{1}{m_{a(j_h)}}\right)
 $$
 where $\cP^{\mvm}_{l,\emptyset}$ is the set of directed paths of length $l$ in $\cT^{\mvm}_{A}(\tau)$ ending at the root, and for $\vj=(j_0, \dots, j_l) \in \cP^{\mvm}_{l,\emptyset}$, $m_{a(j_h)}$ is the out-degree of the $h$-th vertex on the path.
Then, for every continuity point $r$ of the distribution of $\cR^{\mvm}_{\emptyset, c}$ under $\pr_a$, 
    \[n^{-1}\sum_{v\in \cG^{\mvm}_n} \ind\set{a(v) = a, R^{\mvm}_{v,c}(n) > r} \convas \pi_{a}\pr_a(\cR^{\mvm}_{\emptyset, c} > r).  \]

	Let $\lambda^{\mvm}_c$ denote the Perron-Frobenius eigen-value of the matrix $\vM^{\mvm,\sss(c)}$ 
	\begin{equation}
\label{eqn:matrix-mc-defnt}
	\vM^{\mvm,\sss(c)} = \left(\vM_{(a,b)}^{\mvm,\sss(c)}:= c\phi^{\mvm}_{a,b}\frac{m_a}{m_b} + \phi^{\mvm}_a\ind\set{a=b} \right)_{a, b\in [K]}.
\end{equation}

	 Then:
	\begin{enumeratea}
		\item There exist finite constants $0< B_1\le B_2< \infty$ such that for all $a\in [K]$, the limiting PageRank distribution of vertices of type $a$, i.e. the distribution of $\cR^{\mvm}_{\emptyset, c}$ under $\pr_a$, satisfies:
		\[B_1 r^{-2/\lambda^{\mvm}_c} \leq \pr_a(\cR^{\mvm}_{\emptyset, c} > r) \leq B_2 r^{-2/\lambda^{\mvm}_c}.\]
		\item Further for any $a\in [K]$ and $n\geq 1$ defining:
		\[\scI_n^{\sss(a,b)} = \set{\vj:= (j_0, j_1, \ldots, j_n) \in [K]^n: j_0 =a, \, j_n =b},\]
we have
		\begin{equation}
		\label{eqn:expec-page-ranknt}
			\E_a(\cR^{\mvm}_{\emptyset, c})  = (1-c)\left[1+ \sum_{b \in [K]} \frac{m_a}{m_b}\sum_{n=1}^\infty c^n \sum_{\vj\in \scI_n^{\sss(a)}} \prod_{l=0}^{n-1} \left(\frac{\phi^{\mvm}_{j_j, j_{l+1}}}{2-\phi^{\mvm}_{j_l}}\right) \right].
		\end{equation}
		\end{enumeratea}
	 
\end{theorem}

\begin{rem} 
	
\begin{enumeratei}
\item If there is only one type ($K=1$), or if $K \ge 2$ and all rows of the matrix $\kappa$ are identical, it is straightforward to check using $\partial_{a}V^{\mvm}_{\mvpi}(\eta^{\mvm})=0$ for all $a \in [K]$ and \eqref{eqn:phi-defnt} that $\phi^{\mvm}_{a,b} = \pi_b m_b/(\sum_{l \in [K]}\pi_lm_l)$ and $\phi^{\mvm}_a = 1$ for all $a,b \in [K]$. Thus, in this case, $(m_1,\dots,m_K)$ is a corresponding right eigenvector for the matrix $\vM^{\mvm,\sss{(c)}}$ and $\lambda^{\mvm}_c = 1 + c$ for any $c \in (0,1)$. This implies that the limiting PageRank tail exponent in Theorem \ref{thm:page-rank-tailsnt} is $2/(1+c)$ and the tail exponent of the limiting degree distribution obtained in Theorem \ref{thm:deg-nt} is $2$. In particular, these exponents are \emph{independent of the out-degree vector $\mvm$}. For $K=1$, these exponents also follow from \cite{banerjee2021pagerank}. However, as seen in \cite{banerjee2023degree} for $K=1$, the out-degree significantly influences the degree separation between the `hubs' (maximal degree vertices) and the remaining vertices. This leads to the ansatz that, although the degree tail exponents are the same across attributes in this case, increasing the out-degree of a given type will lead to the maximal degree vertex coming from the same type with high probability. Verifying this is deferred to future work. 
\item Suppose $m_a \equiv m \ge 1$ for all $a \in [K]$. Denote corresponding quantities in the tree case ($m=1$ but same $\mvpi, \kappa(\cdot,\cdot)$) by dropping the $\mvm$ in the superscript. It follows from the form of $V_{\mvpi}^{\mvm}(\cdot)$ in \eqref{eqn:vpi-defnt} that $\eta_a^{\mvm} = m \eta_a$ for all $a \in [K]$. This implies that $\phi^{\mvm}_{a,b} = \phi_{a,b}$ and $\phi^{\mvm}_a = \phi_a$ for all $a,b \in [K]$. Consequently, the tail exponents for the limiting PageRank distribution, as well as the limiting degree distribution, match in the tree and non-tree cases.
\end{enumeratei}
\end{rem}

\section{Main results: Uniform attachment}
\label{sec:exten-res}
The goal of this section is to extend our results to the $\gamma=0$ (uniform attachment) case where incoming vertices attach to pre-existing vertices based purely on their type (and are agnostic to degree information), thus showing that this class of models also exhibit the phenomenon of resolvability where local geometry can be approximated by carefully chosen continuous time branching processes. For simplicity, we discuss only the tree case $\mvm \equiv 1$, although results here can be extended to the non-tree setting as for the $\gamma=1$ case.

Analogous to  \eqref{eqn:phi-def} define,
\begin{equation}
	\label{eqn:zero-phi}
	\chi_b = \frac{\pi_b}{\sum_{a^\prime} \pi_{a^\prime} \kappa(a^\prime, b) }, \qquad \varphi_{a,b} = \kappa(a,b) \chi_b, \qquad \varphi_a = \sum_{b\in [K]} \phi_{a,b}. 
\end{equation}
The following is easy to check from construction. 
\begin{lemma}\label{lem:gamma-mat-def}
	Consider the functionals in \eqref{eqn:zero-phi} and define the matrix $\vN = (\varphi_{a,b})_{a,b\in [K]}$. Then, 
	\begin{enumeratea}
		\item $\sum_{a \in [K]} \pi_a \varphi_a =1$. 
		\item The matrix $\vN$ {\cg has} Perron-Frobenius eigen-value $\lambda_{\PF}(\vN) =1$ with corresponding left eigen-vector $\mvpi$ namely the original attribute distribution. 
	\end{enumeratea}
\end{lemma}

\begin{defn}
	\label{def:model-u-corr-p-gamma-0}
	Fix $\mvpi, \kappa$ satisfying Assumptions \ref{ass:finite-case} and let $\gamma = 0$.  Construct $\mvvarphi = (\varphi_a:a\in [K])$ as in \eqref{eqn:zero-phi}. Using $\mvvarphi, \kappa$,  consider the construction of the model class $\scrU(0,\mvvarphi, \kappa)$ started with root of an arbitrary type $a^\prime \in [K]$, constructed using the branching process $\set{\BP_{a^\prime;(0,\mvvarphi, \kappa)}(t):t\geq 0}$ as in Definition \ref{defn:ctbp}. 
\end{defn}

\begin{theorem}[Local weak convergence for uniform model]\label
	Fix finite attribute space {\cg $[K]$}. Let $\gamma = 0$ and fix $\mvpi, \kappa$ satisfying Assumptions \ref{ass:finite-case} and consider the sequence of networks  constructed as $\set{\cG_n}_{n\geq 0} \sim\scrP(0, \mvpi, \kappa)$. Then, as in Definition \ref{def:local-weak} \eqref{it:fringe-b}, 
	\[\cG_n \probcrf \cT_\infty,\]
	where $\cT_\infty$ is the unique {\cg {\tt sin}}-tree with fringe distribution 
	\begin{equation}
	\label{eqn:fringe-limit0}
		\fpm_{\infty, 0}(\cdot) = \cL(\BP_{A;(0,\mvvarphi, \kappa)}(\tau^\prime)),
	\end{equation}
where $\set{\BP_{A;(0,\mvvarphi, \kappa)}(t):t\geq 0}$, is a branching process as in Definition \ref{def:model-u-corr-p-gamma-0} where $A\sim \mvpi$ and $\tau^\prime$ is an independent exponential random variable with rate $\lambda =1$.  This implies the following.
\begin{enumeratea}
	\item For each fixed $a\in [K]$, \textcolor{black}{define $\tilde{\varphi}_a = 1/(1+\varphi_a)$}.  Then the empirical distribution of degrees of vertices of attribute type $a$ satisfies $\vp_n^a \probc \pr(\Geom(\tilde{\varphi}_a) = \cdot)$, namely the limit distribution is Geometric with success probability $\tilde{\varphi}_a$. 
	\item Analogous to \eqref{eqn:t-power}, for any damping factor $c\in (0,1)$, one can construct a limit PageRank random variable $\cR_{\emptyset, c}$. For every continuity point $r$ of the distribution of $\cR_{\emptyset, c}$ under $\pr_a$, 
    \[n^{-1}\sum_{v\in \cG_n} \ind\set{a(v) = a, R_{v,c}(n) > r} \probc \pi_{a}\pr_a(\cR_{\emptyset, c} > r).  \]
	There exist finite constants $0< B_1< B_2< \infty$ such that for all $a\in [K]$, the limit PageRank distribution of vertices of type $a$, i.e. the distribution of $\cR_{\emptyset, c}$ under $\pr_a$ satisfies:
			\[B_1 r^{-1/c} \leq \pr_a(\cR_{\emptyset, c} > r) \leq B_2 r^{-1/c}.\]
\end{enumeratea}
\end{theorem}

\begin{rem}\ 
	
\begin{enumeratei} 
\item The proof of this result follows in an identical fashion to the results in Section \ref{sec:main-res} via suitable modification and is omitted as it has no new ideas. 
\item The degree distribution result in a simpler (two attribute) setting was previously derived in \cite{bhamidi2022community}.
\item Note that the above result implies that the degree distribution (consisting of appropriate mixtures of Geometric distributions) has exponential tails; however the limit PageRank distribution has a power law (explicitly computable) tail exponent which might be at first sight surprising; see \cite{banerjee2021pagerank} for the genesis of such results for dynamic network models. {\cg Intuitively, this can be understood by noting that a high PageRank value of a vertex results from the vertex having either a high in-degree or having a child with a high PageRank score (see \cite[Page 5-6]{banerjee2021pagerank}). For dynamic graphs discussed here, older vertices tend to have higher in-degrees and are typically close to other high degree (and high PageRank) vertices. This \emph{reinforcement} results in the PageRank having heavier tails than degree. Mathematically, this results from the interpretation of the PageRank of a vertex as its progeny size under a percolated branching process (see Definition \ref{def:perc-bp} below). Via the local weak convergence result, this implies that the limiting PageRank $\cR_{\emptyset, c}$ behaves like $e^{c\tau^{\prime}}$ (approximate size of root progeny in $\BP_{A;(0,\mvvarphi, \kappa)}(\tau^\prime)$), which has the stated tail behavior.}
\end{enumeratei}
\end{rem}

\section{Applications to network sampling: rare minorities, relative ranking and sampling bias}\label{fairsec}

The previous sections described general asymptotic results for attribute driven models. The goal of this section is to illustrate ramifications and insight provided by the theory developed above in concrete cases of interest in applications. Given that most social networks can only be indirectly observed, network sampling, and its impact on inferential pipelines of the true network based on the sample, is a significant research endeavor across multiple communities,  see for example the surveys \cite{gile2018methods,crawford2018identification} and the references therein to pointers to this vast field. Here we explore specific questions in this vast field, motivated by the attributed network context, namely (a) settings where PageRank based and other exploration based sampling schemes are able to sample rare minorities, (b) effects of homophily and out-degrees on relative ranking of minorities and, (c) the insight provided regarding inferring properties of the underlying network through sampled portions of the graph.

\subsection{Sampling mechanisms and attribute representation: } For clarity in exposition, in this Subsection and the following one, we restrict ourselves to the tree case. We consider the following major sampling schemes: 

\begin{enumeratea}
	\item {\bf Uniform node sampling $(\uNS)$:} Here one picks a vertex uniformly at random from $\cG_n$. 
	\item {\bf Sampling proportional to degree $(\dNS)$:} Pick a vertex uniformly at random and then pick a neighbor of this vertex uniformly at random. 
	\item {\bf Sampling proportional to in-degree $(\indNS)$:} Pick a vertex at random and then select the parent; by convention, if the root is picked (which happens with probability $o_{\pr}(1)$ as $n\to\infty$) then select the root. {\cg Note that, other than if the root is a leaf, this sampling scheme results in only non-leaf vertices being sampled. }
	\item {\bf Sampling proportional to PageRank $(\prNS)$: } Fix a damping factor $c$ and sample a vertex with probability proportional to the PageRank scores $\set{\fR_{v,c}: v\in \cG_n}$ as defined in Section \ref{sec:func-interest}.  In the context of the (tree) network model $\set{\cG_n:n\geq 1}$ starting with a single root at time zero,  by the proof of \cite[Theorem 1.1]{chebolu2008pagerank}, this can be accomplished by the following ``local'' algorithm:
	\begin{enumeratei}
		\item Pick a vertex $V$ uniformly at random from $\cG_n$. 
		\item Independently let $G\sim \Geom(1-c)-1$ (here $\Geom(\cdot)$ is a Geometric random variable with prescribed parameter with support starting at one). 
		\item Starting from $V$ traverse $G$ steps towards the root (i.e. using the directions of edges in $\cG_n$ from child to parent), stopping at the root,  if the root is reached before $G$ steps. Sample the terminal vertex. 
	\end{enumeratei}
	\item {\bf Fixed length sampling $(\prfNS)$:} Fix $M\geq 0$. Consider the same implementation of the PageRank scheme but here the halting distribution is taken to be $G \equiv M$. Abusing notation, we use $\prfNS$ to denote this sampling scheme. 
\end{enumeratea}

We start with a general theorem that quantifies the asymptotic sampling probabilities of attribute types {\cg via the above sampling schemes. In the next Section, using the setting of sampling rare minorities, we show how one can use these explicit formulae to gain insight into specific applications}. Recall the functionals defined in \eqref{eqn:phi-def} and the random walk $\vS$ in Proposition \ref{prop:mc-descp}. 

\begin{theorem}\label{thm:network-sampling}
	Let $\set{\cG_n:n\geq 1} \sim \scrP(1,\mvpi,\kappa)$ satisfying Assumption \ref{ass:finite-case}. Let ${\cg U_n}$ be a random node sampled from $\cG_n$, using one of the sampling schemes above and let $a({\cg U_n})$ be the corresponding attribute. Then $\forall~ b\in [K]$, as $n\to\infty$,
	\begin{enumeratea}
		\item\label{it-uNS} Under uniform sampling $\pr_{\uNS}(a({\cg U_n}) =b|\cG_n)\convas \pi_b $.
		\item\label{it-uDNS} Under sampling proportional to degree $\pr_{\dNS}(a({\cg U_n}) =b|\cG_n)\convas \eta_b$.
		\item Under sampling proportional to in-degree, 
		\[\pr_{\indNS}(a({\cg U_n}) =b|\cG_n)\convas \eta_b \phi_b = \pi_b \frac{\phi_b}{2-\phi_b} = \pi_b \Psi_b \E^{\vS}_b\left[\frac{1}{\Psi_{S_1}}\right]. \]
		\item Under sampling proportional to PageRank, letting $G\sim \Geom(1-c)-1$ independent of $\vS$, 
		\[\pr_{\prNS}(a({\cg U_n}) =b|\cG_n)\convas \pi_b \Psi_b \E^{\vS}_b\left[\frac{1}{\Psi_{S_G}}\right]. \]
		Since $\vS$ has stationary distribution $\set{\pi_a \Psi_a: a\in [K]}$, 
		\[\lim_{c\uparrow 1} \lim_{n\to\infty} \pr_{\prNS}(a({\cg U_n}) =b|\cG_n)\stackrel{\mathrm{a.s.}}{=} \pi_b \Psi_b.  \]
		\item Under fixed length walk sampling, 
		\[\pr_{\prfNS}(a({\cg U_n}) =b|\cG_n)\convas \pi_b \Psi_b \E^{\vS}_b\left[\frac{1}{\Psi_{S_M}}\right].\]
		In particular, 
		\[\lim_{M\uparrow \infty} \lim_{n\to\infty} \pr_{\prfNS}(a({\cg U_n}) =b|\cG_n)\stackrel{\mathrm{a.s.}}{=} \pi_b \Psi_b.  \]
	\end{enumeratea}  
\end{theorem} 
\begin{rem}
	\label{rem:904}
	Part \eqref{it-uNS} follows directly from the dynamics of construction of $\scrP$. Part \eqref{it-uDNS} follows from \cite[Proposition 3.1]{jordan2013geometric}. 
\end{rem}
\subsection{Network sampling: Implications for rare minority sampling }
An area of signficant research interest in the context of network sampling comprises settings where there is a particular rare minority which has higher propensity to connect within itself as opposed to majority vertices; for substantial recent applications and impact of such questions, see \cite{mouw2012network,merli2016sampling,stolte2022impact}. In such settings, devising schemes where one gets a non-trivial representation of minorities is challenging if the sample size is much smaller than the network size. Consider the case $K=2$ (two attributes) and $\pi_1 \ll \pi_2$ {\cg so that type 1 vertices can be interpreted as rare minorities in the population}. In this case, uniform sampling will clearly not be fair as the sampled vertices will tend to be more often from the attribute $2$ class. Therefore, it is desirable to \emph{explore} the network locally around the initial (uniformly sampled) random vertex. As uniformly sampled vertices lie mostly in the \emph{fringe} (edge) of the network, the exploration should try to travel towards the `centre', thereby traversing edges along their natural direction. 

The goal is to seek a tradeoff between two competing interests: to avoid high sampling costs, the explored set of vertices should not be too large; however, to ensure non-trivial representation, this set should not be too small. This leads us to analyze the sampling schemes discussed above in this context.

When the $\kappa_{i,j}$ are all comparable in magnitude, any exploration started from a type $2$ vertex hits a type $1$ vertex reasonably quickly. Thus, sampling schemes (b)-(e) will all perform reasonably well, much better than (a).
When $\kappa_{1,2}$ is very small, incoming type $2$ vertices more likely connect to the same type. Therefore, any exploration started from a type $2$ vertex and traversing outbound edges has to spend a while before hitting a type $1$ vertex.  Hence, any sampling scheme with the goal of having a non-trivial representation of the rare minority vertices will necessarily have a high exploration cost. {\cg Such questions lead} to the following specific case of model class $\scrP$ with two attributes $1,2$ with, 
\begin{equation}
\label{eqn:rare-model}
	\mvkappa = (\kappa(i,j))_{1\leq i,j\leq 2} = \begin{pmatrix}
		1 & 1 \\
		a & 1 
	\end{pmatrix},
	\qquad \mvpi = \frac{1}{1+\theta} (\theta, 1), \qquad \theta \ll1.
\end{equation}
Thus, 
\begin{enumeratei}
	\item Type $1$ vertices are relatively \emph{rare} compared to type $2$ vertices; we will often refer to type $1$ vertices as minorities and type $2$ as majorities. 
	\item Newly entering majority vertices into the population have equal propensity to connect to minority or majority vertices. Minorities have (relatively) {\bf much higher} propensity to connect to other minority vertices, as compared to majority vertices. 
\end{enumeratei}
We are interested in exhibiting a scenario where the \emph{uniform and degree-based sampling schemes are not effecient in sampling from such rare minorities but the PageRank and fixed length walk sampling schemes are} (for large network size and small $a, \theta$). This requires letting $\theta$ and $a$ go to zero in a dependent way.
We achieve this by analyzing the setting
\begin{equation}
\label{eqn:thet-scaling}
	\theta := \theta(a) = D\sqrt{a},
\end{equation}
where $D>0$ is a fixed constant and where $a\downarrow 0$. The following Theorem summarizes our findings. {\cg The proof entails understanding properties of the explicit formulae in Theorem \ref{thm:network-sampling} in this specific context. } 
\begin{theorem}\label{thm:rare-sampling}
	Consider the model $\set{\cG_n:n\geq 1}\sim \scrP(1, \mvpi, \kappa)$ with {\cg two type attribute space $[2]$} and choices of propensity and attribute proportions satisfying \eqref{eqn:rare-model} with the scaling \eqref{eqn:thet-scaling} for fixed $D>0$. Consider the network sampling schemes in the setting of Theorem \ref{thm:network-sampling}. Then as $a\downarrow 0$:
	\begin{enumeratea}
		\item Under uniform node sampling, 
		\[\pr_{\uNS}(a({\cg U_n}) = 1|\cG_n) \convas D\sqrt{a} + O(a). \]
		\item For sampling proportional to degree,
		\[\pr_{\dNS}(a({\cg U_n}) =1|\cG_n) \convas 2D\sqrt{a} - (4D^2 +\frac{1}{2})a + O(a^{3/2}). \]
		\item For random in-degree based sampling, 
		\[\pr_{\indNS}(a({\cg U_n}) =1|\cG_n) \convas 3D\sqrt{a} +O(a). \]
		\item For PageRank based sampling and fixed length walk sampling:
		\begin{align*}
			\lim_{c\uparrow 1} \lim_{n\to\infty} \pr_{\prNS}(a({\cg U_n}) =1|\cG_n) \stackrel{\mathrm{a.s.}}{=} &\frac{2D^2 -\frac{1}{2} + \sqrt{\left((2D^2-\frac{1}{2})^2 + 4D^2\right)}}{2D^2 +\frac{1}{2} + \sqrt{\left((2D^2-\frac{1}{2})^2 + 4D^2\right)}} + O(\sqrt{a})\\
			& \stackrel{\mathrm{a.s.}}{=}\lim_{M\uparrow \infty} \lim_{n\to\infty} \pr_{\prfNS}(a({\cg U_n}) =1|\cG_n).
		\end{align*} 
	\end{enumeratea} 
\end{theorem}

\begin{rem}\label{mixrem}
\begin{enumeratei}
\item Although Theorem \ref{thm:rare-sampling} states the results in (d) for $c \uparrow 1$ and $M \uparrow \infty$, we believe the exact probabilities are close to the limiting values on taking $M = 1/(1-c) = M^*/\sqrt{a}$ for large enough $M^*$, and $n$ large so that the diameter of the network $\cG_n$ (which in the tree case should be $\Theta(\log n)$) is sufficiently greater than $2M^*/\sqrt{a}$.
The calculations in the proof of Theorem \ref{thm:rare-sampling} show that $\Psi_1 = \Theta(1/\sqrt{a})$, $\Psi_{2} = \Theta(1)$, and the transition probabilities of the random walk $\vS$ in Proposition \ref{prop:mc-descp} satisfy $\pr^\vS_i(S_1 =j) = \Theta(\sqrt{a})$ when $i \neq j$ and $\pr^\vS_i(S_1 =j) =\Theta(1)$ when $i=j$. Thus, the probability of sampling a type $1$ vertex starting from a type $2$ vertex in one step of $\vS$ is $\Theta(\sqrt{a})$ leading to claims (b) and (c) above. However, the random walk is seen to `mix' in $M^*/\sqrt{a}$ steps (for sufficiently large $M^*$) and hence, if the network is large enough so that $\vS$ started from a uniformly sampled vertex does not get near the root in these many steps, the minority sampling probabilities stabilize around their limiting value $\pi_1\Psi_1 = \Theta(1)$.
\item The function $D \mapsto \frac{2D^2 -\frac{1}{2} + \sqrt{\left((2D^2-\frac{1}{2})^2 + 4D^2\right)}}{2D^2 +\frac{1}{2} + \sqrt{\left((2D^2-\frac{1}{2})^2 + 4D^2\right)}}$ is monotone and approaches $0$ as $D \rightarrow 0$ and $1$ as $D \rightarrow \infty$. It takes value $1/2$ (`equal representation') at $D=1/2$.
\end{enumeratei}
\end{rem}

\color{black}
\section{Discussion}
\label{sec:disc}

\subsection{Related work}
{\cg Section \ref{sec:motivation-overview} describes the tip of the iceberg in terms of research in the network science community related to attributed network data. Here we will describe related work in the probability community. For general overviews of probabilistic approaches to network models (albeit not covering the specific class of models in this paper), see \cite{durrett2006random,van2016random,van2023random}. For static inhomogeneous random graphs where a fixed vertex set of individuals form connections according to underlying latent types,  see \cite{bollobas2007phase,lovasz2012large}. A crucial proof technique in this paper is stochastic approximations; particularly relevant is the survey \cite{pemantle2007survey}. 
For the use of PageRank in identifying community structure in Stochastic Block Models, see \cite{avrachenkov2018mean,banerjee2023pagerank}. }

\subsection{Network sampling and bias}
{\cg
Given that the majority of social networks can only be indirectly observed, network sampling mechanisms have a long and extensive history across multiple disciplines \cite{gile2018methods,10.1145/1150402.1150479,10.1145/2601438,kolaczyk_2017,10.1145/3442202}. Rigorous understanding of network sampling mechanisms for the model considered in this paper has been intractable and was one of our motivations for proving the abstract local weak limit theory for these models. In the network science community, 
the sensitivity of various centrality measures to sampling schemes has been studied via numerics, see \cite{borgatti2005centrality,costenbader2003stability,kossinets2006effects,smith2017network,wang2012measurement}. Specific to this paper related to attributed networks, \cite{karimi2022minorities} gives a nice recent survey. 
The use of the discussed model in the context of gaining insight for various questions related to bias in social networks was initiated in \cite{Wagner:2017,Karimi:2018} using fluid limits.
An in-depth numerical analysis of this model for various sampling schemes
was carried out in \cite{Wagner:2017}. 
In \cite{espin2018towards}, the main goal was to predict attribute information  for unlabelled vertices, where partial information is provided through a sampled set of vertices.
 In \cite{ferrara2022link}, these questions were taken one step further to understand the impact of recommender systems on the structure of the network (for example, reinforcing popularity bias). 
}

\color{black}

\subsection{Future work}\label{futw}
We started the paper by describing the general setting for the model with preferential attachment parameter $\gamma \in [0,1]$ but then, to keep this paper to manageable length, and to tackle the main models studied till date in the literature, we specialized to the setting of $\gamma \in \set{0,1}$, describing the ramifications of resolvability that allows one to relate the $\scrP$ model to a corresponding $\scrU$ model with explicit description of the parameters for the corresponding $\scrU$. In work in progress we are in the process of extending this work to (a) proving resolvability of the  general sublinear $\gamma \in (0,1)$ regime and (b) identifying the local limits in the more general attribute type space $\cS$ consisting of a compact metric space, completing the picture initiated by the upper and lower bounds on the degree distribution obtained in \cite[Theorem 2.3]{jordan2013geometric}.

\subsection{Organization of the proofs}
\label{sec:proofs-main}

{\cg 
 We start in earnest in Section \ref{sec:proofs-local} with the proof of local weak convergence in the tree regime namely Theorem \ref{thm:fringe-convergence}. Section \ref{sec:proofs-page-rank} contains all the proofs related to PageRank, namely Theorems \ref{thm:page-rank} and \ref{thm:page-rank-tails} as well as Prop. \ref{prop:lc-bound} and \ref{prop:mc-descp}. Section \ref{sec:proof-max} contains the proof of the maximal degree asymptotics, Theorem \ref{thm:max-degree}. Section \ref{sec:proofs-extensions} contains proofs of extensions such as the non-tree regime. We conclude in Section \ref{sec:net-sampling} with proofs related to sampling including Theorem \ref{thm:network-sampling} and \ref{thm:rare-sampling}.  }



%

\section{Proofs: Local weak convergence}
\label{sec:proofs-local}
The main goal of this section is to prove Theorem \ref{thm:fringe-convergence}. 
\textcolor{black}{Recall the two definitions of local weak convergence, in particular Definition \ref{def:local-weak}(a). {\cg Recall the asserted limit $\BP_{A;(1,\mvnu, \kappa)}(\tau)$ in Theorem \ref{thm:fringe-convergence}, namely the multi-type branching process {\bf stopped} at a random finite time $\tau\sim \Exp(2)$.} Observe that the expected number of children of the ancestor in $\BP_{A;(1,\mvnu, \kappa)}(\tau)$ is
$
\sum_{a,b \in [K]} \pi_a \vM_{a,b} =1
$
using Proposition \ref{prop:left-eigen}. Hence, using \cite[Proposition 3]{aldous-fringe} (adapted to our attributed network setting), we conclude that $\cL(\BP_{A;(1,\mvnu, \kappa)}(\tau))$ is a fringe distribution in the sense of Definition \ref{fringedef}. Thus, by Lemma \ref{ftoeflemma}, the following is enough to complete the proof of Theorem \ref{thm:fringe-convergence}.}

 \begin{theorem}
 	\label{thm:just-fringe-convg}
	Let $\set{\cG_n}_{n\geq 0} \sim\scrP(1, \mvpi, \kappa)$ as in the setting of Theorem \ref{thm:fringe-convergence}. Then 
	\[\cG_n \probfr \cL(\BP_{A;(1,\mvnu, \kappa)}(\tau)), \qquad \mbox{ as } n\to\infty, \]
	where $\set{\BP_{A;(1,\mvnu, \kappa)}(t):t\geq 0}$, is a Branching process as in Definition \ref{def:model-u-corr-p} where $A\sim \mvpi$ and $\tau$ is an independent exponential random variable with rate $\lambda =2$.
 \end{theorem}

\begin{proof}
{\cg Before diving into proof, let us provide an outline. There are two major components to the proof:
\begin{enumeratea}
\item {\bf Properties of the asserted limit:} First we show in Proposition \ref{prop:prop-of-fp} that the asserted limit, namely the probability distribution $\cL(\BP_{A;(1,\mvnu, \kappa)}(\tau))$ has an explicit formula given by \eqref{eqn:pm-def} and is the {\bf unique distribution} on the space of attributed trees admitting the specific recursive description in Proposition \ref{prop:prop-of-fp}(b). 
\item {\bf Stochastic approximation for fringe distributions of $\scrP$:} Next, we study the evolution of the fringe distributions of $\scrP$, culminating in \eqref{eqn:up-up-away}. Using techniques from stochastic approximation (in particular Lemma \ref{lem:jordan} proven in \cite{jordan2013geometric}) we show that these evolution equations imply that the fringe distributions converge almost surely to a limit that satisfies the recursive equation in  Proposition \ref{prop:prop-of-fp}(b). By uniqueness established in the first step, this completes the proof. 
\end{enumeratea}
}
{\cg Let us now commence the proof. }	We start with some terminology following \cite{rudas2007random}.  Let $\overrightarrow{\bbT_{\cS}}$ be the space $\bbT_{\cS}$ but where for every element in $\vt$:
	\begin{enumeratea}
		\item The vertices of $\vt$ are labelled using the Ulam-Harris set $\cN = \cup_{n=0}^\infty \bN^n$ with $\bN^0 = \set{\emptyset}$ denoting the root of $\vt$. This in particular gives the birth order of the children of each vertex. {\cg Thus for example label $1\in \bN$ represents the first (oldest) child of the root, $13\in \bN^2$ represents the third oldest child of the first offspring of the root etc.   }
		\item Every edge has a direction from child to parent. 
	\end{enumeratea}
Write $\overrightarrow{\bbT}$ for the space of such trees where we ignore attribute information.   Note that the topology of any tree in $\overrightarrow{\bbT}$ is uniquely determined by its list of vertex labels. 

For $\vt\in \overrightarrow{\bbT}$, let $\cH(\vt)$ denote the collection of all historical orderings of $\vt$. In brief, $\sch_{\vt} = (v_0 = \emptyset, v_1,\ldots, v_{|\vt|-1}) \subseteq \cN^{|\vt|} $ is a historical ordering of $\vt$ if it gives a possible birth ordering of vertices in $\vt$ starting from $\emptyset$ till its completion; formally for each $0\leq i\leq |\vt| -1$, $\cT(\sch_{\vt},i) := \set{v_0, v_1, \ldots, v_i} \in \overrightarrow{\bbT}$. 

\textcolor{black}{Next, recall the functionals $\phi_{\cdot, \cdot}$ and $\phi_{\cdot}$ defined in \eqref{eqn:phi-def}. Given a fixed tree $\cT \in \overrightarrow{\bbT_{\cS}}$, for $0\leq i\leq |\cT|-1 $, let $\rho_i =a({v_i})$ denote the attribute of the $i$-th vertex and $\deg({v_i},\cT)$ denote the degree of ${v_i}$.    Define the weight of $\cT$ as:
\begin{equation}\vW(\cT) := \sum_{i=0}^{|\cT| -1} \phi_{\rho_i} \deg(v_i, \cT). \label{eqn:haalp2}\end{equation}}
 Next given  $\vt \in \overrightarrow{\bbT_{\cS}}$ and a historical ordering $\sch_{\vt} \in \cH(\vt)$, define the following sequence of weight functionals: 
\textcolor{black}{\begin{align}
	\vW(\vt, \sch_{\vt}, k)&:= \vW(\cT(\sch_{\vt}, k)), \qquad 0\leq k\leq |\vt|-1, \label{eqn:haalp} \\ 
	 \vw_{\cic}(\vt, \sch_{\vt}, k+1) &:= \deg(v_k,\cT(\sch_{\vt}, k))  \phi_{\rho_{k}, \rho_{k+1} }, \qquad 0\leq k\leq |\vt|-2, \label{eqn:weight-order-def}
\end{align}}
\textcolor{black}{where $\rho_k, \rho_{k+1}$ denote the attributes of the $k$-th and $(k+1)$-th vertices in the historical ordering $\sch_{\vt}$.}
Note that for any such historical ordering, the final term in {\cg \eqref{eqn:haalp}, namely with $k=|\vt|-1$, is the same and matches the expression in \eqref{eqn:haalp2} }, irrespective of the ordering:
\begin{equation}
\label{eqn:807}
	\vW(\vt, \sch_{\vt}, |\vt|-1) = \sum_{i=0}^{|\vt|-1} \phi_{\rho_i} \deg(v_i, \vt) = \vW(\vt). 
\end{equation}

Next given $\vt \in \overrightarrow{\bbT_{\cS}}$ note that each vertex $v\in \vt$ (with non-zero children) has an ordering of the children from oldest to youngest, and thus one can talk about the youngest child $y_{v}$ of $v$. For $v \in \vt$ define the indicators,
\begin{equation}
\label{eqn:ind-def}
	I_v:= \ind\set{v \mbox{ has a child in $\vt$ and $y_{v}$ is a leaf in $\vt$}}. 
\end{equation}
If $I_v= 1$, denote by $\vt^{\sss(v)}$ the tree obtained by deleting the youngest child of $v$ (which is necessarily a leaf). Let $\rho_{y_v}$ denote the attribute type of the youngest child of $v$. 

Define the measure, 
\begin{equation}
\label{eqn:pm-def}
	\fp(\vt):= \sum_{\sch_{\vt} \in \cH(\vt)} \frac{2 \pi_{\rho_0}}{2+\vW(\vt)} \prod_{k=0}^{|\vt|-2} \left[\frac{\vw_{\cic}(\vt, \sch_{\vt}, k+1)}{2+ \vW(\vt, \sch_{\vt}, k)}\right], \qquad \vt \in \overrightarrow{\bbT_{\cS}}. 
\end{equation}
Here the product above is taken to be one if $|\vt|=1$. Recall the asserted limit $\fpm_{\infty}$ of the fringe distribution \eqref{eqn:fringe-limit}. Analogously for each $a\in [K]$, write $\fpm_{a, \infty} = \cL(\BP_{a; (1,\mvnu, \kappa)}(\tau))$ where as before $\tau\sim \Exp(2)$ independent of the branching process. 
\begin{prop}
	\label{prop:prop-of-fp}
	\begin{enumeratea}
		\item The measure $\fp(\cdot)$ in \eqref{eqn:pm-def} is in fact a probability measure and has the equivalent description, 
		\[\fp(\vt) = \pi_{\rho_0} \fpm_{\rho_0, \infty}(\vt), \qquad \vt \in \overrightarrow{\bbT_{\cS}}, \]
		where $\vt \in \overrightarrow{\bbT_{\cS}} $ with root type $\rho_0$. 
		\item $\fp(\cdot)$ is the unique measure on $\overrightarrow{\bbT_{\cS}}$ satisfying the recursive equation 
		\begin{equation}
		\label{eqn:recur-first-part}
			\fpm(\vt) = \frac{\sum_{v\in \vt} I_v  \cdot  \fpm(\vt^{\sss(v)})\frac{\deg(v,\vt)-1}{2} \phi_{\rho_v , \rho_{y_v}}}{1+\sum_{v\in \vt} \phi_{\rho_v} \frac{\deg(v,\vt)}{2}},
		\end{equation}
		with boundary conditions for $\vt = \set{v_0, a(v_0) = b}$, i.e. a tree consisting of a single vertex of type $b\in [K]$, given by 
		\begin{equation}
		\label{eqn:recur-boundary}
			\fpm(\vt) = \frac{2\pi_b}{2+ \phi_{b}}, \qquad b\in [K]. 
		\end{equation}
	\end{enumeratea}
\end{prop}

\begin{proof}[{\bf Proof of Proposition \ref{prop:prop-of-fp}}:] We start with (a). The proof follows the similar lines as the proof of \cite[Theorem 2(b)]{rudas2007random}. Fix $\vt \in \overrightarrow{\bbT_{\cS}} $ and assume the root is of type $a\in [K]$. For the rest of the proof write $\BP(\cdot) = \BP_{a; (1,\mvnu, \kappa)}(\cdot)$; we will continue to use $\pr_a$ and $\E_a$ for the probability and expectation operators for the process $\BP$ to remind us that the process starts with root of type $a$.
	
	 For any finite time $t>0$ we will also view $\BP(t)$ as a random element of $\overrightarrow{\bbT_{\cS}}$. Note that for $t=0$, $\BP(0)$ is a single vertex of type $a$ and thus $\BP(0) \subseteq \vt$.  Define the stopping times, 
	\[\tau_{\vt}:= \sup\set{t\geq 0: \BP(t) \subseteq \vt},\]
	and 
	\[\tau_{\vt}^\prime:= \sup\set{t\geq 0: \BP(t) \subsetneq \vt}.\]
	By Fubini, 
	\begin{equation}
	\label{eqn:301}
		\pr(\BP(\tau) = \vt) = \int_0^\infty 2e^{-2t} \pr_a(\BP(t) = \vt)dt = \E_a\left[\left(e^{-2\tau_{\vt}^\prime} - e^{-2\tau_{\vt}} \right)\ind\set{\tau_{\vt}^\prime < \tau_{\vt}}\right]. 
	\end{equation}
	Now on the event $\set{\tau_{\vt}^\prime < \tau_{\vt}}$, $\BP(\tau_{\vt}^\prime) = \vt$ and the combined rate of evolution to the next transition is $\vW(\vt)$. Thus conditional on $\set{\tau_{\vt}^\prime < \tau_{\vt}}$, \[\tau_{\vt} - \tau_{\vt}^\prime \sim \Exp(\vW(\vt)),\] independent of $\BP(\tau_{\vt}^\prime)$.  Using this in \eqref{eqn:301} gives, 
	\begin{align}
		\pr(\BP(\tau) = \vt) &= \frac{2}{2+\vW(\vt)} \E_a(e^{-2\tau_{\vt}^\prime} \ind\set{\tau_{\vt}^\prime < \tau_{\vt}}) \notag\\
		&=\frac{2}{2+\vW(\vt)} \sum_{\sch \in \cH(\vt)} \E_a(e^{-2T_{|\vt|-1}} \ind\set{(\eta_0, \eta_1, \ldots, \eta_{|\vt|-1}) = \sch}),\label{eqn:302}
	\end{align}
	where $\eta_i$ is the label of the $i$-th born vertex into $\BP$ and
	\begin{equation}
	\label{eqn:313}
		T_k := \inf\set{t\geq 0: |\BP(t)| = k+1}, \qquad k\geq 0. 
	\end{equation}
	By properties of exponential random variables and the dynamics of the process $\BP$, for any fixed historical ordering $\sch = (s_0, s_1, \ldots, s_{|\vt|-1})$,
	\begin{equation}
	\label{eqn:322}
		\pr_a((\eta_0, \eta_1, \ldots, \eta_{|\vt|-1}) = (s_0, s_1, \ldots s_{|\vt|-1})) = \prod_{k=0}^{|\vt|-2} \left[\frac{\vw_{\cic}(\vt, \sch, k+1)}{\vW(\vt, \sch, k)}\right].
	\end{equation}
	Further conditional on the event $\set{(\eta_0, \eta_1, \ldots, \eta_{|\vt|-1}) = (s_0, s_1, \ldots s_{|\vt|-1})}$, $\set{T_{k+1} - T_k:0\leq k\leq |\vt|-2}$ are conditionally independent with $T_{k+1} - T_k \sim \Exp(\vW(\vt, \sch, k))$. Thus we get for each $\sch=(s_0, s_1, \ldots s_{|\vt|-1})\in \cH(\vt)$, 
	\begin{align}
	\E_a(e^{-2T_{|\vt|-1}} \ind\set{(\eta_0, \eta_1, \ldots, \eta_{|\vt|-1}) = \sch}) &= \prod_{k=0}^{|\vt|-2} \left[\frac{\vw_{\cic}(\vt, \sch, k+1)}{\vW(\vt, \sch, k)}\right]	\prod_{k=0}^{|\vt|-2} \frac{\vW(\vt, \sch, k)}{2+\vW(\vt, \sch, k)} \notag\\
	& = \prod_{k=0}^{|\vt|-2} \left[\frac{\vw_{\cic}(\vt, \sch, k+1)}{2+\vW(\vt, \sch, k)}\right]. 	\label{eqn:334}
	\end{align}
	Using the final expression in \eqref{eqn:334} in \eqref{eqn:302} and comparing with the definition of $\fp$ in \eqref{eqn:pm-def} completes the proof of (a). 
	
	Let us now prove part (b) of the Proposition; for this it is enough to show:
	\begin{enumeratei}
		\item $\fp$ satisfies the boundary conditions namely \eqref{eqn:recur-boundary}. This is immediate by the definition of $\fp$ in \eqref{eqn:pm-def}. 
		\item $\fp$ satisfies the recursive equation \eqref{eqn:recur-first-part}: Now consider $\vt$ with $|\vt|\geq 2$. For ease of notation, for any vertex $u\in \vt$, write $d_u = \deg(u,\vt) $ for the degree in $\vt$.  Let $\cY_{\vt}:=\set{v: I_v =1}$ denote the collection of vertices whose youngest child is a leaf. Note that the historical orderings $\cH(\vt)$ can be partitioned as the disjoint union 
		\[\cH(\vt) = \sqcup_{u\in \cY_{\vt}} \set{\sch_{\vt} = (v_0, v_1, \ldots, v_{|\vt|-1}) \in \cH(\vt): v_{|\vt|-1} = y_u }=: \sqcup_{u\in \cY_{\vt}} \tilde{\cH}^{\sss(u)}(\vt),\] 
		i.e. based on the identity of the final vertex (necessarily a leaf) that was adjoined to form $\vt$. \textcolor{black}{Further $\tilde{\cH}^{\sss(u)}(\vt)$ can be viewed as all historical orderings of the tree $\vt^{\sss(u)}$ followed by the addition of $y_u$ in the final step to obtain $\vt$. Thus with minor {\cg abuse} of notation, we will write $\tilde{\cH}^{\sss(u)}(\vt) = {\cH}(\vt^{\sss(u)})$}. Now starting with the form of $\fp$ in \eqref{eqn:pm-def}, using the above partition of $\cH(\vt)$ and the from of the weight sequences in \eqref{eqn:weight-order-def}, we get, 
		\begin{align*}
			\fp(\vt) &= \sum_{u\in \cY_{\vt}} \sum_{\sch \in \cH(\vt^{\sss(u)})} \frac{2 \pi_{\rho_0}}{2+\vW(\vt)} \prod_{k=0}^{|\vt^{\sss(u)}|-2} \left[\frac{\vw_{\cic}(\vt^{\sss(u)}, \sch_{\vt^{\sss(u)}}, k+1)}{2+ \vW(\vt^{\sss(u)}, \sch_{\vt^{\sss(u)}}, k)}\right]\frac{(d_u-1) \phi_{\rho_u, \rho_{y_u}}}{(2+\vW(\vt^{\sss(u)}))}, \\
			&= \sum_{u\in \cY_{\vt}} \sum_{\sch \in \cH(\vt^{\sss(u)})} \frac{2 \pi_{\rho_0}}{2+\vW(\vt^{\sss(u)})} \prod_{k=0}^{|\vt^{\sss(u)}|-2} \left[\frac{\vw_{\cic}(\vt^{\sss(u)}, \sch_{\vt^{\sss(u)}}, k+1)}{2+ \vW(\vt^{\sss(u)}, \sch_{\vt^{\sss(u)}}, k)}\right]\frac{(d_u-1) \phi_{\rho_u, \rho_{y_u}}}{(2+\vW(\vt))}, \\ 
			&=  \frac{\sum_{u\in \cY_{\vt}} \fp(\vt^{\sss(u)})(d_u-1)\phi_{\rho_u, \rho_{y_u}} }{2+\vW(t)}, \quad \mbox{using \eqref{eqn:pm-def}.}
		\end{align*}
		Using the form $\vW$ in \eqref{eqn:807} completes the proof. 
	\end{enumeratei}  
\end{proof}

\noindent {\bf Proof of Theorem \ref{thm:just-fringe-convg}:} We assume that $\set{\cG_n:n\geq 0}$ is constructed as an increasing sequence of networks on a common probability space so we can make statements a.s. Each $\cG_n$ is viewed as an element in $\overrightarrow{\bbT_{\cS}}$ with children of each vertex ordered according to their order of birth into the system. From Section \ref{sec:fri-decomp}, recall that for $\vt\in \overrightarrow{\bbT_{\cS}}$ and $v\in \vt$, $f_0(v,\vt) \in \overrightarrow{\bbT_{\cS}}$ denotes the fringe of vertex $v \in \vt$. Writing $n_0 = |\cG_0|$ and define the empirical count and fringe density in $\cG_n$ respectively via, 
\[c_n(\vt):= \sum_{v\in \cG_n}\ind\set{f_0(v,\cG_n) = \vt}, \qquad \fP_n(\vt):= \frac{c_n(\vt)}{n+n_0} , \qquad \vt\in \overrightarrow{\bbT_{\cS}}.\]  
  For the rest of the argument $\vt$ will be a fixed element of $\overrightarrow{\bbT_{\cS}}$. \textcolor{black}{The following quantity gives the average degree of vertices with a given attribute type}:
\begin{equation}
\label{eqn:540}
	\tilde{Y}_a^{\sss(n)}:= \frac{\sum_{v\in \cG_n: a(v) = a} \deg(v,n)}{2(n+n_0)}, \qquad a\in [K],
\end{equation} 
The next two definitions are relative to the fixed $\vt$:
\begin{equation}
\label{eqn:541}
	\Phi_{n,u, {y_u}}:=\frac{\kappa(\rho_u, \rho_{y_u})\pi_{\rho_{y_u}}}{\sum_{a\in [K]} \kappa(a,\rho_{y_u})\tilde{Y}_a^{\sss(n)}}, \quad u\in \cY_{\vt},  \qquad \Psi_{n,u}:=\sum_{a\in [K]} \frac{\kappa(\rho_{u}, a) \pi_a}{\sum_{b\in [K]} \kappa(b,a)\tilde{Y}_b^{\sss(n)}},\quad  u\in \vt. 
\end{equation}
Then observe that, by the dynamics of the process $\set{\cG_n:n\geq 0}$, 
\begin{equation}
\label{eqn:626}
	\pr(c_{n+1}(\vt) = c_n(\vt)+1|\cG_n) = \sum_{u\in \cY_{\vt}} c_n(\vt^{\sss(u)})\left[ \Phi_{n,u, y_{u}} ~\cdot ~ \frac{(d_u-1)}{2(n+n_0)}\right],
\end{equation}
where as before $d_u = \deg(u,\vt)$. Further, 
\begin{equation}
\label{eqn:633}
	\pr(c_{n+1}(\vt) = c_n(\vt)-1|\cG_n) = c_n(\vt) \sum_{u\in \vt} \frac{\Psi_{n,u}\cdot d_u}{2(n+n_0)}. 
\end{equation}
Combining \eqref{eqn:626} and \eqref{eqn:633} and dividing throughout by $(n+n_0+1)$ gives (with $I_u$ as in \eqref{eqn:ind-def}):
\begin{align*}
	\E(\fP_{n+1}(\vt)|\cG_n) =\frac{1}{n+n_0+1}\bigg[&\left((n+n_0) - \sum_{u\in \vt} \Psi_{n,u}\cdot \frac{d_u}{2}\right)\fP_n(\vt)\\
	&+ \sum_{u\in \vt} \fP_n(\vt^{\sss(u)})\left( \Phi_{n,u, y_{u}} ~\cdot ~ \frac{(d_u-1)}{2}\right)I_u \bigg],
\end{align*}
taking by convention terms in the second sum with $I_u=0$ as $0$.
Rearranging gives, 
\begin{align}
	\E(\fP_{n+1}(\vt)|\cG_n) - \fP_n(\vt) = \frac{1}{n+n_0+1}&\bigg[\sum_{u\in \vt} \fP_n(\vt^{\sss(u)})\left( \Phi_{n,u, y_{u}} ~\cdot ~ \frac{(d_u-1)}{2}\right)I_u \notag \\
	&-\left(1+\sum_{u\in \vt} \Psi_{n,u}\cdot \frac{d_u}{2}\right)\fP_n(\vt) \bigg]. \label{eqn:up-up-away}
\end{align}
Now note that the first term on the right hand side involves empirical proportions of trees which are strict sub-trees of $\vt$. This suggests that if proportions of these trees converge, then under regularity conditions so should $\fP_n(\vt)$, where the limit satisfies a recursive equation inherited from \eqref{eqn:up-up-away}. We now implement this program. We first paraphrase the following slight variant of the stochastic approximation result from \cite{jordan2013geometric}. The proof is a slight modification of \cite[Lemma 3.3]{jordan2013geometric} and \cite[Lemma 2.6]{pemantle2007survey} and is omitted. {\cg Applying this to the evolution equation in \eqref{eqn:up-up-away} will complete the proof of Theorem \ref{thm:just-fringe-convg}. }
\begin{lemma}[{\cite[Lemma 3.3]{jordan2013geometric}}]\label{lem:jordan}
	Suppose $\set{A_n, B_n, \xi_n, R_n :n\geq 1}$ be real valued random variables. Suppose $a \in \mathbb{R}$ and $k>0$ are constants. Suppose that 
	\[B_{n+1} - B_n = \frac{1}{n}(A_n - kB_n +\xi_n) + R_{n+1}.\]
	Further suppose 
	\begin{enumeratei}
		\item $A_n \convas a$ as $n\to\infty$. 
		\item $\sum_{n=1}^{\infty} R_n <\infty$ a.s.
		\item $\E(\xi_n) = 0$ and $\xi_n$ is bounded.  
		\item $\sup_n B_n < \infty$ a.s.
	\end{enumeratei}
	Then $B_n\convas a/k$ as $n\to\infty$.  
\end{lemma}
{\cg If $\vt$ is a tree consisting of a single vertex $\vt = \set{v}$ of attribute type $a(v)=a$ then note that the empirical fringe proportion $\fP_n(\vt)$ is just the proportion of leaf vertices of type $a$. The main goal of \cite{jordan2013geometric} was understanding asymptotics for the degree distribution in this model.  } Thus, by \cite[Theorem 2.2]{jordan2013geometric}, 
\begin{equation}
\label{eqn:141}
	\fP_n(\vt) \convas \varpi(\vt):= \frac{2\pi_a}{2+\phi_a} = \fp(\vt).
\end{equation}
Fix $l\geq 2$ and assume that for 
\begin{equation}
\label{eqn:hypo-vt}
	\forall ~\tilde{\vt} \in \overrightarrow{\bbT_{\cS}}  \mbox{  with } |\tilde{\vt}| \leq (l-1), \qquad \fP_n(\tilde{\vt}) \convas \varpi(\tilde{\vt}), \qquad \mbox{ as } n\to \infty, 
\end{equation}
for some positive limit constant $\varpi(\tilde{\vt})$ which satisfies the boundary condition in \eqref{eqn:141}. Let $\vt \in\overrightarrow{\bbT_{\cS}} $ with $|\vt| = l$ and consider the evolution equation in \eqref{eqn:up-up-away}. We note:
\begin{enumeratea}
	\item By \cite[Proposition 3.1]{jordan2013geometric}, the functional defined in \eqref{eqn:540},
	\begin{equation}
	\label{eqn:8522}
		\tilde{Y}_a^{\sss(n)} \convas \eta_a \mbox{ with } \mveta = (\eta_1, \ldots, \eta_K) \mbox{ as in Lemma \ref{lem:minimizer}.} 
	\end{equation}  
	\item Thus for the functionals in \eqref{eqn:541}, for each $u\in \vt$,
	\[\Psi_{n,u} \convas \sum_{a\in [K]} \frac{\kappa(\rho_{u}, a) \pi_a}{\sum_{b\in [K]} \kappa(b,a)\eta_b} = \phi_{\rho_u}.  \]
\item Further, for each $u\in \cY_{\vt}$, 
\[\Phi_{n,u, y_u} \convas \phi_{\rho_u, \rho_{y_u}}, \qquad \mbox{ as } n\to\infty. \]
\end{enumeratea}
Now we will use Lemma \ref{lem:jordan}; mimicking the notation in this Lemma:
\begin{enumeratei}
	\item  $B_n = \fP_n(\vt)$;
	\item Next let, 
	\[A_n = \sum_{u\in \vt} \fP_n(\vt^{\sss(u)})\left( \Phi_{n,u, y_{u}} ~\cdot ~ \frac{(d_u-1)}{2}\right)I_u - \sum_{u\in \vt} \left(\Psi_{n,u} - \phi_{\rho_u} \right)\cdot\frac{d_u}{2}\cdot \fP_n(\vt), \]
\item Finally define,
\[k =\left(1+\sum_{u\in \vt} \phi_{\rho_u}\cdot \frac{d_u}{2}\right), \qquad \xi_n = (n+n_0+1)(\fP_{n+1}(\vt) - \E(\fP_{n+1}(\vt)|\cG_n)), \qquad R_n = 0. \]
\end{enumeratei}  
Clearly, $B_n \le 1$ for all $n$. Also, note that
$$
|\xi_n| \le |(n+n_0+1)\fP_{n+1}(\vt) - (n+n_0)\fP_{n}(\vt)| + \E[|(n+n_0)\fP_n(\vt) - (n+n_0+1)\fP_{n+1}(\vt)| \, |\cG_n)] \le 2.
$$
Thus by Lemma \ref{lem:jordan} and hypothesis in \eqref{eqn:hypo-vt} we get that $\fP_n(\vt) \convas \fpm(\vt)$, where $\fpm(\cdot)$ satisfies the recursive equation in \eqref{eqn:recur-first-part}. Using Proposition \ref{prop:prop-of-fp} now completes the proof of Theorem \ref{thm:just-fringe-convg}. 
\end{proof}

\textcolor{black}{\subsection{Proof of Theorem \ref{thm:deg-homophil-linear}}
We first prove part (a). Fix $a \in [K]$. Define $S^a_0=0$ and $S^a_k := \sum_{l=1}^{k} \frac{E_l}{\phi_a(l + 1)}, \, k \in \mathbb{N}$, where $\{E_l\}$ are \emph{i.i.d.} $\Exp(1)$ random variables. Recall $\tau \sim \Exp(2)$. By Theorem \ref{thm:fringe-convergence} and using standard properties of exponential random variables, we obtain for any $k \ge m_a$,
\begin{align*}
\vp^a_{\infty}(k) = \pr_a\left(S^a_{k-1} \le \tau < S^a_{k}\right) = \frac{2}{2 + \phi_ak}\prod_{i=1}^{k-1}\frac{\phi_a i}{2 + \phi_a i},
\end{align*}
where the product is taken to be $1$ if $k=1$. Part (a) follows from the above.}

\textcolor{black}{To prove part (b), Note that, by Theorem \ref{thm:fringe-convergence}, as $n \rightarrow \infty$,
$$
\frac{\cE_{n, aa}}{n\pi_a} \convas \frac{\phi_{a,a}}{2-\phi_a}, \qquad  \frac{\cV_{n, a}}{n} \convas \pi_a
$$
where the first limit can be checked to equal the expected number of type $a$ children by the root in $\BP_a(\tau)$, and the second limit equals the probability of the root being of type $a$ in $\BP_{A;(1,\mvnu, \kappa)}(\tau)$. Moreover, $p_n = |\cE_n|/{n \choose 2} = \frac{2}{n}(1 + o(1))$. Combining these observations,
$$
D_{n, a} = |\cE_{n, aa}|/({|\cV_{n, a}| \choose 2} p_n) \convas \frac{\phi_{a,a}}{\pi_a(2-\phi_a)} = \frac{\vM_{(a,a)}}{\pi_a},
$$
proving the first assertion. The second assertion follows similarly.}

\section{Proofs: PageRank asymptotics}
\label{sec:proofs-page-rank}
 In this section we prove the various results relating to asymptotics for PageRank.

\subsection{\bf Proof of Theorem \ref{thm:page-rank}} This follows from the local weak convergence in Theorem \ref{thm:fringe-convergence} coupled with results in \cite{garavaglia2020local,banerjee2021pagerank}. The proof of \cite[Theorem 4.6]{banerjee2021pagerank}, with minor modifications to replace the degree with the attribute type and convergence in probability in the local weak sense with almost sure convergence, yield the result.

\subsection{\bf Proof of Theorem \ref{thm:page-rank-tails}}
{\cg Before diving into the proof, let us give an outline. We first relate the PageRank function $\cR_{\emptyset, c}(\cdot)$ to a probabilistic construction of percolation on the original branching process (see Def. \ref{def:perc-bp}). Thus, understanding asymptotic properties of PageRank scores turns out to be related to understanding the rate of growth of this percolated branching process which is summarized in Proposition \ref{prop:r-zc-moments}. This forms the heart of the proof and careful understanding of the evolution dynamics of this percolated branching process requires tools from semi-Markov renewal theory. The estimates provided by Proposition \ref{prop:r-zc-moments}  are used at the end of this section to complete the proof of Theorem \ref{thm:page-rank-tails}.  }

\begin{defn}[Percolation on $\BP_a$]
	\label{def:perc-bp}
	Fix $a\in [K]$ and damping factor $c\in (0,1)$. For any $t\geq 0$, write $\BP^c_a(t)$ for the connected cluster of the root (which is also a tree) when we retain each edge $e\in \BP_a(t)$ with probability $c$ and delete with probability $(1-c)$, independently across edges. Write $\set{\BP_a^c(t):t\geq 0}$ for the \textcolor{black}{corresponding non-decreasing rooted tree valued Markov process where children born to vertices in the connected cluster of the root are retained with probability $c$ at their time of birth.} Let $\scZ_a^c(t) = |\BP_a^c(t)|$ denote the size of the cluster at time $t$.  
\end{defn} 
Then by definition in \eqref{eqn:t-power}, if the root $\emptyset$ has type $a$, then
\begin{equation}
\label{eqn:217}
	\cR_{\emptyset, c}(t) = (1-c)\E_a(\scZ_a^c(t)|\BP_a(t)), \qquad t\geq 0. 
\end{equation}
 {\cg Recall the matrix $\vM^{\sss(c)}$ in \eqref{eqn:matrix-mc-def}  with Perron-Frobenius eigen-value $\lambda_c$ and let $\mvh = (h_1, \ldots, h_K)$ denote the corresponding (strictly positive) right eigen-vector of $\vM^{\sss(c)}$,  normalized as $\sum_{i=1}^k h_i = 1$. }

\begin{prop}\label{prop:r-zc-moments}
	\begin{enumeratea}
		\item \label{it:mom} For any $a\in [K]$ and $k\geq 1$,
	\[\sup_{t\in \bR_+} \E_a\bigg[\big(e^{-\lambda_c t } \cR_{\emptyset, c}(t)\big)^k\bigg] \leq (1-c)^k \sup_{t\in \bR_+} \E_a\bigg[\big(e^{-\lambda_c t } \scZ_{a}^c(t)\big)^k\bigg] < \infty. \] 
	\item $\exists ~\delta_1>0$ such that $\forall a\in [K]$,
	\[\liminf_{t\to \infty} \pr_a(e^{-\lambda_c t} \cR_{\emptyset, c}(t) \geq \delta_1) > 0. \]
	\end{enumeratea}
	
\end{prop}
\begin{rem}
	\label{rem:perron-l-1}
	For any $\lambda > \max_a \phi_a$, consider the matrix $\tilde{\vM}^{\sss(c)}(\lambda) = (c\phi_{a,b}/(\lambda - \phi_a))_{a,b\in [K]}$.  Compare this with \eqref{eqn:matrix-def}.  An equivalent description of $\lambda_c$ and $\mvh$ is as follows: $\lambda_c$ is the unique value for which the Perron root $\tilde{\vM}^{\sss(c)}(\lambda_c)$ is one, with corresponding eigen-vector $\mvh$ (note that $\lambda_c> \max_a \phi_a$ by Proposition \ref{prop:lc-bound}). Thus, in view of \cite{jagers1996asymptotic,jagers1989general}, we are considering $\BP_a^{c}(\cdot)$ as a (multi-type) branching process in its own right and quantifying it's evolution. 
\end{rem}
\begin{proof}[Proof of Part (a)]\label{pf:lem-r-zc-a}
	The first inequality follows from conditional Jensen's inequality, using \eqref{eqn:217} since, 
{\cg \[\E((\cR_{\emptyset, c}(t))^k) = (1-c)^k\E\big[\big(\E_a(\scZ_a^c(t)|\BP_a(t))\big)^k\big] \leq (1-c)^k\E((\scZ_a^c(t))^k). \]}
	 To prove the second inequality, first define the processes, 
	\begin{equation}
	\label{eqn:ut-yt-def}
		U(t):=\sum_{b\in [K]} h_b Y_b(t), \qquad Y_b(t):= \sum_{v\in \BP_a^c(t), a(v) = b} \deg(v,t), \qquad t\geq 0. 
	\end{equation}
As before, $\deg(v,t)$ is the degree of $v$ in the full branching process $\BP_a(t)$, {\bf not} just in the percolated process $\BP_a^c(t)$. Let $\cA$ denote the generator of the continuous time Markov process $\BP_a(\cdot)$. Using the fact that each new vertex created by an individual in $\BP_a^c(\cdot)$ is retained with probability $c$, we get:
\begin{align}
	\cA U(t) &= \sum_{i\in [K]} h_i \phi_i Y_i(t) + c\sum_{i,j} h_i \phi_{j,i} Y_j(t) \notag \\
	&= \sum_{i,j \in [K]} h_i \vM_{ji}^{\sss(c)} Y_j(t) = \lambda_c \sum_{j\in [K]} h_j Y_j(t) = \lambda_c U(t) \label{eqn:933},
\end{align}
{\cg where for fixed $i\in [K]$, in the first line, the first term accounts for the birth of every new individual, which adds to the degree of the parent, irrespective of whether it is retained in $\BP^c_a$ or not, while the second corresponds to the contribution of newly born vertices of type $i$ to parents of type $j$, {\bf that are retained} (thus the factor $c$).} In the second line, we used the definition of $\lambda_c$ and $\mvh$. Writing $u_a(t) = \E_a(U(t))$, using \eqref{eqn:933} gives $u_a^\prime(t) =\lambda_c u_a(t)$ with $u_a(0) = h_a$. Thus, 
\begin{equation}
\label{eqn:944}
	u_a(t):= \E_a(U(t)) = h_a e^{\lambda_c t}, \qquad t\geq 0.  
\end{equation} 

\noindent {\bf Induction statement:} For any $k \ge 1$ and any $a\in [K]$, $\exists \mu_{a, k}<\infty$ such that, 
\begin{equation}
	\label{eqn:ind-hyp}
	\sup_{t <\infty} \E_a(e^{-\lambda_c k t}[U(t)]^k) \leq \mu_{a, k}. 
\end{equation}
The above holds for $k=1$ with $\mu_{a,1} = h_a$ for all $a\in [K]$. Suppose for some $J \ge 2$ the induction statement holds for all $a \in [K]$ and $1\leq k\leq J-1$. Now consider $k=J$. \textcolor{black}{For $i \in [K]$, let $X_i$ be a random variable which takes value $0$ with probability $1-c$ and $h_j$ with probability $c \phi_{i,j}/\phi_i$ for $j \in [K]$. Denote by $\E_{X_i}[\cdot]$ the expectation taken only with respect to $X_i$ (conditional on $\BP(t)$). Then the generator $\cA$ of the continuous time Markov process $\BP_a^c(\cdot)$, applied to $U(t)^J$, takes the following form:
\begin{align}
	&\cA[(U(t))^J] = \sum_{i} \E_{X_i}[(U(t) +h_i + X_i)^J - U(t)^J] \phi_i Y_i(t) \notag \\
	&\quad= J(U(t))^{J-1}\sum_{i} \E_{X_i}[h_i + X_i] \phi_i Y_i(t) + \sum_{l=2}^{J} {J \choose l} (U(t))^{J-l} \sum_{i}  \E_{X_i}[(h_i + X_i)^l] \phi_i Y_i(t)\notag\\
	&\quad\le J(U(t))^{J-1}\sum_{i} \E_{X_i}[h_i + X_i] \phi_i Y_i(t) + \sum_{l=2}^{J} {J \choose l} 2^{l-1} (U(t))^{J-l} \sum_{i}  \E_{X_i}[h_i + X_i] \phi_i Y_i(t)\notag\\
	&\quad = J(U(t))^{J-1} \sum_{i,j} h_j \vM_{i,j}^{\sss(c)} Y_i(t) + \sum_{l=2}^{J} {J \choose l} 2^{l-1} (U(t))^{J-l}\sum_{i,j} h_j \vM_{i,j}^{\sss(c)} Y_i(t) \notag \\
	&\quad = J \lambda_c (U(t))^J +\lambda_c \sum_{l=2}^J {J \choose l}2^{l-1}(U(t))^{J-l+1},\notag\\ \label{eqn:1009}
\end{align}
where in going from the second to third line, we used $h_i + X_i \le 2$ for all $i \in [K]$, and in going from the fourth to fifth line, the terms simplify as in \eqref{eqn:933}. Thus,  writing $u_{a, J}(t) = \E_a([U(t)]^J)$ we get, 
\begin{align}
	\frac{d}{dt} e^{-\lambda_c J t} u_{a, J}(t) &\leq \lambda_c e^{-\lambda_c t} \sum_{l=2}^J {J \choose l}2^{l-1} \E_a\big[e^{-(J-l+1)\lambda_c t}(U(t))^{J-l+1}\big] \notag \\
	&\leq \lambda_c e^{-\lambda_c t}\sum_{l=2}^{J} {J \choose l}2^{l-1}\mu_{a, J-l+1}, \label{eqn:1016}
\end{align} 
where in the last bound we used the induction hypothesis \eqref{eqn:ind-hyp}. Integrating we get, 
\[\sup_{t <\infty}e^{-\lambda_c J t} u_{a, J}(t) \leq \sum_{l=2}^{J} {J \choose l}2^{l-1}\mu_{a, J-l+1} := \mu_{a, J} <\infty,\]
extending the result for $J$.} To finish the proof note that 
\[(e^{-\lambda_c t} \scZ_a^c(t))^k \leq (\inf_{b\in [K]} h_b)^{-k} (e^{-\lambda_c t} U(t))^k.\]
\end{proof}

\begin{proof}[Proof of Part (b)]
	\label{pf:lem-r-zc-b}
	The proof relies crucially on the following Lemma which implies a uniform positive lower bound on the first moment.
	\begin{lemma}\label{lem:lower-exp-zc}
		For any $a\in [K]$, 
		\[\lim_{t\to\infty} \E_a[e^{-\lambda_c t}\cR_{\emptyset, c}(t) ] ={(1-c)}\lim_{t\to\infty} \E_a[e^{-\lambda_c t}\scZ_a^c(t)] := c_{a,1} >0. \]
	\end{lemma}

	\noindent {\bf Proof of (b) assuming {\cg Lemma \ref{lem:lower-exp-zc}}:} Now we essentially use the second moment method (via the Paley-Zygmund inequality). Using part (a) note $c_2 := \sup_{a\in [K]} \mu_{a,2} < \infty$. Using Lemma \ref{lem:lower-exp-zc}, let $c_1 := \inf_{a\in [K]} c_{a,1} >0$. Then for any $a\in [K]$,
	\begin{align}
		\liminf_{t\to\infty} \pr_a\big(e^{-\lambda_c t}\cR_{\emptyset, c}(t) \geq \frac{c_1}{4}\big) &\geq \liminf_{t\to\infty}\pr_a\left(e^{-\lambda_c t}\cR_{\emptyset, c}(t) \geq \frac{\E_a(e^{-\lambda_c t}\cR_{\emptyset, c}(t))}{2}\right)\notag \\
		&\geq \liminf_{t\to\infty} \frac{\E_a(e^{-\lambda_c t}\cR_{\emptyset, c}(t))}{4\E_a([e^{-\lambda_c t}\cR_{\emptyset, c}(t)]^2)} \geq \frac{c_1^2}{4c_2} >0. \notag
	\end{align}   
	\qed 
	
	\noindent {\bf Proof of Lemma \ref{lem:lower-exp-zc}:} By \eqref{eqn:217}, for any $t\geq 0$,  $\E_a[e^{-\lambda_c t}\cR_{\emptyset, c}(t) ] ={(1-c)}\E_a[e^{-\lambda_c t}\scZ_a^c(t)]$.  Thus it is enough to analyze the process $\scZ_a^c$. For $l\geq 1$, write {\cg $\varrho_l$ for the $l$-th offspring of the root $\emptyset$ in $\BP_a^c$, and let  $a(\varrho_l) \in [K]$ and $\beta_l >0 $ for the corresponding attribute type and birth time. Then,
	\begin{equation}
	\label{eqn:dist}
		\scZ_a^{c}(t) \stackrel{d}{=} 1 + \sum_{k=1}^\infty \scZ_{a(\varrho_k)}^c(t-\beta_k), 	
		\end{equation} 
	where, conditional on the types, $\set{\BP_{a(\rho_k)}^c(\cdot):k\geq 1}$ are independent percolated branching processes as in Definition \ref{def:perc-bp} and $\scZ_{a(\rho_k)}^c(\cdot) = |\BP_{a(\rho_k)}^c(\cdot)|$. }
For $a,b \in [K]$, let $\mu^c(a,b;  \cdot)$ denote the intensity measure of the point process of type $b$ children emanating from a type $a$ parent. Then \eqref{eqn:dist} gives
\begin{equation}
\label{eqn:259}
	\E_a(\scZ_a^c(t)) = 1+ \int_0^t \sum_{b\in [K]} \E_b[\scZ_b^c(t-s)]\mu^c(a,b; ds). 
\end{equation}
Define
\begin{equation}
\label{eqn:819}
	m(a,t) = \frac{e^{-\lambda_c t} \E_a(\scZ_a^c(t))}{h_a}, \qquad \mu^{*,c}(a,b; dt) = e^{-\lambda_c t} \mu^c(a,b; dt) \frac{h_b}{h_a}.
\end{equation}  
Then by \eqref{eqn:259},
\begin{equation}
\label{eqn:r*}
	m(a,t) = \frac{e^{-\lambda_c t}}{h_a} + \int_0^t \sum_{b\in [K]} m(b, t-s) \mu^{*,c}(a,b;ds).
\end{equation}
Recall the evolution of children of type $b$ from type $a$ parent in the (un-percolated) process $\BP$ from Section \ref{sec:proofs-local}. Write $N_a(\cdot)$ for the total offspring process (again in $\BP$) of a type $a$ parent. Let $\set{B_{i, a\leadsto b}: i\geq 1}$ be an \emph{i.i.d.} $\Bern(c\phi_{a,b}/\phi_a)$ sequence (also independent of $N_a$). Then we have 
\begin{align}
	\mu^c(a,b,[0,t]):= \E\left(\sum_{l=1}^{N_a(t)} B_{l, a\leadsto b} \right) = \frac{c\phi_{a,b}}{\phi_a} \E_a(N_a(t)) = \frac{c\phi_{a,b}}{\phi_a}(e^{\phi_a t} -1). \label{eqn:430}
\end{align} 
For $\alpha > \max_a \phi_a$, we define and get (using \eqref{eqn:430}) the explicit formula:
\begin{equation}
\label{eqn:428}
	\hat{\rho}^{c}_{a,b}(\alpha):= \int_0^\infty e^{-\alpha t} \mu^c(a,b,dt)= \frac{c\phi_{a,b}}{\alpha -\phi_a} = (\tilde{\vM}^{\sss(c)}(\alpha))_{a,b},
\end{equation}
with $\tilde{\vM}^{\sss(c)}(\cdot)$ as in Remark \ref{rem:perron-l-1}. In particular (as described in the Remark), the vector $\vh$ is the Perron-Frobenius right eigen-vector of $\hat{\rho}^{c}_{a,b}(\lambda_c)$ with Perron root one. In particular for any fixed $a\in [K]$,
\begin{equation}
\label{eqn:mu-*-eigen}
	\sum_{b\in [K]} \mu^{*,c}(a,b; \bR_+) = \sum_{b\in [K]} \frac{\hat{\rho}^{c}_{a,b}(\lambda_c) h_b}{h_a} =1. 
\end{equation}
Thus, following \cite[Equation 1.1c]{nummelin1978uniform}, the family of measures $\set{\mu^{*,c}(a,b,\cdot): a,b\in [K]}$ is a semi-Markov kernel. Let $\bX= \set{(X_n, T_n): n\geq 0}$ denote the corresponding Markov renewal process namely the Markov chain on $[K]\times \bR_+$ with transition probabilities given by 
\begin{equation}
\label{eqn:503}
	\pr((a,t), A\times \Gamma ) = \sum_{b\in A}\mu^{*,c}(a,b; \Gamma -t), \qquad \Gamma \in \cB(\bR_+), A\subseteq [K]. 
\end{equation}
{\cg Intuitively, the above says that although the percolated branching process grows exponentially fast in continuous time, the expected population size rescaled as in \eqref{eqn:819} satisfies a renewal-type equation \eqref{eqn:r*}. This equation is described in terms of the transition kernel of a Markov process $(X_n,T_n)$ that is homogeneous in the sense that the conditional distribution of $(X_{n+1}, T_{n+1}-T_n)$ given the entire past of the process only depends on $X_n$.}

We write $\pr_{a,t}^{\bX}$ for the above transition kernel and $\E_{a,t}^{\bX}$ for the corresponding expectation operator.  The corresponding renewal measure is given by,
\begin{equation}
	\label{eqn:frstar}
	\fR^*(a,b,dt):= \E_{a,0}^{\bX}[\sum_{n=0}^{\infty} \ind\set{X_n=b, T_n\in dt}], 
	\end{equation}
\textcolor{black}{where we have suppressed its dependence on $c$ for notational convenience.
Letting $G(a,t) = e^{-\lambda_c t}/h_a$ and using \eqref{eqn:r*}, we see that 
\begin{equation}
\label{eqn:818}
	m(a,t) = \sum_{b \in [K]}\int_0^tG(b,t-s)\fR^*(a,b,ds)=: \fR^*\otimes G(a,t),
\end{equation}
where $\otimes$ denotes the convolution operation.} 

Now note that for $\bX$:
\begin{enumeratea}
	\item The embedded $[K]$ valued Markov chain $(X_n:n\geq 0)$ is clearly irreducible and recurrent. 
	\item For any $a\in [K]$,
	\begin{align*}
		\pr_{a,0}^{\bX}(T_1 >t) &= \sum_{b\in [K]} \int_{t}^{\infty} e^{-\lambda_c s} \mu^c(a,b,ds) \frac{h_b}{h_a} = \sum_{b\in [K]} \int_{t}^{\infty} e^{-\lambda_c s} c\phi_{ab}e^{\phi_a s} \frac{h_b}{h_a} ds\\
		&= \sum_{b\in [K]} \left[\frac{c \phi_{ab}}{\lambda_c - \phi_a}\cdot \frac{h_b}{h_a}\right] e^{-(\lambda_c - \phi_a)t} = e^{-(\lambda_c -\phi_a)t},
	\end{align*}
	where again in the last line, we have used the eigen-vector characterization of $\mvh$. In particular for any $a\in [K]$, under $\pr_{a,0}^{\bX}$,
	\[T_1 \sim \Exp(\lambda_c - \phi_a), \qquad  \Rightarrow \qquad \E_{a,0}^{\bX}(T_1) < \infty.\]
	Thus the Markov renewal process is positive recurrent in the sense of \cite[Section 2]{nummelin1978uniform}. {\cg This implies that, starting from any state $a$, the expected time to hit $b$ is finite (that is, if $\sigma := \inf\{n \ge 0: X_n=b\}$, then $\E(T_{\sigma})<\infty$).}
	\item The irreducibility condition stated in \cite[Corollary on P129]{nummelin1978uniform} also clearly holds.
\end{enumeratea}
Thus all the conditions for \cite[Theorem 5.1]{nummelin1978uniform} are satisfied. Write $\mvrho = (\rho_1, \rho_2, \ldots, \rho_K)$ for the (strictly positive) left eigen-vector for the matrix $\tilde{\vM}^{(c)}(\lambda_c)$ corresponding to eigen-value one.  By \cite[Corollary to Thm 5.1, part(ii), P135]{nummelin1978uniform}, using the characterization in \eqref{eqn:818},
\begin{equation}
\label{eqn:816}
	m(a,t) = \frac{e^{-\lambda_c t} \E_a(\scZ_a^c(t))}{h_a} \longrightarrow  \frac{\frac{1}{\lambda_c} \sum_{b\in [K]} \frac{\rho_b}{h_b}}{\sum_{b\in [K]} \frac{\rho_b}{\lambda_c -\phi_b}} >0, \qquad \mbox{ as } t\to\infty.
\end{equation}
This completes the proof. 

\end{proof}

\begin{proof}[{\bf Completing the proof of Theorem \ref{thm:page-rank-tails}(a):}]\label{pf:thm-page-rank-tails}
We start with the upper bound. Fix any integer $\ell > 2/\lambda_c$. By Proposition \ref{prop:r-zc-moments}(a) and Markov's inequality, {\cg writing $H_{\ell} :=  \sup_{t\in \bR_+} \E_a\bigg[\big(e^{-\lambda_c t } \cR_{\emptyset, c}(t)\big)^\ell\bigg] <\infty$, we obtain
\[\pr_a(\cR_{\emptyset, c}(t) \ge r) \leq \frac{\E_a\bigg[\big(e^{-\lambda_c t } \cR_{\emptyset, c}(t)\big)^\ell\bigg]}{(re^{-\lambda_c t})^{\ell}} \leq \frac{H_{\ell}}{(re^{-\lambda_c t})^{\ell}}, \qquad \forall~a\in [K]~ \forall r>0. \]}
Using the distributional characterization of the limit PageRank from \eqref{lwlpr} we get, 
\begin{align*}
	\pr_a(\cR_{\emptyset, c} \geq r) &= \int_0^\infty \pr_a(\cR_{\emptyset, c}(t) \ge r) 2 e^{-2t} dt \leq \int_0^{\frac{\log{r}}{\lambda_c}} \frac{H_{\ell}}{(re^{-\lambda_c t})^{\ell}} 2e^{-2t} dt + \int_{\frac{\log{r}}{\lambda_c}}^{\infty} 2e^{-2t} dt,\\
	&\leq\frac{2H_{\ell}}{r^\ell}\int_0^{\frac{\log{r}}{\lambda_c}} e^{(\ell \lambda_c -2)t} dt + \frac{1}{r^{2/\lambda_c}} \le \left(\frac{2H_{\ell}}{\ell \lambda_c -2} +1\right) \frac{1}{r^{2/\lambda_c}}. 
\end{align*}
To conclude, we prove the lower bound. By Proposition \ref{prop:r-zc-moments}(b), $\exists~ t_0, \delta, \eta >0$ such that for all $a\in [K]$, $t\geq t_0$, $\pr_a(e^{-\lambda_c t} \cR_{\emptyset, c}(t) \geq \delta) \geq \eta$. Hence for all $r\geq \delta e^{\lambda_c t_0}$ and any $a \in [K]$,
\[\pr_a(\cR_{\emptyset, c} \geq r) \geq \int_{\frac{\log{(r/\delta)}}{\lambda_c}}^\infty \pr_a(\cR_{\emptyset, c}(t) > r) 2 e^{-2t} dt \geq \eta \operatorname{exp}\left(-\frac{2\log{(r/\delta)}}{\lambda_c}\right) = \eta\left(\frac{r}{\delta}\right)^{-2/\lambda_c}, \]
completing the proof. 
\end{proof}

\subsubsection{\bf Proof of Theorem \ref{thm:page-rank-tails}(b)} {\cg Recall the renewal measure $\fR^*$ in \eqref{eqn:frstar}.} Using \eqref{lwlpr}, the definition of the mean process $m(\cdot, \cdot)$ from \eqref{eqn:819} and its characterization using the renewal measure in \eqref{eqn:818}, and defining $\theta:=2-\lambda_c >0$, we get, 
\begin{align}
	(1-c)^{-1}&\E_a(\cR_{\emptyset, c})=\int_0^\infty 2e^{-2t} \E_a(\scZ_a^c(t)) dt = h_a \int_0^\infty  2e^{-\theta t} m(a,t) dt \notag\\
	&= h_a\int_0^\infty 2e^{-\theta t}\int_0^t \sum_{b\in [K]} \frac{1}{h_b} e^{-\lambda_c(t-s)} \fR^*(a,b; ds) dt\notag\\
	&= h_a\int_0^\infty  \sum_{b\in [K]} \frac{1}{h_b} \int_s^\infty 2e^{-2t}  dt~ e^{\lambda_c s}\fR^*(a,b; ds)\notag\\
	&= h_a\int_0^\infty \sum_{b\in [K]}\frac{1}{h_b} e^{-\theta s} \fR^*(a,b; ds)
	= h_a\int_0^\infty\sum_{b\in [K]}\frac{1}{h_b} \fR^*(a,b; [0,s]) \theta e^{-\theta s} ds \notag\\
	&= h_a \E\left[\sum_{b\in [K]} \frac{1}{h_b} \fR^*(a,b;[0,S^\theta])\right] = h_a \E\left[\sum_{b\in [K]}\frac{1}{h_b} \sum_{n\geq 0}\pr_{a,0}^{\bX}(X_n=b, T_n\leq S^\theta)\right]\label{eqn:947}
\end{align}
where $S^\theta\sim \Exp(\theta)$, independent of the Markov renewal process $\bX$, and the last two expectations above are taken with respect to the distribution of $S^\theta$. For $n\geq 1$, recall the notation $\scI_n^{\sss(a)}$ in the statement of Theorem \ref{thm:page-rank-tails}(b) and analogously for $a,b\in [K]$ define,
\[\scI_n^{\sss(a,b)}:= \set{\vj:=(j_0,j_1,\ldots, j_n)\in [K]^n: j_0 = a, j_n =b}.\] 
Recalling the definition of $\hat{\rho}$ from \eqref{eqn:428},  for $l,k\in [K]$ define, 
\[p_{lk}:= \frac{c\phi_{lk}}{\lambda_c - \phi_l}\cdot \frac{h_k}{h_l} = \hat{\rho}^{c}_{l,k}(\lambda_c)\cdot  \frac{h_k}{h_l}.\]
Note that \eqref{eqn:mu-*-eigen} implies the following. 
\begin{lemma}\label{lem:pmf}
	For each fixed $l\in [K]$, $\vp_l:= (p_{lk}:k\in [K])$ is a probability mass function with strictly positive entries. Further the matrix $\vP = (p_{lk})_{l,k\in [K]}$ is the transition matrix of an irreducible Markov chain on $[K]$. 
\end{lemma}
Next consider the definition of $\mu^{*,c}(l,k;\cdot)$ in \eqref{eqn:819} in terms of $\mu^c(l,k;\cdot)$ and the corresponding explicit expression for $\mu^c(l,k;\cdot)$ in \eqref{eqn:430}. Observe that, 
\begin{equation}
\label{eqn:937}
	\mu^{*,c}(l,k;dt) = e^{-\lambda_c t} e^{\phi_l t} c\phi_{lk} dt \frac{h_k}{h_l} = \left[(\lambda_c -\phi_l)e^{-(\lambda_c -\phi_l)t}dt\right] p_{lk}. 
\end{equation}
Using this in the definition of the Markov kernel \eqref{eqn:503} for the process $\bX$ gives the following description of $\bX$.
\begin{prop}
	\label{prop:vx-prop}
	The transition kernel of the Markov process $\bX$ has the following equivalent definition: conditional on $X_0 = l$ and $T_0=0$, 
	\begin{enumeratea}
		\item Sample $X_1$ using pmf $\vp_l$. 
		\item Sample $T_1\sim \Exp(\lambda_c - \phi_l)$. 
	\end{enumeratea} 
	In particular, $(X_1, T_1)$ are independent conditional on $X_0$. 
\end{prop}
Using this characterization in \eqref{eqn:947} gives 
\[\E(\fR^*(a,b;[0,S^\theta])) = \delta_{a,b} +\sum_{n\geq 1}\sum_{\vj \in \cI_n^{\sss(a,b)}} p_{j_0j_1} \cdots p_{j_{n-1}j_n} \pr(\tau_{j_0}+\tau_{j_1}+\cdots \tau_{j_{n-1}}\leq S^{\theta}),\]
where in the summand above, for each $\vj$, $\tau_{j_0}, \ldots, \tau_{j_{n-1}}$ are independent with $\tau_{j_{l}}\sim \Exp(\lambda_c - \phi_{j_l})$. Using independence of $S^\theta$ from the $\tau_{\cdot}$ sequence,  note that, for any term in the above summand, 
\[\pr(\tau_{j_0}+\tau_{j_1}+\cdots \tau_{j_{n-1}}\leq S^{\theta}) = \E(-\operatorname{exp}(\theta(\tau_{j_0}+\tau_{j_1}+\cdots \tau_{j_{n-1}}))).\]
Now using independence of the $\tau_{\cdot}$ sequence, the explicit form of $(p_{lk})$ and the standard formula for the Laplace transform of the exponential distribution,
\begin{align*}
\E(\fR^*(a,b;[0,S^\theta])) &= \delta_{a,b} +\sum_{n\geq 1}\sum_{\vj \in \cI_n^{\sss(a,b)}} \frac{h_{b}}{h_{a}}\prod_{l=0}^{n-1}\frac{c\phi_{j_lj_{l+1}}}{\lambda_c - \phi_{j_l}}\prod_{l=0}^{n-1}\frac{\lambda_c - \phi_{j_l}}{2 - \phi_{j_l}}\\
&=\delta_{a,b} +\sum_{n\geq 1}c^n\sum_{\vj \in \cI_n^{\sss(a,b)}} \frac{h_{b}}{h_{a}}\prod_{l=0}^{n-1}\frac{\phi_{j_lj_{l+1}}}{2 - \phi_{j_l}}.
\end{align*} 
Therefore,
\begin{equation*}
\E_a(\cR_{\emptyset, c}) = (1-c)h_a \E\left[\sum_{b\in [K]} \frac{1}{h_b} \fR^*(a,b;[0,S^\theta])\right] = (1-c)\left(1 + \sum_{n\geq 1}c^n\sum_{\vj \in \cI_n^{\sss(a)}} \prod_{l=0}^{n-1}\frac{\phi_{j_lj_{l+1}}}{2 - \phi_{j_l}}\right)
\end{equation*}
completing the proof of Theorem \ref{thm:page-rank-tails}(b).

\subsubsection{\bf Proof of Theorem \ref{thm:page-rank-tails}(c)} To prove the first assertion that $\sum_{a\in [K]} \pi_a \E_a(\cR_{\emptyset, c}) =1$, note that by (b) just proven, 
\begin{equation}
\label{eqn:127}
	\sum_{a\in [K]} \pi_a \E_a(\cR_{\emptyset, c}) = (1-c) + (1-c)\left[\sum_{n=1}^\infty c^n  \sum_{a\in [K]}\pi_a \sum_{\vj\in \scI_n^{\sss(a)}} \prod_{l=0}^{n-1} \left(\frac{\phi_{j_j, j_{l+1}}}{2-\phi_{j_l}}\right) \right].
\end{equation} 
Using Proposition \ref{prop:left-eigen} it is easy to check that for each $n$, the summand above,
\[\sum_{a\in [K]}\pi_a \sum_{\vj\in \scI_n^{\sss(a)}} \prod_{l=0}^{n-1} \left(\frac{\phi_{j_j, j_{l+1}}}{2-\phi_{j_l}}\right)  = \sum_{b\in [K]} \pi_b = 1.\]
Using this in \eqref{eqn:127} completes the proof of the first assertion. 

To prove the second claim, consider for any fixed $n$, the sum of the normalized PageRanks,  $\Upsilon_n := \sum_{v\in \cG_n} R_{v,c}(n)$. Note that a vertex $v$ which is a descendant of another vertex $u$ at graph distance $l$ from $u$ contributes $(1-c)c^l$ to the PageRank of $u$. Thus the contribution of $v$ to $\Upsilon_n$ is bounded by $(1-c)\sum_{l=0}^\infty c^l =1$. In particular, 
\begin{equation}
\label{eqn:142}
	\frac{\Upsilon_n}{n} \leq \frac{|\cG_n|}{n} \to 1, \qquad \mbox{as } n\to\infty. 
\end{equation}
Next, by Theorem \ref{thm:page-rank}(b), writing $F_n(\cdot)$ for the empirical cdf of the (normalized) PageRank scores in $\cG_n$ and $F_\infty$ for the limit cdf of $\cR_{\emptyset, c}$ where the root type is selected according to distribution $\mvpi$, $F_n \convd F$ almost surely. Hence, by Fatou's Lemma, the following hold almost surely:
\begin{equation}
\label{eqn:149}
	1 = \E_{\mvpi}(\cR_{\emptyset, c}) = \int_0^\infty (1-F_\infty(x)) dx \leq \liminf{n \rightarrow \infty} \int_0^\infty (1-F_n(x)) dx = \liminf_{n\to\infty} \frac{\Upsilon_n}{n}.
\end{equation} 
Combining \eqref{eqn:142} and \eqref{eqn:149} shows that $\Upsilon_n/n\convas 1$ proving the second assertion in (c). 

\textcolor{black}{To prove the final assertion, fix $a \in [K]$ and let $V_n$ denote a uniformly chosen vertex in $\cG_n$. Define the random variable $Y_n^a := R_{V_n,c}(n)\ind\set{a(V_n) =a}$. By Theorem \ref{thm:page-rank}(b), almost surely, the conditional law of $Y_n^a$ given $\cG_n$ converges in distribution to the law of $Y^a := \cR_{\emptyset, c}\ind\set{a(\emptyset) =a}$ as $n \rightarrow \infty$. Hence, by Fatou's Lemma, almost surely,
\begin{equation}\label{fa1}
\liminf_{n \rightarrow \infty} \E\left[Y_n^a \vert \cG_n\right] \ge \E[Y^a] = \pi_a \mathbb{E}_a[\cR_{\emptyset, c}].
\end{equation}
Moreover, from the previous two assertions in part (c) of the Theorem, almost surely,
\begin{align}\label{fa2}
\sum_{a \in [K]}\liminf_{n \rightarrow \infty}\E\left[Y_n^a \vert \cG_n\right] &\le \liminf_{n \rightarrow \infty} \sum_{a \in [K]}\E\left[Y_n^a \vert \cG_n\right] = \liminf_{n \rightarrow \infty}\E\left[R_{V_n,c}(n) \vert \cG_n\right]\notag\\
&= \liminf_{n \rightarrow \infty}\frac{\Upsilon_n}{n} = 1 = \sum_{a \in [K]} \pi_a \mathbb{E}_a[\cR_{\emptyset, c}].
\end{align}
The last assertion follows from \eqref{fa1} and \eqref{fa2}.}

%
%

\section{Proofs: Maximal degree asymptotics}
\label{sec:proof-max}
{\cg
{\cg This Section is devoted to the proof of Theorem \ref{thm:max-degree}. }
The proof follows similar lines to the maximal degree analysis of a single type preferential attachment in \cite{mori2005maximum}. The main idea of the proof is to construct appropriate Martingales to track the evolution of degrees of vertices; in order to understand maximal degree behavior we will need martinagles that track higher moments of the degree evolution. To ease notation, assume that the network starts with a single vertex of type $a$ {with initial degree set to one}. In this Section, the root will be denoted by vertex $1$. Else, one can repeat the same argument below, starting from the first (thus random) time of birth of a vertex of type $a$. Thus to prove (a) of the Theorem, here ${\cg \fF_n^a}$ denotes the degree of the root in $\cG_n$. Recall the functionals $(\tilde{Y}_b^{\sss(n)}:b\in [K])$ of normalized degree sums of different attributes in \eqref{eqn:540}. In the following, we simply write $j$ for the $j$-th arriving vertex.
\begin{lemma}
\label{lem:mart-degree-evolution}
For each $a\in [K]$, integer time indices $\ell\geq 1$ and $1\leq j < n < \infty$ define, 
\[\fN_\ell^{\sss(a)} = \sum_{b\in [K]} \frac{\kappa(a,b)\pi_b}{\sum_{a^\prime \in [K]} \kappa(a^\prime ,b) \tilde{Y}_{a^\prime}^{\sss(\ell)}}, \qquad C_{j,n}^{\sss(a)} = \prod_{\ell=j}^{n-1}\left(1+\frac{\fN_\ell^{\sss(a)}}{2(\ell+1)}\right). \]
Suppose vertex $j$ is of type $a$. Then $\set{\deg(j,\cG_n)/C_{j,n}^{\sss(a)}:n\geq j}$ is an $\bL^2$ bounded positive martingale.  
\end{lemma}
\begin{proof}
The martingale assertion follows directly from the dynamics of the process $\set{\cG_j:j\geq 0}$. Writing $\Delta[\ell+1, j] = \deg(j,\cG_{\ell+1}) - \deg(j,\cG_{\ell})$, it is easy to check that 
\begin{equation}
\label{eqn:912}
\E(\Delta[\ell+1, j]|\cG_n) = \frac{\fN_\ell^{\sss(a)}}{2(\ell+1)} \deg(j,\cG_{n}), \end{equation}
resulting in the stated assertion regarding the martingale property. 
To prove $\bL^2$ boundedness note that, 
 \[{\deg(j,\cG_{n+1}) +1 \choose 2} = {\deg(j,\cG_{n}) +1 \choose 2}\left(1+\frac{2\Delta[n+1, j]}{\deg(j,\cG_{n})}\right).\]
 Thus using \eqref{eqn:912} again shows that 
 \[\set{{\deg(j,\cG_{n}) +1 \choose 2}\cdot \prod_{\ell=j}^{n-1}\frac{1}{\left(1+\frac{\fN_\ell^{\sss(a)}}{(\ell+1)}\right)}: n\geq j},\]
 is a martingale. The $\bL^2$ boundedness now follows upon noting that $\prod_{\ell=j}^{n-1}\left(1+\frac{\fN_\ell^{\sss(a)}}{(\ell+1)}\right)/\prod_{\ell=j}^{n-1}\left(1+\frac{\fN_\ell^{\sss(a)}}{2(\ell+1)}\right)^2 \le 1$ for all $1\leq j < n < \infty$.
\end{proof}

\noindent{\bf Completing the proof of Theorem \ref{thm:max-degree}: } We start with (a). Recall that we use $1$ for the root of the tree.  By Lemma \ref{lem:mart-degree-evolution} and Doob's submartingale convergence theorem for $\bL^2$ bounded martingales, with $C_{n}^{\sss(a)} = C_{1,n}^{\sss(a)}$ we obtain a non-negative non-degenerate random variable $W_a$ such that, 
\[\frac{\deg(1,\cG_n)}{C_{n}^{\sss(a)}} \convas W_a,\]
proving the first assertion in (a). The second assertion on the scaling of $C_n$ follows from noting that, 
\[\sum_{\ell=1}^{n-1} \frac{\fN_\ell^{\sss(a)}}{2(\ell+1)} - \sum_{\ell=1}^{n-1} \frac{(\fN_\ell^{\sss(a)})^2}{8(\ell+1)^2} \le \log{C_n}  \le  \sum_{\ell=1}^{n-1} \frac{\fN_\ell^{\sss(a)}}{2(\ell+1)},\]   
and using the fact that by the work of \cite{jordan2013geometric} and quoted as \eqref{eqn:8522},
 \begin{equation}
 \label{eqn:upst-to-phi}
 \fN_\ell^{\sss(a)} \convas \sum_{b\in [K]} \pi_b \frac{\kappa(a,b)}{\sum_{a^\prime} \kappa(a^\prime, b)\eta_{a^\prime}} = \sum_{b\in [K]} \kappa(a,b) \nu_b = \phi_a, \qquad \mbox{ as } \ell\to\infty. 
 \end{equation}
We now prove (b). Since we are proving asymptotics for the maximum degree, we essentially need control over higher moments of the evolution of degrees thus necessitating a strengthening of Lemma \ref{lem:mart-degree-evolution}. The proof follows identical lines with suitable modifications. 
\begin{lemma}
\label{lem:mart-k}
For any $k\geq 1$, the process
\[\set{{\deg(j,\cG_{n}) +k-1 \choose k}\cdot \prod_{\ell=j}^{n-1}\frac{1}{\left(1+\frac{k\fN_\ell^{\sss(a)}}{2(\ell+1)}\right)}: n\geq j},\]
is a martingale. 
\end{lemma}
Now fix $\eps >0$ as in the statement of (b).  Our goal is to show that for any $\delta >0$,
\begin{equation}
\label{eqn:to-show}
\limsup_{n\to \infty}\pr({\cg \fN_n^a}/n^{\frac{\phi_a +  \eps}{2} } > \delta) =0. 
\end{equation}
For the rest of the argument we write ${\cg \fN_n = \fN_n^{a}}$.  It is enough to show, for some fixed $k\geq 1$ and any $\delta >0$, 
\begin{equation}
\label{eqn:enough-to-show}
\cE_{n}(k,\eps,\delta):= \pr\left(\frac{1}{n^{k(\frac{\phi_a +  \eps}{2})}} {{\cg \fN_n}+ k-1 \choose k}  > \delta\right) \xrightarrow{n \rightarrow \infty} 0. 
\end{equation}
Fix $\eps_1< \eps$ and $k\geq 1$ with $k\frac{\phi_a + \eps_1}{2} > 1$. Fix $B \geq 1$ and consider the stopping time 
\[T_B = \inf\set{\ell\geq B: \fN_\ell^{\sss(a)} \geq \phi_a + \eps_1  }.\]
Now note that for any $n$, 
\[{{\cg \fN_n}+ k-1 \choose k} \leq \sum_{v=1}^n {\deg(v, \cG_n)+ k-1 \choose k} \ind\set{a(v) =a}.  \]
Thus using Lemma \ref{lem:mart-k} and the optional sampling theorem, there exist constants $C_1(B) < C_2(B)<  \infty $ such that for $n>B$, 
\begin{align*}
\cE_{n}(k,\eps,\delta) &\leq \pr(T_B \leq n) + \frac{C_1(B)}{\delta n^{k(\frac{\phi_a+\eps}{2})}}\sum_{j=1}^n {\prod_{\ell=j}^n\left(1+\frac{k}{2(\ell+1)} (\phi_a + \eps_1)\right)}\\
&\leq \pr(T_B \leq n) + \frac{C_2(B)}{\delta n^{k(\frac{\phi_a+\eps}{2})}} n^{k(\frac{\phi_a+\eps_1)}{2}} \sum_{j=1}^n \frac{1}{j^{k(\frac{\phi_a+\eps_1)}{2}}} \\ 
&\leq \pr(T_B \leq n) + O\left(\frac{1}{n^{k\frac{(\eps - \eps_1)}{2}}}\right).
\end{align*}
Thus for any $B\geq 1$,
\[\limsup_{n\to \infty}\cE_{n}(k,\eps,\delta) \leq \pr(T_B < \infty).  \]
However again using \eqref{eqn:upst-to-phi}, we get that $\limsup_{B\to\infty}\pr(T_B < \infty) = 0$. This completes the proof of \eqref{eqn:enough-to-show} and thus (b). 
}

\section{Proofs: non-tree regime}
	\label{sec:proofs-extensions}
	We will now prove the local limit theorems for the non-tree network model and its implications discussed in Section \ref{sec:nt}. When clear from context, we will suppress dependence of some quantities on $\mvm$ for notational convenience.
\subsection{Proof of Lemma \ref{lem:minimizernt}:}
From the dynamics of the graph process $\{\cG_n\}_{n \ge 0}$, it can be checked that
\begin{equation}\label{yevol}
\E\left[\tilde Y_a^{\mvm,\sss{(n+1)}} \vert \cG_n\right] = \tilde Y_a^{\mvm,\sss{(n)}} + \frac{1}{n+n_0 + 1}\left[\frac{1}{2}m_a\pi_a + \frac{1}{2}\sum_{b \in [K]} m_b \pi_b \frac{\tilde Y_a^{\mvm,\sss{(n)}}\kappa_{a,b}}{\sum_{l \in [K]}\tilde Y_l^{\mvm,\sss{(n)}}\kappa_{l,b}} - \tilde Y_a^{\mvm,\sss{(n)}}\right].
\end{equation}
In the language of \cite{pemantle2007survey} (see Theorem 2.13 there), the above implies that $\{\tilde Y_a^{\mvm,\sss{(n)}}\}_{n \ge 0}$ evolves according to a stochastic approximation algorithm whose asymptotic pseudotrajectories are governed by flows of the ODE
$$
\dot \vY(t) = -G(\vY(t))
$$
in $\mathbb{R}_+^K$, where $G:\mathbb{R}_+^K \rightarrow \mathbb{R}_+^K$ is given by $G_i(\vy) := y_i \partial_i V_{\mvpi}(\vy)$, $i \in [K]$, and $V_{\mvpi}(\cdot)$ is defined in \eqref{eqn:vpi-defnt}. The function $V^{\mvm}_{\mvpi}(\cdot)$ is strictly convex and diverges to infinity on approaching the boundary of $\mathbb{R}_+^K$. Moreover, the set $S^{\mvm}$ defined in the Lemma is an invariant compact set for the flow $\vY(\cdot)$ (that is, if $\vY(0) \in S^{\mvm}$, then $\vY(t) \in S^{\mvm}$ for all $t \ge 0$). Further, $V^{\mvm}_{\mvpi}(\cdot)$ is a \emph{Lyapunov function} in the sense that $t \rightarrow V_{\mvpi}(\vY(t))$ is non-increasing, and strictly decreasing away from the critical points of $V^{\mvm}_{\mvpi}(\cdot)$ (where $\nabla V^{\mvm}_{\mvpi}=0$). From these observations, we conclude that $V^{\mvm}_{\mvpi}(\cdot)$ has a unique critical point corresponding to a global minimum in $\mathbb{R}_+^K$ that lies in the interior of the set $S^{\mvm}$. The Lemma now follows on appealing to \cite[Proposition 2.18]{pemantle2007survey}.

\subsection{Proof of Theorem \ref{thm:fringe-convergencent}}
The proof is along the same lines as that of Theorem \ref{thm:just-fringe-convg} using stochastic approximation techniques. 

For a vertex $v$ in a graph $G \in \mathbb{G}$, define the \emph{in-component} of $v$ to be the incoming neighborhood of $v$, written as $U_{\infty}(v, G)$, obtained in a similar way as $U_{\infty}(\emptyset)$ in Section \ref{lwcnt}.

For $n \ge 0$ and a given tree $\vt \in \bT_{\cS}$ (viewed as a directed graph with edges pointing from children to parents), define
$$
c_n(\vt) := \sum_{v \in \cG^{\mvm}_n} \ind\set{U_{\infty}(v, \cG^{\mvm}_n) \cong \vt}, \qquad \fP_n(\vt) := \frac{c_n(\vt)}{n+n_0}.
$$
Also recall that for $u \in \vt$, $I_u$ denotes the indicator that the youngest child $y_u$ of $u$ is a leaf in $\vt$ and $\vt^{(u)}$ denotes the tree obtained by deleting this child. We will denote the attribute type and degree of $u \in \vt$ (in-degree + out-degree) respectively by $\rho_u$ and $d_u$ (suppressing dependence on $\vt$ as all calculations below are for fixed $\vt$).

Since multiple edges are added for each new added vertex, controlling the evolution of the number of copies of a given subgraph in the growing network sequence is potentially more involved. The next lemma helps with this by showing that for any given $\vt \in \bT_{\cS}$, with high probability for large $n$, any in-component in $\cG^{\mvm}_n$ that is isomorphic to $\vt$ receives at most one incoming edge from $v_{n+1}$. 

\begin{lemma}\label{nottwo}
Fix $\vt \in \bT_{\cS}$ and write $\{T^n_i(\vt) : 1 \le i \le c_n(t)\}$ for the in-components in $\cG^{\mvm}_n$ isomorphic to $\vt$. Then,
$$
\pr\left(T^n_i(\vt) \text{ receives more than one incoming edge from } v_{n+1} \text{ for some } 1 \le i \le c_n(t) \, \vert \, \cG^{\mvm}_n\right) = O(n^{-1}).
$$
\end{lemma}

\begin{proof}
Write $|\vt|$ for the number of vertices in $\vt$ and let
$$
d_{\vt} := \operatorname{max}_{v \in \vt} (m_{\rho_v} + d_v), \quad \kappa_{\vt}^j := \operatorname{max}_{v \in \vt}  \kappa(\rho_v,j), \ j \in [K].
$$
Then,
\begin{align*}
&\pr\left(T^n_i(\vt) \text{ receives more than one incoming edge from } v_{n+1} \text{ for some } 1 \le i \le c_n(t) \, \vert \, \cG^{\mvm}_n\right)\\
&\le c_n(\vt)\sum_{j \in [K]} \pi_j {m_j \choose 2} \left(\frac{|\vt|\cdot d_{\vt}\cdot \kappa_{\vt}^j}{\sum_{v^\prime \in \cG_n^{\mvm} } \kappa(a(v'),a(v_{n+1})) \deg(v^\prime, n)}\right)^2.
\end{align*}
The lemma now follows upon observing that $c_n(\vt) \le n + n_0$ and each term in the sum is $O(n^{-2})$.
\end{proof}

The rest of the proof is very similar to the proof of Theorem \ref{thm:just-fringe-convg} and we provide only the outline. Fix $\vt \in \bT_{\cS}$. Recall $\cY_{\vt}$ and $\Phi_{n,u, y_{u}}$ (with $\tilde{Y}_a^{\sss{(n)}}$ replaced by $\tilde{Y}_{a}^{\mvm,\sss{(n)}}$) from that proof. Define
$$
\Psi_{n,u,a}:=\frac{\kappa(\rho_{u}, a) \pi_{a}}{\sum_{b\in [K]} \kappa(b,a)\tilde{Y}_b^{\mvm,\sss{(n)}}},\quad  u\in \vt, a\in [K].
$$

Then, by the dynamics of the process $\set{\cG_n:n\geq 0}$, and using Lemma \ref{nottwo},
\begin{equation}
\label{eqn:626nt}
	\pr(c_{n+1}(\vt) = c_n(\vt)+1|\cG^{\mvm}_n) = \sum_{u\in \cY_{\vt}} c_n(\vt^{\sss(u)})\left[ \Phi_{n,u, y_{u}} ~\cdot ~ \frac{m_{\rho_{y_u}}(d_u-1)}{2(n+n_0)}\right] + O(n^{-1}),
\end{equation}
and
\begin{equation}
\label{eqn:633nt}
	\pr(c_{n+1}(\vt) = c_n(\vt)-1|\cG^{\mvm}_n) = c_n(\vt) \sum_{u\in \vt} \sum_{a \in [K]} \frac{\Psi_{n,u,a}\cdot d_um_{a}}{2(n+n_0)} + O(n^{-1}). 
\end{equation}
Using \eqref{eqn:626nt} and \eqref{eqn:633nt}, we obtain:
\begin{align*}
	\E(\fP_{n+1}(\vt)|\cG_n) =\frac{1}{n+n_0+1}\bigg[&\left((n+n_0) - \sum_{a \in [K]}\sum_{u\in \vt} \Psi_{n,u,a}\cdot \frac{d_um_a}{2}\right)\fP_n(\vt)\\
	&+ \sum_{u\in \vt} \fP_n(\vt^{\sss(u)})\left( \Phi_{n,u, y_{u}} ~\cdot ~ \frac{m_{\rho_{y_u}}(d_u-1)}{2}\right)I_u \bigg] + O(n^{-2}).
\end{align*}
Rearranging gives, 
\begin{align}
	\E(\fP_{n+1}(\vt)|\cG_n) - \fP_n(\vt) = \frac{1}{n+n_0+1}&\bigg[\sum_{u\in \vt} \fP_n(\vt^{\sss(u)})\left( \Phi_{n,u, y_{u}} ~\cdot ~ \frac{m_{\rho_{y_u}}(d_u-1)}{2}\right)I_u \notag \\
	&-\left(1+\sum_{a \in [K]}\sum_{u\in \vt} \Psi_{n,u,a}\cdot \frac{d_um_a}{2}\right)\fP_n(\vt) \bigg] + O(n^{-2}).
\end{align}
Using the above stochastic approximation form with Lemma \ref{lem:jordan} (now with $R_n = O(n^{-2})$), we conclude that $\fP_n(\vt) \convas \fpm(\vt)$, where $\fpm(\cdot)$ satisfies the recursion
\begin{equation}
		\label{eqn:recur-first-partnt}
			\fpm(\vt) = \frac{\sum_{v\in \vt} I_v  \cdot  \fpm(\vt^{\sss(v)})\frac{d_v-1}{2} \phi^{\mvm}_{\rho_v , \rho_{y_v}}m_{\rho_{y_v}}}{1+\sum_{v\in \vt}\sum_{a \in [K]}m_a \phi^{\mvm}_{\rho_v,a} \frac{d_v}{2}},
		\end{equation}
		with boundary conditions for $\vt = \set{v_0, a(v_0) = b}$, i.e. a tree consisting of a single vertex of type $b\in [K]$, given by 
		\begin{equation}
		\label{eqn:recur-boundarynt}
			\fpm(\vt) = \frac{2\pi_b}{2+ m_b\phi^{\mvm}_{b}}, \qquad b\in [K]. 
		\end{equation}
This recursion has a unique solution satisfied by the law of $\cT^{\mvm}_{A}(\tau)$, which follows exactly as the proof of Proposition \ref{prop:prop-of-fp}. As $\fpm(\cdot)$ is the distribution of the progeny tree of a randomly stopped branching process, it is fully supported on $\bT_{\cS}$, thereby completing the proof of Theorem \ref{thm:fringe-convergencent}. 

\subsection{Proofs of Theorem \ref{thm:deg-nt} and Theorem \ref{thm:page-rank-tailsnt}}
We first prove Theorem \ref{thm:deg-nt}. Fix $a \in [K]$. Define $S^a_0=0$ and $S^a_k := \sum_{l=1}^{k} \frac{E_l}{\phi^{\mvm}_a(l + m_a)}, \, k \in \mathbb{N}$, where $\{E_l\}$ are \emph{i.i.d.} $\Exp(1)$ random variables. Recall $\tau \sim \Exp(2)$. By Theorem \ref{thm:fringe-convergencent} and using standard properties of exponential random variables, we obtain for any $k \ge m_a$,
\begin{align*}
\vp^{\mvm,a}_{\infty}(k) = \pr_a\left(S^a_{k-m_a} \le \tau < S^a_{k-m_a+1}\right) = \frac{2}{2 + \phi^{\mvm}_ak}\prod_{i=m_a}^{k-1}\frac{\phi^{\mvm}_a i}{2 + \phi^{\mvm}_a i},
\end{align*}
where the product is taken to be $1$ if $k=m_a$. Theorem \ref{thm:deg-nt} follows from the above.

To prove Theorem \ref{thm:page-rank-tailsnt}, we proceed as in the proof of Theorem \ref{thm:page-rank-tails}. The almost sure convergence follows from Theorem \ref{thm:fringe-convergencent} and results in \cite{garavaglia2020local,banerjee2021pagerank}.

To study the $\pr_a$-distributional properties of $\cR^{\mvm}_{\emptyset,c}$, observe that it is the root PageRank in $\cT^{\mvm}_a(\tau)$ (with root type $a \in [K]$). Define the percolated branching process $\BP_a^{c,\mvm}$ similarly as in Definition \ref{def:perc-bp} by randomly deleting edges in $\BP_a^{\mvm}$ except that edges emanating from type $b$ vertices in $\BP_a^{\mvm}$ are now retained with probability $c/m_b$. Let $\cZ_a^{\mvm,c}(t) = |\BP_a^{\mvm,c}(t)|, \, t \ge 0$. Then $\cR^{\mvm}_{\emptyset,c} = \cR^{\mvm}_{\emptyset,c}(\tau)$, where
$$
\cR^{\mvm}_{\emptyset,c}(t) = (1-c) \E_a(\cZ_a^{c,\mvm}(t) \vert \BP_a^{\mvm}(t)), \quad t \ge 0.
$$
To prove Theorem \ref{thm:page-rank-tailsnt}(a), note that, in light of the calculations after \eqref{eqn:816}, it suffices to prove the analogue of Proposition \ref{prop:r-zc-moments}. To prove part (a), let $\vh^{\mvm} = (h^{\mvm}_1,\dots, h^{\mvm}_K)$ denote the right eigenvector of $\vM^{\mvm,\sss{(c)}}$ defined in \eqref{eqn:matrix-mc-defnt} corresponding to Perron-Frobenius eigen-value $\lambda^{\mvm}_c$, normalized so that $\sum_{l=1}^K h^{\mvm}_l=1$. Define
\begin{equation*}
		U^{\mvm}(t):=\sum_{b\in [K]} \frac{h^{\mvm}_b}{m_b} Y^{\mvm}_b(t), \qquad Y^{\mvm}_b(t):= \sum_{v\in \BP_a^{\mvm,c}(t), a(v) = b} (\operatorname{in-deg}^{\mvm}(v,t) + m_b ), \qquad t\geq 0. 
	\end{equation*}
where $\operatorname{in-deg}^{\mvm}(v,t)$ is the in-degree (no of children) of $v$ in the full branching process $\BP^{\mvm}_a(t)$. Then, proceeding as in \eqref{eqn:933}-\eqref{eqn:1016}, it follows that $\{e^{-\lambda^{\mvm}_c t}U^{\mvm}(t) : t \ge 0\}$ is a positive martingale whose moments are uniformly bounded in time, implying part (a).

To prove part (b), it suffices to prove the analogue of Lemma \ref{lem:lower-exp-zc}. Denote by $\mu^{\mvm,c}(a,b; \cdot)$ the intensity measure of the point process encoding reproduction times of type $b$ children by a type $a$ parent in the branching process $\BP_a^{c,\mvm}$. Recall that the proof of Lemma \ref{lem:lower-exp-zc} relies on expressing the intensity measure in terms of a Markov renewal process $\bX = \{(X_n, T_n) : n \ge 0\}$: 
$$
\mu^{*,\mvm,c}(a,b; \cdot) := e^{-\lambda^{\mvm}_c t}\mu^{\mvm,c}(a,b; \cdot) \frac{h^{\mvm}_b}{h^{\mvm}_a} = \pr^{\bX,\mvm}_{a,0}(X_1 =b, T_1 \in dt)
$$
In this case, the $X$ and $T$ processes are independent of each other, the transition kernel of the former process is given by 
$$
p^{\mvm}_{ab} := \frac{c\phi^{\mvm}_{a,b} m_ah^{\mvm}_b}{(\lambda^{\mvm}_c - \phi^{\mvm}_a)m_b h^{\mvm}_a}, \quad a,b \in [K],
$$
and $T_i - T_{i-1}, \, i \ge 1,$ are \emph{i.i.d.} $\Exp(\lambda^{\mvm}_c - \phi^{\mvm}_a)$. Now, verifying the irreducibility and recurrence of the $X$ process and the positive recurrence of the Markov renewal process, and appealing to the results of \cite{nummelin1978uniform}, Lemma \ref{lem:lower-exp-zc} follows. This completes the proof of Theorem \ref{thm:page-rank-tailsnt}(a). Theorem \ref{thm:page-rank-tailsnt}(b) follows upon using the explicit form of the transition kernel $(p^{\mvm}_{lk})$ in place of $(p_{lk})$, and $\lambda^{\mvm}_c$ in place of $\lambda_c$, in the proof of Theorem \ref{thm:page-rank-tails}(b).

\section{Proofs: Network sampling}
\label{sec:net-sampling}
\subsection{Proof of Theorem \ref{thm:network-sampling}} By Remark \ref{rem:904}, we need to only prove parts (c)-(e). {\cg Recall that $\indNS$ and $\prNS$ represent sampling proportional to in-degree and proportional to PageRank respectively. }

\begin{enumeratea}
	\item[{\bf Proof of (c):}] Note that 
	\[\pr_{\indNS}(a({\cg U_n}) = b|\cG_n) = \frac{\sum_{{v\in \cG_n} } \deg(v, \cG_n)\ind\set{a(v) = b} - \#\set{v\in \cG_n, a(v) = b}}{n} + o_{\pr}(1), \]
	where the final term accounts for the adjustment at the root (whose out-degree is $0$). 
	Using the dynamics of the model class $\scrP$ and \eqref{eqn:8522} proved in \cite{jordan2013geometric} gives 
	\[\pr_{\indNS}(a({\cg U_n}) = b|\cG_n) \to 2\eta_b - \pi_b = \eta_b\left(2-\frac{\pi_b}{\eta_b}\right) = \eta_b \phi_b = \pi_b \frac{\phi_b}{2-\phi_b},\]
	where the final two equalities follow from \eqref{eqn:phi-def}. To prove the final assertion, note that 
	\[\pi_b \Psi_b \E^{\vS}_b\left[\frac{1}{\Psi_{S_1}}\right] = \pi_b \Psi_b \sum_{a\in [K]} \frac{\phi_{ba}}{2-\phi_b} \cdot \frac{\Psi_a}{\Psi_b}\cdot \frac{1}{\Psi_a} =\pi_b \frac{\phi_b}{2-\phi_b}. \]
	\item[{\bf Proof of (d):}] First note that by Theorem \ref{thm:page-rank-tails}(c),
	\[\pr_{\prNS}(a({\cg U_n}) = b|\cG_n) = \frac{\sum_{v\in \cG_n} R_{v,c}(n) \ind\set{a(v) =b}}{\sum_{v\in \cG_n} R_{v,c}(n)} \convas \pi_b \E_b(\cR_{\emptyset, c})\]
	Now using Proposition \ref{prop:mc-descp} for the equivalent description of the expected limiting PageRank completes the proof. 
	\item[{\bf Proof of (e):}] This follows in an identical fashion to the proof of (d). The proof is omitted.    
\end{enumeratea}

\subsection{Proof of Theorem \ref{thm:rare-sampling}}
{\cg The plan is to apply Theorem \ref{thm:network-sampling} in the specific context of rare minorities,  with the initial goal to derive explicit expressions for the various functionals in this Theorem in the setting where the parameters satisfy the scaling given in \eqref{eqn:rare-model} and \eqref{eqn:thet-scaling}.} Recall the function $V_{\mvpi}$ in \eqref{eqn:vpi-def}. In this case, we want to find $\mveta = (\eta, 1-\eta)$ that minimizes the univariate function
\[V_{\mvpi}(y) = 1-\frac{1}{2(1+\theta)}\left[\theta\log(y) +\theta\log(y+a(1-y)) + \log(1-y)\right], \qquad 0\leq y \leq 1.\]
Solving this minimization problem gives, 
\[\eta = \frac{2\theta - a-3\theta a + \sqrt{(2\theta-a-3\theta a)^2 + 4\theta a(1-a)(1+2\theta)}}{2(1-a)(1+2\theta)}.\]
In the regime $a\downarrow 0$ with $\theta$ satisfying the scaling in \eqref{eqn:thet-scaling}, it is easy to verify using Taylor approximations of various quantities appearing in the formula above that,
\begin{equation}
	\label{eqn:capl-222}
	\eta = 2D\sqrt{a} -(4D^2 +\frac{1}{2})a + O(a^{3/2}).
\end{equation}
Plugging this estimate into the expressions for the functionals $\phi_{\cdot, \cdot}$ and $\phi_{\cdot}$ in \eqref{eqn:phi-def} gives, 
\begin{align*}
	\phi_{11}&= \frac{\theta}{(1+\theta)(\eta+a(1-\eta))} = \frac{1}{2}+ \left(D-\frac{1}{4D}\right)\frac{\sqrt{a}}{2} + O(a). \\
	\phi_{12} = \phi_{22} &=\frac{1}{1+\theta} = 1-D\sqrt{a} + D^2 a + O(a^{3/2}). \\
	\phi_{21} &= \frac{a\theta}{(1+\theta)(\eta+a(1-\eta))} = \frac{a}{2}+ O(a^{3/2}). 
\end{align*}
Thus in this case, the $\vM$ matrix in \eqref{eqn:matrix-def} is given by, 
\[\vM = \begin{pmatrix}
	1-\frac{\sqrt{a}}{2D} + O(a)& 2(1-(2D+\frac{1}{4D})\sqrt{a}) + O(a) \\
	\frac{a}{2} + O(a^{3/2}) & 1-2D\sqrt{a} +(4D^2 +\frac{1}{2})a + O(a^{3/2}) 
\end{pmatrix}. \]
The eigen-vector $\mvPsi$ required in Proposition \ref{prop:mc-descp} is given by $\mvPsi = \frac{1+\theta}{1+\psi \theta}(\psi,1)$, where 
\[\psi = \frac{(2D - \frac{1}{2D})\sqrt{a} + O(a) + \sqrt{[(2D-\frac{1}{2D})\sqrt{a} + O(a)]^2 + 4a + O(a^{3/2})}}{a+O(a^{3/2})}. \]
In the $a\downarrow 0$ regime with fixed $D>0$, it can be checked that 
\[\psi = \left[(2D - \frac{1}{2D}) + \sqrt{((2D-\frac{1}{2D})^2 + 4)}\right]\cdot \frac{1}{\sqrt{a}} + O(1). \]
Now we can complete the proof by using these expressions in Theorem \ref{thm:network-sampling}. Parts(a) and (b) following from the scaling of $\pi_{1}$ and $\eta_1 = \eta$ above in the $a\downarrow 0$ regime. Part(c) follows from noting that 
\begin{align*}
	\eta_1 \phi_1 &= (2D\sqrt{a} - (4D^2 + \frac{1}{2})a + O(a^{3/2}))(\frac{3}{2} - (\frac{D}{2} +\frac{1}{8D})\sqrt{a} + D^2 a + O(a^{3/2})) \\
	&=3D\sqrt{a} + O(a). 
\end{align*}
Finally (d) follows from noting that 
\[\pi_1 \Psi_1 = \frac{\theta}{(1+\theta)} \frac{(1+\theta)}{1+\psi \theta}\psi = \frac{\psi \theta}{1+\psi \theta}\]
and thus using the estimates above, 
\[\psi \theta = \left[(2D - \frac{1}{2D}) + \sqrt{((2D-\frac{1}{2D})^2 + 4)}\right]\cdot \frac{1}{\sqrt{a}} \cdot D\sqrt{a} + O(\sqrt{a}). \]
Cancelling the $\sqrt{a}$ term in the numerator and denominator in the first term results in the bound in (d) and completes the proof.

\section*{Acknowledgements}
S. Banerjee is partially supported by the NSF CAREER award DMS-2141621. S. Bhamidi and V. Pipiras are partially supported by NSF DMS-2113662. S. Banerjee, S. Bhamidi and V.Pipiras are partially supported by NSF RTG grant DMS-2134107. S. Bhamidi was partially supported by DMS-2413928, and
DMS-2434559. Part of this material is based upon work supported by the National Science
Foundation under Grant No. DMS-1928930, while S. Banerjee and S. Bhamidi were in
residence at the Simons Laufer Mathematical Sciences Institute in
Berkeley, California, during the Spring 2025 semester. {\cg We thank two anonymous referees for detailed evaluation of the original submission which resulted in significant improvement of the paper. }

\bibliographystyle{imsart-number}
\bibliography{ref,CTBP,pdn,persistence}

\end{document}